\documentclass[11pt,oneside]{amsart}

\usepackage[T1]{fontenc}
\usepackage[utf8]{inputenc}
\usepackage{lmodern}
\usepackage{amsfonts}

\DeclareMathAlphabet\mathbfcal{OMS}{cmsy}{b}{n}
\usepackage{amsthm}
\usepackage{stmaryrd}
\usepackage{MnSymbol}
\usepackage{leftidx}
\usepackage{xcolor}
\usepackage[english]{babel}
\usepackage[french=guillemets]{csquotes}
\MakeOuterQuote{"}
\usepackage[all]{xy}
\usepackage[pagewise,mathlines]{lineno}
\usepackage{enumitem}
\setlist{nosep}
\usepackage[citestyle=alphabetic,bibstyle=mystyle,backend=biber,sorting=nyt,url=false,giveninits=true]
{biblatex}
\usepackage[colorlinks=true, breaklinks=true, urlcolor= black, linkcolor= black, citecolor= black, 
bookmarksopen=true,linktocpage=true,plainpages=false,pdfpagelabels]{hyperref}
\usepackage[capitalize]{cleveref}

\bibliography{short,biblio}

\author{Martin Hils}
\address{Institut f\"{u}r Mathematische Logik und Grundlagenforschung, Westf\"{a}lische Wilhelms-Universit\"{a}t M\"{u}nster, Einsteinstr. 62, D-48149 M\"{u}nster, Germany}
\email{hils@uni-muenster.de}
\thanks{MH was partially supported by the German Research Foundation (DFG) via CRC 878, HI 2004/1-1 (part of the French-German ANR-DFG project GeoMod) and under Germany's Excellence Strategy EXC 2044-390685587, `Mathematics M\"unster: Dynamics-Geometry-Structure'.}
\author{Silvain Rideau-Kikuchi}
\address{Université de Paris Cité and Sorbonne Université, CNRS, IMJ-PRG, F-75006 Paris, France.}
\email{silvain.rideau@imj-prg.fr}  
\thanks{SRK was partially supported by GeoMod AAPG2019 (ANR-DFG), Geometric and Combinatorial Configurations in Model Theory.}
\thanks{Both authors were partially supported by ValCoMo (ANR-13-BS01-0006).  A substantial part of this work was carried out when both authors participated in the thematic trimester ``Model theory, combinatorics and valued fields'' at the ``Institut Henri Poincaré'' in Paris. We are grateful to the IHP and its staff for hosting this event and providing us with this opportunity.}
\title{Un principe d'Ax-Kochen-Ershov imaginaire}
\date{\today}

\keywords{Model Theory, Valued Fields, Classification of Imaginaries, Non-standard Frobenius Automorphism, Separated pairs, Linear structures}
\subjclass[2020]{Primary: 03C45; Secondary: 03C10, 03C60, 12J10, 12L12}

\numberwithin{equation}{section}
\newtheorem*{thmA}{Theorem~A}
\newtheorem*{thmB}{Theorem~B}
\newtheorem*{theorem*}{Theorem}
\newtheorem{theorem}{Theorem}[subsection]
\newtheorem{corollary}[theorem]{Corollary}
\newtheorem{variant}[theorem]{Variant}
\newtheorem{lemma}[theorem]{Lemma}
\newtheorem{proposition}[theorem]{Proposition}

\newtheorem{question}[theorem]{Question}

\newtheorem{claim}[theorem]{Claim}
\newtheorem{fact}[theorem]{Fact}

\theoremstyle{definition}

\newtheorem{definition}[theorem]{Definition}
\newtheorem{notation}{Notation}[section]
\newtheorem{convention}[notation]{Convention}
\newtheorem{remark}[theorem]{Remark}

\newcommand{\Cc}{\mathbb{C}}
\newcommand{\Ff}{\mathbb{F}}

\newcommand{\Nn}{\mathbb{N}}

\newcommand{\Qq}{\mathbb{Q}}
\newcommand{\Rr}{\mathbb{R}}

\newcommand{\Zz}{\mathbb{Z}}

\newcommand{\cC}{\mathcal{C}}

\newcommand{\cE}{\mathcal{E}}

\newcommand{\cL}{\mathcal{L}}

\newcommand{\cO}{\mathcal{O}}

\newcommand{\cP}{\mathcal{P}}

\newcommand{\cU}{\mathcal{U}}

\newcommand{\fL}{\mathfrak{L}}

\newcommand{\fm}{\mathfrak{m}}

\newcommand{\bA}{\mathbf{A}}
\newcommand{\bB}{\mathbf{B}}
\newcommand{\bC}{\mathbf{C}}
\newcommand{\bD}{\mathbf{D}}
\newcommand{\bE}{\mathbf{E}}

\newcommand{\bI}{\mathbf{I}}

\newcommand{\bP}{\mathbf{P}}
\newcommand{\bQ}{\mathbf{Q}}
\newcommand{\bR}{\mathbf{R}}
\newcommand{\bS}{\mathbf{S}}
\newcommand{\bT}{\mathbf{T}}

\newcommand{\bV}{\mathbf{V}}

\newcommand\hA{\widehat{A}}
\newcommand\hC{\widehat{C}}

\newcommand\plim{\varprojlim}

\newcommand{\substr}{\leq}

\newcommand{\subsel}{\preccurlyeq}
\newcommand{\supsel}{\succcurlyeq}
\newcommand{\restr}[2]{{\left.#1\right|_{#2}}}
\newcommand{\tp}[1][]{\mathrm{tp}\ifstrempty{#1}{}{_{#1}}}
\newcommand\ACVF{\mathrm{ACVF}}
\newcommand{\K}{\mathbf{K}}
\newcommand{\dcl}[1][]{\mathrm{dcl}\ifstrempty{#1}{}{_{#1}}}
\newcommand{\acl}[1][]{\mathrm{acl}\ifstrempty{#1}{}{_{#1}}}

\newcommand{\germ}[2]{[#2]_{#1}}
\newcommand{\aut}[1][]{\mathrm{Aut}_{#1}}
\newcommand{\Hom}[1][]{\mathrm{Hom}_{#1}}
\newcommand{\alg}[1]{#1^{\mathrm{a}}}

\newcommand{\val}{\mathrm{v}}
\newcommand{\sminus}{\smallsetminus}
\newcommand{\Val}{\cO}
\newcommand{\Mid}{\fm}
\newcommand{\resf}{\mathrm{res}}
\newcommand{\Geom}{\mathbfcal{G}}

\newcommand{\eq}[1]{#1^{\mathrm{eq}}}

\newcommand{\res}{\mathbf{k}}
\newcommand{\vg}{\mathbf{\Gamma}}
\newcommand{\ld}{\mathbf{ld}}
\newcommand{\card}[1]{|#1|}

\newcommand{\dcleq}{\eq{\dcl}}
\newcommand{\acleq}{\eq{\acl}}
\newcommand{\Lat}{\bS}
\newcommand{\Tor}{\bT}
\newcommand{\isom}{\cong}

\newcommand{\ACFA}{\mathrm{ACFA}}
\newcommand{\ACF}{\mathrm{ACF}}
\newcommand{\RCF}{\mathrm{RCF}}
\newcommand{\RCVF}{\mathrm{RCVF}}
\newcommand{\dual}{\check}
\newcommand{\Th}[1][]{\mathrm{Th}_{#1}}
\newcommand{\code}[1]{\ulcorner #1\urcorner}

\newcommand{\RV}{\mathbf{RV}}
\newcommand{\rv}{\mathrm{rv}}
\newcommand{\LRV}{\cL_{\RV}}
\newcommand{\LRVsig}{\cL_{\RV}^{\sigma}}
\newcommand{\balls}{\bB}

\newcommand{\fballs}[1]{\balls^{[#1]}}
\newcommand{\fres}[1]{\res^{[#1]}}
\newcommand{\pballs}[2]{\balls^{[#1]}_{#2}}

\newcommand{\fpoints}[1]{\K^{[#1]}}
\newcommand{\ppoints}[2]{\K^{[#1]}_{#2}}
\newcommand{\points}[1]{#1^{\cup}}

\newcommand{\qftp}{\tp_0}
\newcommand{\LP}{\cL_{\mathrm{P}}}

\newcommand{\rad}{\mathrm{rad}}

\newcommand{\Ldiv}{\cL_{\mathrm{div}}}
\newcommand{\TP}{\mathrm{S}}
\newcommand{\tree}{\triangleleft}
\newcommand{\treeq}{\trianglelefteqslant}
\newcommand{\treequiv}{\equiv}
\newcommand{\ntreequiv}{\nequiv}
\newcommand{\gen}[2]{\eta_{#1,#2}}
\newcommand{\Hen}[1][]{\mathrm{Hen}_{#1}}
\newcommand{\sHen}[1][]{\mathrm{Hen}^{\sigma}_{#1}}
\newcommand{\open}{\mathring}
\newcommand{\closed}{\overline}

\newcommand\UFresgloc[2]{\textnormal{\textbf{(\(\bE^\infty_{#1,\open{#2}}\))}}}
\newcommand\UFres{\textnormal{\textbf{(\(\bE^\infty_{\res})\)}}}
\newcommand\UFvg{\textnormal{\textbf{(\(\bE^\infty_{\vg})\)}}}

\newcommand\UFradgloc[2]{\textnormal{\textbf{(\(\bE^\infty_{#1,\closed{#2}}\))}}}
\newcommand\Infres{\textnormal{\textbf{(\(\bI_{\res}\))}}}
\newcommand\Ram[1][]{\textnormal{\textbf{(\(\mathbf{FR}_{#1}\))}}}

\newcommand\Prep[2]{\textnormal{\textbf{(\(\bP^{#2}_{#1}\))}}}

\newcommand\Cball{\textnormal{\textbf{(\(\bC_{\balls}\))}}}
\newcommand\CV{\textnormal{\textbf{(\(\bC_{V}\))}}}
\newcommand\Cvg{\textnormal{\textbf{(\(\bC_{\vg}\))}}}
\newcommand\Dens{\textnormal{\textbf{(\(\bD\))}}}
\newcommand\SE{\textnormal{\textbf{(\(\mathbf{SE}\))}}}

\newcommand{\EQ}{\textnormal{\textbf{(\(\bQ_{\K}\))}}}
\newcommand{\Cod}{\textnormal{\textbf{(\(\mathbf{Cb}_{\K}\))}}}
\newcommand{\fa}[1]{\forall_{\!#1}}
\newcommand{\ex}[1]{\exists_{\!#1}}
\newcommand{\Lin}{\mathbf{Lin}}
\newcommand{\V}{\bV}
\newcommand{\Lmod}{\cL_{\mathrm{mod}}}
\newcommand{\Lvs}{\Lmod}

\newcommand\hens[1]{#1^{\mathrm{h}}}
\newcommand\lball[2]{#1[#2]}

\newcommand\Res{\bR}
\newcommand\Ann{\bA}
\newcommand\cval{\val_{\infty}}

\newcommand\GL{\mathrm{GL}}
\newcommand\tensor{\otimes}
\newcommand\cMid{\Mid_\infty}
\newcommand\cres{\res_\infty}
\newcommand\cresf{\resf_\infty}
\newcommand\lin[1]{#1^{\mathrm{leq}}}
\newcommand\id{\mathrm{id}}
\newcommand\W{\mathrm{W}}
\newcommand\VFA{\mathrm{VFA}^{\mathrm{mult}}_{0,0}}

\newcommand\prol{\nabla}
\newcommand\car{\mathrm{char}}
\newcommand\lind{\ind^{\mathrm{ld}}}
\newcommand\nlind{\not\ind^{\mathrm{ld}}}

\newcommand\cRes{\Res_\infty}

\newcommand\LRVun{\cL_{\RV_1}}
\newcommand\LPRVun{\cL_{\RV_1,\bP}}
\newcommand\LRVhyb{\cL_{\RV,\bP}^{\mathrm{hyb}}}
\newcommand\PK{\bP\K}
\newcommand\Pres{\bP\res}
\newcommand\PRes{\bP\Res}
\newcommand\Pvg{\bP\vg}
\newcommand\PRV{\bP\RV}
\newcommand\RVhyb{\RV^{\mathrm{hyb}}}
\newcommand\ac{\mathrm{ac}}
\newcommand\Lac{\cL_\ac}

\renewcommand\leq{\leqslant}
\renewcommand\geq{\geqslant}
\let\phitmp\varphi
\global\let\varphi\phi
\global\let\phi\phitmp
\let\phitmp\relax

\newcommand\Frob{\varphi}

\newcommand{\ind}{\downfree}

\renewcommand\mid{:}

\begin{document}
\begin{abstract}
We study interpretable sets in henselian and \(\sigma\)-henselian valued fields with value group elementarily equivalent to \(\Qq\) or \(\Zz\). Our first result is an Ax-Kochen-Ershov type principle for weak elimination of imaginaries in finitely ramified characteristic zero henselian fields --- relative to value group imaginaries and residual linear imaginaries. We extend this result to the valued difference context and show, in particular, that existentially closed equicharacteristic zero multiplicative difference valued fields eliminate imaginaries in the geometric sorts; the $\omega$-increasing case corresponds to the theory of the non-standard Frobenius automorphism acting on an algebraically closed valued field.

On the way, we establish some auxiliary results on separated pairs of characteristic zero henselian fields and on imaginaries in linear structures which are also of independent interest.
\end{abstract}

\maketitle


\section{Introduction}

In his seminal work "Une théorie de Galois imaginaire" \cite{Poi-ImGal}, Poizat introduced the idea that the classification of certain abstract constructions of model theory --- namely interpretable sets or Shelah's imaginaries --- could play an important role in our comprehension of specific structures. The classification of definable sets, in the guise of quantifier elimination results, has historically been used as a central ingredient in many applications of model theory. But the development of more sophisticated model theoretic tools, in particular stability theory, naturally took place in the larger category of quotients of definable sets by definable equivalence relations, \emph{i.e.}, interpretable sets. Shelah concretised this idea with his \(\mathrm{eq}\) construction that formally makes every interpretable set definable.

However, these interpretable sets immediately escape the realm of well understood and classified objects, complicating the possibility of applying new tools from stability theory in specific examples, in particular coming from algebra. Poizat's idea was that these interpretable sets should also be classified, and he did so in algebraically closed fields and in differentially closed fields.  In both cases,  he showed that they are all definably isomorphic to definable sets, \emph{i.e.}, the categories of definable and interpretable sets are equivalent --- we say that these structures \emph{eliminate imaginaries}. This property  later became an essential feature in model-theoretic applications, \emph{e.g.}, to diophantine geometry and algebraic dynamics. 

The question of elimination of imaginaries also has a very geometric flavour: given a definable family of sets \(X\subseteq Y\times Z\), one wishes to find a definable function \(f : Z \to W\) such that for all \(z_1,z_2\in Z\), \(X_{z_1} := \{y\in Y\mid(y,z_1)\in X\} = X_{z_2}\) if and only if \(f(z_1) = f(z_2)\) --- in other words, one wishes to find a canonical parametrisation of this family where each set appears exactly once. We refer the reader to \cref{imaginary} and \cite[Section~8.4]{TenZie} for further details on these notions and constructions.

Elimination of imaginaries results were then established for numerous
structures, but it was not until work of Haskell, Hrushovski and Macpherson
\cite{HasHruMac-ACVF} that the first complete classification of interpretable
sets in a valued field was proved. In this case, however, the field itself does
not eliminate imaginaries, as both the value group and the residue field are
interpretable but not isomorphic to a definable set. Nevertheless, one can add
certain well understood interpretable sets, the \emph{geometric sorts}. These
sorts consist of the field \(\K\) and, for all \(n\in\Zz_{>0}\), of the space
\(\Lat_n := \GL_n(\K)/\GL_n(\Val)\) of free rank \(n\) \(\Val\)-submodules of
\(\K^n\), where \(\Val\) denotes the valuation ring, and of the space \(\Tor_n
:= \bigcup_{s\in \Lat_n} s/\Mid s\) where \(\Mid\subseteq\Val\) is the unique
maximal ideal --- \emph{cf.}\ \cref{Sub:LanVF} for a precise definition of the
geometric language. The main result of \cite{HasHruMac-ACVF} states that the
theory \(\ACVF\) of algebraically closed non-trivially valued fields eliminates
imaginaries in the geometric sorts: given a definable family of sets
\(X\subseteq Y\times Z\), there exists a definable function \(f : Z \to W\),
with \(W\) a product of geometric sorts, such that for all \(z_1,z_2\in Z\),
\(X_{z_1} = X_{z_2}\) if and only if \(f(z_1) = f(z_2)\) --- equivalently the
category of sets interpretable in an algebraically closed valued field is
equivalent to the category of sets definable in its geometric sorts. One cannot
overstate the impact of this result, as it opened the way for the development of
geometric model theory in the context of valued fields. A beautiful illustration
of the power of these new methods is the work by Hrushovski and Loeser on
topological tameness in non-archimedean geometry \cite{HruLoe}.

In  this paper we consider imaginaries in more general classes of henselian valued fields of characteristic 0, and also in certain 
valued difference fields, \emph{i.e.}, valued fields endowed with a distinguished automorphism compatible with the valuation. 

In the last 25 years, the model theory of existentially closed difference fields, largely developed by Chatzidakis and 
Hrushovski (see \cite{ChHr99}), has led to several spectacular applications --- among others in algebraic dynamics. Note that the corresponding theory \(\ACFA\) does eliminate imaginaries and this fact plays an essential role in later developments. A very deep result of Hrushovski \cite{Hru-Frob}, which takes the form of a Frobenius-twisted version of the Lang-Weil estimates, implies that \(\ACFA\) is in fact the asymptotic theory of Frobenius automorphisms \(\Frob_q\): any non-principal ultraproduct of \((\alg{\Ff_q},\Frob_q)\) is a model of \(\ACFA\). Key properties of algebraic difference varieties may thus be read off from specializations to the Frobenius automorphisms.

It is also natural to consider the non-standard Frobenius acting on an algebraically closed valued field, \emph{i.e.}, the limit theory of the valued difference fields 
$(\alg{\Ff_p(t)},v_t,\Frob_p)$, as the prime \(p\) grows, where $v_t$ is an extension of the $t$-adic valuation. By results of Hrushovski \cite{Hru02} and Durhan \cite{Azg-OmeInc}, this limit theory corresponds to the theory of existentially closed valued difference fields of equicharacteristic zero with $\omega$-increasing automorphism --- \emph{cf.} \cref{val diff} for a detailed discussion. In fact, these structures naturally arise, as early as in Hrushovski's proof of the twisted  Lang-Weil estimates, in the study of algebraic difference varieties, by way of transformal specializations. One may thus expect that the development of a geometric model theory of valued difference fields will  turn out useful in the future in geometric applications --- as it did in the case of \(\ACVF\).

\subsection*{Main results}
The classification of imaginaries in $\ACVF$ by the geometric sorts was later extended to other valued fields: real closed valued fields \cite{Mel-RCVF}, separably closed valued fields of finite imperfection degree \cite{HilKamRid}, and \(p\)-adic fields and their ultraproducts \cite{HruMarRid} --- which allowed to uniformise and extend Denef's result on the rationality of certain zeta functions to interpretable sets.  The question remained whether a general principle underlined all these results. Such a principle was conjectured in the early 2000's by Hrushovski. The present paper establishes it for a large class of henselian fields, which covers most of the examples considered in applications, and extends it to valued fields with operators.

At this level of generality, one cannot expect elimination in the geometric sorts. Indeed, the residue field and the value group can be arbitrary and might not themselves eliminate imaginaries as is the case in all the results cited above. However, a fundamental idea of the model theory of valued fields, the so-called Ax-Kochen-Ershov principle, is that the model theory of a henselian equicharactersitic zero field should be controlled by its value group and residue field. This principle takes its name from the result of Ax and Kochen \cite{AxKoc} and independently Ershov \cite{Ers-AKE} that this is indeed the case for elementary equivalence, but this phenomenon has also been observed with respect to numerous other  aspects of valued fields, from model theoretic tameness (starting with  \cite{Del-NIP}) to motivic integration \cite{HruKaz}.

It is thus tempting to conjecture that, beyond the geometric sorts, imaginaries
in equicharacteristic zero henselian fields only arise from the value group and
the residue field. However, non trivial torsors of the residue field give rise
to serious obstructions to this conjecture. One is thus naturally led to define
the \(\res\)-linear imaginaries. Consider the two sorted language \(\Lvs\) of
\(\Ann\)-modules \(\V\) with the ring structure on \(\Ann\), the group structure
on \(\V\) and scalar multiplication. Given a (unary) interpretable set \(X\)
--- more precisely a definable quotient of the vector space sort \(\V\) --- in the
\(\Lvs\)-theory of dimension \(n\) vector spaces over a field and given some
$\Val$-lattice $s\in \Lat_n$, we can consider the interpretation
\(X^{(\res,s/\Mid s)}\) of \(X\) in the structure \((\res,s/\Mid s)\). We then
define:
\[\Tor_{n,X} := \bigsqcup_{s\in\Lat_n} X^{(\res,s/\Mid s)}.\]

Note that if \(X = \bV\), \(\Tor_{n,X} = \Tor_n\) and if \(X\) is the one element quotient of \(\bV\),  \(\Tor_{n,X} \isom \Lat_n\). In general, the \(\Tor_{n,X}\) are essentially those interpretable sets that admit definable surjections \(\Tor_n \to \Tor_{n,X} \to \Lat_n\). We write \(\lin{\res} := \bigsqcup_{n,X} \Tor_{n,X}\).

Before we state our main results, let us address some technical points. The
first is that we prove a result not only in equicharacteristic 0, but also in
finitely ramified mixed characteristic. In the latter case one needs to also
consider the higher residue rings \(\Res_\ell := \Val/\ell\Mid\), for
\(\ell\in\Zz_{>0}\), where \(\ell\Mid:=\{\ell\cdot x\mid x\in\Mid\}\). These
rings often play a crucial role in this situation, and they also come with their
linear imaginaries and hence we define \(\Tor_{n,\ell,X} :=
\bigsqcup_{s\in\Lat_n} X^{(\Res_\ell,s/\ell\Mid s)}\) and \(\lin{\Res} :=
\bigsqcup_{n,\ell,X} \Tor_{n,\ell,X}\), where \(X\) is now interpretable in the
\(\Lvs\)-theory of free rank \(n\) modules. Moreover, if the residue field
\(\res\) comes with additional structure, in the definition of \(\lin{\res}\)
and \(\lin{\Res}\), we need to consider all \(X\) interpretable in the
corresponding enrichement of the theory of free rank \(n\) modules.

The second point is that eliminating imaginaries often splits in two distinct problems: describing quotients under the action of finite symmetric groups (in other words finding canonical parameters for finite sets) and classifying interpretable sets up to one-to-finite correspondences: given a definable family of sets \(X\subseteq Y\times Z\), one wishes to find a one-to-finite definable correspondence \(F : Z \to W\) such that for all \(z_1,z_2\in Z\), \(X_{z_1} = X_{z_2}\) if and only if \(F(z_1) = F(z_2)\). This latter property is usually referred to as weak elimination of imaginaries and will be the main focus of this paper.

Our first main result is the following Ax-Kochen-Ershov principle for weak elimination of imaginaries, where \(\eq{\vg}\) refers to the collection all sets interpretable in the (enriched) ordered abelian group \(\vg\). 

\begin{thmA}[\cref{AKE EI}]
Let \((K,v)\) be a characteristic zero henselian valued field, possibly with angular components and added structure on the value group \(\vg\) and, separately, the residue field \(\res\). Assume that:
\begin{itemize}[leftmargin=50pt]
\item[\Cvg]  The induced structure on \(\vg\) is definably complete;
\item[\Ram] For every \(\ell\in\Zz_{>0}\), the interval \([0,\val(\ell)]\) is finite and \(\res\) is perfect;
\item[\Infres] The residue field \(\res\) is infinite;
\item[\UFres] The induced theory on \(\res\) eliminates \(\exists^\infty\).
\end{itemize}
Then \(K\) weakly eliminates imaginaries in the sorts \(\K\cup\eq{\vg}\cup\lin{\Res}\).
\end{thmA}

All the results with angular component, even in the algebraically closed (equicharacteristic zero) case, are new. 
Angular components are (compatible) multiplicative morphisms \(\ac_n : \K^\times\to\Res_n^\times\) extending the residue map on \(\Val^\times\). They play a key role in the development of the model theoretic study of valued fields,  in particular in the Cluckers-Loeser treatment of motivic integration \cite{CluLoe-MI}, by providing uniform cell decomposition results.

Definable completeness of the value group --- that is, the fact that every
definable subset of \(\vg\) has a supremum --- is a necessary hypothesis for the
conclusion to hold, since otherwise additional definable cuts appear, inducing
more definable $\Val$-submodules and hence more complex imaginaries. It is worth
noting that \(\mathrm{PRES}=\Th(\Zz)\) and \(\mathrm{DOAG}=\Th(\Qq)\) are the
only complete theories of pure ordered abelian groups which are definably
complete. As both \(\mathrm{PRES}\) and \(\mathrm{DOAG}\) eliminate imaginaries,
we thus get $\vg=\eq{\vg}$ under the assumptions of Theorem~A in case $\vg$ is
not enriched.

As a corollary, Theorem~A yields that if $F$ is a field of characteristic 0 which eliminates $\exists^\infty$, the theories of the valued fields $F((t))$ and $F((t^{\Qq}))$ (with or without angular components) weakly eliminate imaginaries in the sorts $(\K\cup\lin{\res})$ --- noting that $\vg\cong\bS_1$ may be identified with a sort in $\lin{\res}$. In the particular case that $F$ is (of characteristic 0 and) algebraically closed, real closed or pseudofinite, using results of Hrushovski on linear imaginaries, we deduce that $F((t))$ and $F((t^{\Qq}))$ (with or without angular components) eliminate imaginaries in the geometric sorts, after naming some constants in the pseudofinite and in the real closed case (see \cref{ACF-RCF-Psf} for the precise statement), thus obtaining an absolute elimination result in these cases.

Without angular components, this provides  alternate proofs of Mellor's result \cite{Mel-RCVF} for \(\Rr((t^{\Qq}))\) and Hrushovski-Martin-Rideau's result \cite{HruMarRid} for \(F((t))\) where \(F\) is pseudofinite of characteristic 0. Independent work of Vicaria \cite{Vic-EILaur} also yields the case of \(\Cc((t))\), although her work also applies to more general value groups.

In mixed characteristic, the main example covered by Theorem~A is
$W(\alg{\Ff_p})$ --- where $\alg{F}$ denotes the (field-theoretic) algebraic
closure of $F$ --- the (fraction field of the) ring of Witt vectors with
coefficients in $\alg{\Ff_p}$, and more generally finite extensions of $W(F)$
for any perfect infinite field $F$ of characteristic $p$ which eliminates
$\exists^\infty$. However, in mixed characteristic, \(\lin{\Res}\) involves
modules over higher residue rings. We conjecture that, when \(F = \alg{F}\),
these linear structures also eliminate imaginaries. Thus, Theorem~A provides an
important step towards proving that the imaginaries of $W(\alg{\Ff_p})$ are
classified by the geometric sorts as well.\medskip

The second main result of this paper concerns valued difference  fields.
Quantifier elimination (and hence an Ax-Kochen-Ershov principle for elementary equivalence) has been proved for various classes. First for isometries in \cite{Sca-RelFrob,BelMacSca,AzgvdD}, then for \emph{\(\omega\)-increasing} automorphisms --- for every \(x\in\Mid\), \(\val(\sigma(x)) > \Zz\cdot \val(x)\) --- in \cite{Azg-OmeInc,Hru02}. Both of these contexts were subsumed in later work of Kushik \cite{Pal-Mul} on \emph{multiplicative} automorphisms where the automorphism acts as multiplication by some element of an ordered field  (\emph{cf.} \cref{VFA} for precise definitions). Finally Durhan and Onay \cite{DurOna} proved that these results hold without any hypothesis on the automorphism.

Our second result focuses on the multiplicative setting where we prove an absolute elimination of imaginaries result for the respective model-companions:

\begin{thmB}[\cref{VFA-EI}]
The theory \(\VFA\) eliminates imaginaries in the geometric sorts.
\end{thmB}

Since the isometric case and the \(\omega\)-increasing case correspond, respectively, to the asymptotic theory of \(\Cc_p\) with an isometric lifting of the Frobenius and to \(\alg{\Ff_p(t)}\) with the Frobenius, an immediate corollary of these results is a uniform elimination of imaginaries for large \(p\) in these structures. By elimination of imaginaries in $\ACVF$, the result is even uniform for all $p$ in the latter case.


\subsection*{Overview of the paper}
The proofs of both theorems follow the same general strategy and many technical results are shared between the two. The proof consists of three largely independent steps (Sections 3 to 5). 

In stable theories, every type \(p\) --- a maximal consistent set of definable sets --- is definable, \emph{i.e.}, for every definable \(X\subseteq Y\times Z\), the set \(\{z\in Z\mid X_z \in p\}\) is definable. This was used in many proofs of weak elimination of imaginaries in the stable context to reduce the problem of finding canonical parameters for definable sets to finding canonical parameters for types; which, counter-intuitively maybe, is a simpler problem. In \cite{Hru-EIACVF}, Hrushovski formalised the idea that even in an unstable context, this reduction could also prove useful, provided definable types were dense: over any algebraically closed imaginary set of parameters, any definable set contains a definable type.

The first step of the proof consists in proving such density results. But the above statement cannot hold in the full generality of henselian equicharacteristic zero valued fields since it might already fail in the residue field. We prove however that, under certain hypotheses, \emph{quantifier free} definable types are dense, \emph{cf.} \cref{def dens}. This result does not apply to discrete valued fields since the family of intervals contains arbitrarily large finite sets. In \cref{inv dens}, we do however prove that the density of quantifier free invariant types holds in this context. 

The proof improves on similar results in \cite{Rid-VDF} and the general idea is the same. In arity one, we look for a minimal finite set of balls covering the given definable set. The general case proceeds by fibration in relative arity one and by considering germs of functions into the space of (finite sets of) balls instead of actual balls. This fibration process is where most of the technical assumptions of Theorem~A are used, in particular the elimination of \(\exists^\infty\) in the residue field.

In contexts with an absolute elimination of quantifiers (as, \emph{e.g.}, in \cite{Rid-VDF,HilKamRid}) this first density result (and the implicit computation of canonical bases) suffices to conclude that weak elimination of imaginaries holds. In equicharacteristic zero henselian (and \(\sigma\)-henselian) fields, types come with more information than quantifier free types; an information that mostly lives in the short exact sequence \(1\to \res^\times \to \RV^\times := \K^\times/(1+\Mid) \to \vg^\times \to 0\). The second step of our proof, \cref{inv}, consists in showing that quantifier free invariant types have invariant completions over \(\RV\) (and \(\res\)-vector spaces) --- this generalises to mixed characteristic by considering the higher residue rings.

By quantifier elimination relative to \(\RV\), this step reduces to, given an
invariant type, computing canonical generators of the structure generated by
(realisations of) the type in \(\RV\). Note that in the conclusion of \cref{inv}
the types considered are invariant over definable sets which are of the same
size as the model. We do however show various folklore results implying that
this is a well behaved notion when these sets are stably embedded.

The third step consists in studying imaginaries in \(\RV\), which is left as a black box in the previous steps. We show, in the spirit of \cite[Section~3.3]{HruKaz}, that the imaginaries in the short exact sequence \(1\to \res^\times \to \RV^\times \to \vg^\times \to 0\) come essentially from \(\res\) and \(\vg\). To establish the results we need, as in \cite[Section~3.3]{HruKaz}, we consider more generally structures given by (enriched) short exact sequences \(0\rightarrow \bA\rightarrow \bB\rightarrow \bC\rightarrow 0\) of \(R\)-modules for some ring \(R\). But our result is in a sense orthogonal to the one of Hrushovski and Kazhdan since we require \(\bB\) to be a pure (in the sense of model theory) extension of \(\bA\) and \(\bC\), which can be both arbitrarily enriched, whereas \cite[Lemma~3.21]{HruKaz} has strong hypotheses on \(\bA\) and \(\bC\) and no hypothesis on \(\bB\).

\Cref{T:EI-ShortExSequ} is the first version of a series of such reductions of increasing complexity so as to cover the various cases that we require, the ultimate version, \cref{V:EI-ShortExSequ-Complex-C} allowing controlled torsion in \(\vg\), auxiliary sorts on both the \(\res\) and \(\vg\) sides and considering not one but a projective system of short exact sequences.

These three steps put together allow us to prove relative results like Theorem~A. However, absolute results like Theorem~B or \cref{ACF-RCF-Psf} require one last ingredient: the classification of imaginaries in collections of vector spaces --- linear structures in the terminology of \cite{Hru-GpIm}. Our contribution consists in a twisted version, \emph{cf.} \cref{F:ChPi-TA}.(4), of Hrushovski's result on $\ACF_0$-linear structures with flags and roots endowed with an automorphism (the final step to prove Theorem~B) and a version, \cref{RCF lin EI}, for real closed fields.

The plan of the paper is as follows. In Section~2, we provide some preliminary results on imaginaries, separated pairs of valued fields (in the sense of Baur), on valued difference fields and on linear structures. Section~3 is devoted to the proof of the two density results for definable (resp. invariant) types mentioned above. The fact that invariant quantifier free types are invariant over $\RV$ (and $\Res_n$-modules), is established in Section~4. In Section~5, we prove the results about imaginaries in certain (enriched) short exact sequences of modules. Finally, in Section~6, we put everything together and prove our main results, in particular Theorem~A and Theorem~B.\medskip

\subsection*{Acknowledgement} The authors would like to thank the anonymous referees for their numerous comments and advice on an earlier version of this manuscript.


\section{Preliminaries}

\begin{convention}\label{A-points}
Throughout this paper, if \(M\) is an \(\cL\)-structure, \(X\) is \(\cL(M)\)-definable and \(A\subseteq M\), then \(X(A)\) denotes \(X\cap A\).
\end{convention}

We adopt this convention, as there are too many structures at play to not be explicit as to which definable (or algebraic) closures we want to consider. For this reason, setting  \(X(A):=X\cap\dcl(A)\) would lead to ambiguities.

\subsection{Imaginaries}\label{imaginary}

Let \(T\) be an \(\cL\)-structure. The language \(\eq{\cL}\) is the language containing \(\cL\) with one additional sort \(S_X\) for every \(\cL\)-definable set \(X \subseteq Y\times Z\), where \(Y\) and \(Z\) are product of sorts, and one additional symbol \(f_X : Z \to S_X\). The \(\eq{\cL}\)-theory \(\eq{T}\) is then obtained as the union of \(T\), the fact that the \(f_X\) are surjective and that their fibers are the classes of the equivalence relation defined by \(X_{z_1} = \{y\in Y: (y,z_1)\in X\} = X_{z_2}\).

Any \(M\models T\) has a unique expansion to a model of \(\eq{T}\) denoted \(\eq{M}\) --- whose points are called the \emph{imaginaries}. Throughout this paper, notations with exponent \(\mathrm{eq}\) --- like \(\dcleq\) or \(\acleq\) --- will refer to the \(\eq{\cL}\)-structure of some ambient \(\eq{M}\).

Given \(M\models T\) and an \(\cL(M)\)-definable set \(X\), we denote by \(\code{X} \subseteq \eq{M}\) the intersection of all \(A = \dcleq(A)\subseteq \eq{M}\) such that \(X\) is \(\eq{\cL}(A)\)-definable. By construction of \(\eq{T}\), \(X\) is \(\eq{\cL(\code{X})}\)-definable, so it is the smallest \(\dcleq\)-closed set of definition for \(X\). Any \(\dcleq\)-generating subset of \(\code{X}\) is called \emph{a code} of \(X\).

If \(\mathcal{D}\) is a collection of sorts of \(\eq{\cL}\) --- equivalently, a collection of \(\cL\)-inter\-pretable sets --- we say that \(X\) is \emph{coded} in \(\mathcal{D}\) if it is \(\eq{\cL}(\mathcal{D}(\code{X}))\)-definable --- \emph{i.e.}, it admits a code in \(\mathcal{D}\). The theory \(T\) is said to \emph{eliminate imaginaries} in \(\mathcal{D}\) if, for every \(M\models T\), every \(\cL(M)\)-definable set \(X\) is coded in \(\mathcal{D}\) --- equivalently, for every \(e\in\eq{M}\), there is some \(d\in \mathcal{D}(\dcleq(e))\) such that \(e\in\dcleq(d)\). By compactness, this is equivalent to the definition in the introduction, provided \(\dcl(\emptyset)\) contains two elements. Finally, we say that the theory \(T\) \emph{weakly eliminates imaginaries} in \(\mathcal{D}\) if for every \(e\in\eq{M}\), there is some \(d\in \mathcal{D}(\acleq(e))\) such that \(e\in\dcleq(d)\).

We refer the reader to \cite[Section~8.4]{TenZie} for a detailed exposition of these notions.

\subsection{The languages of valued fields}\label{Sub:LanVF}

Any valued field \((K,v)\) can be considered as a structure in the language \(\Ldiv\) with one sort \(\K\) for the valued field, the ring language and a binary relation \(x|y\) interpreted as \(v(x) \leq v(y)\). Note that in every language of valued fields that we will consider there is a sort \(\K\) for the valued field and hence, whenever \(M\) is a structure representing a valued field, \(\K(M)\) will denote the underlying valued field.

The language \(\Ldiv\) owes its widespread use to the following result, essentially due to Robinson \cite{Rob-ACVF}:

\begin{fact}
The \(\Ldiv\)-theory \(\ACVF\) of algebraically closed non trivially valued fields eliminates quantifiers.
\end{fact}

\begin{notation}
We will write \(\cL_0 := \Ldiv\) and throughout this paper, notations with an index \(0\) --- like \(\dcl_0\), \(\acl_0\) or \(\qftp\) --- will refer to the quantifier free \(\cL_0\)-structure; or equivalently the structure induced by any model of \(\ACVF\) containing the valued field under consideration.
\end{notation}

Given a valued field, seen as an \(\cL_0\)-structure, we will denote by \(\Val := \{x\in\K\mid v(x) \geq 0\}\) its valuation ring, \(\Mid := \{x\in\K\mid v(x) > 0\}\) its maximal ideal, \(\res := \Val/\Mid\) its residue field and \(\vg := \K/\Val^\times\) its value \emph{monoid}; it is the union of the value group \(\vg^\times = \K^\times/\Val^\times\) and the class of \(0\) that we usually denote \(\infty\). We also denote \(\resf : \Val \to \res\) and \(\val : \K \to \vg\) the canonical projections.
More generally, for every \(n\in\Zz_{>0}\), we write \(\Res_n := \Val/n \Mid\). Let also \(\resf_n : \Val \to \Res_n\) be the canonical projection, \(\cRes\) the (pro-definable) set \(\plim_n \Res_n\) and \(\cresf : \Val \to \cRes\) the natural map. Note that, working in a sufficiently saturated model, \(\cRes \isom \Val/\cMid\), where \(\cMid := \{x\in\K\mid \cval(x)>\Delta_\infty\}\) and  \(\Delta_{\infty} \leq \vg\) is the convex subgroup generated by \(\val(\car(\res))\), in mixed characteristic, and \(\Delta_\infty = 0\), otherwise. It is a valuation ring whose fraction field is naturally identified with the residue field \(\cres\) associated to the (equicharacteristic) valuation \(\cval : \K \to \vg \to \vg/\Delta_{\infty}\). We also define \(\Res := \bigsqcup_{n>0} \Res_n\).

Although most of the present paper is rather insensitive to the choice of
language for valued fields --- or, rather, we work in \(\eq{\cL_0}\) --- we will
at times need to work in certain languages tailored for specific elimination
results. The first of them is the Haskell-Hrushovski-Macpherson geometric
language. For every \(n\in\Zz_{>0}\), let \(\Lat_n \isom \GL_n(\K)/\GL_n(\Val)\)
be the (interpretable) set of rank \(n\) free sub-\(\Val\)-modules of \(\K^n\),
and \(\Tor_n := \bigcup_{s\in \Lat_n} s/\Mid s\). Let \(\Lat := \bigcup_n
\Lat_n\), \(\Tor := \bigcup_n \Tor_n\) and \(\Geom := \K \cup \Lat \cup \Tor\).
We also denote by \(s_n : \GL_n(\K) \to \Lat_n\), \(t_n : \GL_n(\K) \to \Tor_n\)
and \(\tau_n : \Tor_n \to \Lat_n\) the canonical projections\label{tau}. We will
identify \(\Lat_n\) with the zero section inside \(\Tor_n\). Note that
\(\GL_n(K)\) naturally acts transitively on \(\Lat_n\) and
\(\Tor_n\sminus\Lat_n\), and that the map \(\tau_n\) is compatible with these
actions.

These interpretable sets (and the \emph{geometric language} of which they are
the sorts, which also contains the maps \(s_n\), \(t_n\) and \(\tau_n \)) were
introduced to classify imaginaries in \(\ACVF\):

\begin{fact}[{\cite[Theorem\,1.0.1]{HasHruMac-ACVF}}]\label{EI ACVF}
The theory \(\ACVF\) eliminates imaginaries in the geometric sorts \(\Geom\).
\end{fact}

The second language that we will use allows for a description of definable sets
in certain henselian fields. The exact language that we use was introduced  by
Flenner in \cite{Fle11}. For every \(n\in\Zz_{>0}\), let \(\RV_n\) be the
multiplicative monoid \(\K/(1 + n \Mid)\); it is the union of the group
\(\RV_n^\star = \K^\times/(1+n\Mid)\) and the class of \(0\), also denoted
\(0\). Let \(\rv_n : \K \to \RV_n\) and \(\rv_{n,m} : \RV_n \to \RV_m\), where
\(m\) divides \(n\), denote the canonical projections.
The valuation induces a map \(\RV_n\to\vg\) that we also denote \(\val\). This map induces a short exact sequence:
\[1\to \Res_n^\times \to \RV_n^\times \to \vg^\times \to 0.\]

\begin{remark}
If \(\val(n) = \val(m)\), then \(\rv_{n,m} : \RV_n \isom \RV_m\). In particular, in equicharacteristic zero all \(\RV_n\) are canonically isomorphic to \(\RV_1\). We allow this redundancy in order to have a uniform treatment of characteristic zero henselian valued fields.

Moreover, in positive characteristic \(p\), if \(n\) is prime to \(p\), \(\RV_n \isom \RV_1\) and othewise \(\RV_n \isom \K\). In that case, it makes more sense to only consider \(\RV_1\) --- see below.
\end{remark}

Moreover,  \(\RV_n\) is endowed with the trace of addition which we denote, in Krasner's hyperfield manner: \(\zeta \oplus \xi := \{\rv_n(x+y) \mid \rv_n(x) = \zeta \text{ and  }\rv_n(y) = \xi\} \subseteq \RV_n\). We say that \(\zeta \oplus \xi\) is well-defined when \(\zeta \oplus \xi = \{\chi\}\) is a singleton, and we often write \(\zeta \oplus \xi = \chi\) in that case.

\begin{remark}\label{rv ball}
Note that for any two disjoint balls \(b_1\) and \(b_2\), in some valued field \((K,v)\), and any \(a_i,c_i\in b_i\), \(\rv_1(a_1-a_2) = \rv_1(c_1-c_2)\). We will denote by \(\rv_1(b_1-b_2)\) this common value. If \(b_1\cap b_2\neq\emptyset\), by convention, \(\rv_1(b_1-b_2) = 0\).
\end{remark}

We denote by \(\RV_\infty\) the (pro-definable) set \(\plim_n \RV_n\), and \(\rv_\infty : \K\to\RV_\infty\) denotes the natural map. Note that \(\RV_\infty \isom \K/(1+\Mid_\infty)\), as pro-definable sets. We also denote \(\RV := \bigsqcup_n \RV_n\).

Let \(\LRV\) be the language with sorts \(\K\), \(\vg\) and \(\RV_n\), for all \(n\in\Zz_{>0}\), the ring structure on \(\K\), 
ordered (abelian) monoid structure with a constant for \(\infty\) on \(\vg\), multiplication, constants \(0\), \(1\) and a ternary predicate \(\oplus\) on each \(\RV_n\), the valuation map \(\val : \K\to\vg\) and the maps \(\rv_n : \K\to\RV_n\) and \(\rv_{n,m} : \RV_n\to\RV_m\). Let \(\LRVun\) be its restriction to the sorts \(\K\), \(\vg\) and \(\RV_1\). 

\begin{remark}\label{define-oplus}
If the interval 
\([0,\val(n)]\) is finite, then the predicate \(\oplus\) on \(\RV_n\) is definable (in general with parameters, and without parameters in case $\val(n)=0$) using addition on \(\Res_n\).
\end{remark}

\begin{proof}
If \(\val(\zeta)\leq\val(\xi)\), then \(\zeta \oplus\xi = \zeta\cdot r_n^{-1}(1 + (\zeta^{-1}\xi))\), where \(r_n\) is the map sending \(\rv_n(x)\) to \(\resf_n(x)\) whenever \(x\in\Val\). The remaining cases are dealt with by symmetry. Therefore, it suffices to show that the map \(r_n\) is definable. Let \(\pi\in\RV_n\) be an element of minimal positive valuation. Then for every \(\xi\in \RV_{n}\), if \(\val(\xi) = \ell\val(\pi) \in [0,\val(n)]\), then \(r_n(\xi) = r_n(\pi)^\ell \cdot (\rv_n(\pi)^{-\ell} \xi)\); and if \(\val(\xi)> \val(n)\), \(r_n(\xi) = 0\). So \(r_n\) is indeed definable (with parameters \(\pi\) and \(r_n(\pi)\), unless \(\val(n) = 0\)).
\end{proof}


\begin{definition}\label{benign}
We say that a valued field \((K,v)\) is:
\begin{itemize}
\item \emph{algebraically maximal} if it does not admit non trivial immediate algebraic extensions;
\item \emph{Kaplansky} if \(\vg^\times(K)\) is \(p\)-divisible and any finite extension of \(\res(K)\) has degree prime to \(p\), where \(p = \car(\res(K))\) if it is positive and \(p = 1\) otherwise;
\item \emph{finitely ramified} if for any $\ell\in \Zz_{>0}$ the interval \([0,\val(\ell)]\) in \(\vg(K)\) is finite.
\end{itemize}
\end{definition}

Note that a finitely ramified valued field is algebraically maximal if and only if it is henselian, \emph{cf.} \cite[Theorem~4.1.10]{EngPres}.

The following quantifier elimination results are due, respectively, to Basarab \cite[Theorem\,B]{Bas-EQHens} in characteristic zero and Delon \cite[Théorème~3.1]{Del-These} in positive characteristic (see also \cite[Corollary~2.2 and Theorem~2.6]{Kuh}):

\begin{fact}[{}]\label{EQ rv}
\begin{itemize}
\item Let \(\cL\) be an \(\RV\)-enrichment of \(\LRV\) and \(T\) an \(\cL\)-theory containing the theory \(\Hen[0]\) of henselian valued fields of characteristic zero. Then \(T\) eliminates field quantifiers.
\item Let \(\cL\) be an \(\RV_1\)-enrichment of \(\LRVun\) and \(T\) an \(\cL\)-theory containing the theory of equicharacteristic \(p\) algebraically maximal Kaplansky valued fields, for some fixed \(p>0\). Then \(T\) eliminates field quantifiers.
\end{itemize}
\end{fact}

\subsection{Separated pairs of valued fields}

In this section, we will gather some results about separated pairs of valued fields, in particular 
concerning pure stable embeddedness of the residue field and value group pairs in specific contexts. 
In equicharacteristic zero, most of the results below follow from work of Leloup \cite{Lel90}; and from work of Rioux \cite{Rioux-PhD}, in unramified mixed characteristic.

Recall that an extension $L/K$ of valued fields is called \emph{separated} if every finite-dimensional $K$-vector subspace of $L$ admits 
a $K$-\emph{valuation basis}, \emph{i.e.}, a $K$-basis $(b_1,\ldots,b_n)$ which is \emph{valuation independent} over $K$: for any $a_1,\ldots,a_n\in K$ one has 
$\val(\sum a_ib_i)=\min\val(a_ib_i)$. Also, for field extensions $K\subseteq L\subseteq U$ and $K\subseteq K'\subseteq U$, we write $L\lind_{K}K'$ if $L$ and $K'$ are linearly disjoint over $K$.

\begin{definition}
Let $K\subseteq L\subseteq U$ and $K\subseteq K'\subseteq U$ be valued field extensions.
\begin{itemize}
\item We say $L$ and $K'$ are \emph{$\vg\res$-independent} over $K$, denoted by $L\ind^{\vg\res}_KK'$, if $\res(L)\lind_{\res(K)}\res(K')$ and $\vg(L)\cap \vg(K')=\vg(K)$.
\item Assume that $L/K$ is separated. Then $L$ is said to be \emph{valuatively disjoint} from $K'$ over $K$, denoted by $L\ind^{vd}_KK'$, 
if whenever a tuple $(b_1,\ldots,b_n)$ from $L$ is valuation independent over $K$, it is valuation independent over $K'$.
\end{itemize}
\end{definition}

\begin{fact}\label{F:characterize-vd}
Let $K\subseteq L\subseteq U$ and $K\subseteq K'\subseteq U$ be valued field extensions, with $L/K$ separated and $L\ind^{\vg\res}_KK'$. Set $L':=LK'$. Then we have the following:
\begin{enumerate}
\item $L\ind^{vd}_KK'$ --- in particular, $L\lind_KK'$;
\item $L'/K'$ is separated;
\item $\res(L')=\res(L)\res(K')$ and $\vg(L')=\vg(L)+\vg(K')$;
\item If $L_1\subseteq U$ and $f:L\cong L_1$ is an isomorphism over $K\cup \res(L)\cup\vg(L)$, then $f$ extends (uniquely) to an isomorphism $f':L'\cong L_1K'$ over $K'$.
\end{enumerate}
\end{fact}

\begin{proof}
This is shown by adapting the proof of the corresponding result for $K$
maximally valued from \cite[Proposition~12.11]{HasHruMac-Book}.
\end{proof}

\subsubsection{Reduction to \(\RV\)} Most of this paper will be concerned with characteristic zero finitely ramified fields, however, for future reference, we will state and prove certain results, mostly regarding pairs, in all characteristics, as the arguments are essentially identical. 

\begin{notation}
Given a multisorted language $\cL$, we let $\cL_{\bP}$ be the associated language of pairs, \emph{i.e.}, for every sort $\mathbf{S}$ from $\cL$ we add a unary predicate $\bP\mathbf{S}$ of sort $\mathbf{S}$ to the language. If $N$ is an $\cL$-substructure of $M$, we will consider the pair of $\cL$-structures $\widetilde{M}=(M,N)$  as an $\cL_{\bP}$-structure in the natural way, \emph{i.e.}, $\bP\mathbf{S}(\widetilde{M})=\mathbf{S}(N)$ for each sort $\mathbf{S}$. We denote $\bP(\widetilde{M})=N$  the whole $\cL$-substructure singled out by the $\bP\mathbf{S}$'s. Instead of $\widetilde{M}=(M,\bP(\widetilde{M}))$, we will often write 
$(M,\bP(M))$. Given a quantifier free \(\cL\)-definable set \(X\), we extend the above notation and write \(\bP X\) for the \(\LP\)-definable set whose points in \((M,\bP(M))\) are the \(\bP(M)\)-points of \(X\).
\end{notation}

In \(\LPRVun\), the class of separated pairs of valued fields may be axiomatized. For technical reasons, we will consider such pairs in a hybrid language, adding higher $\RV$ sorts for the small valued field. Formally, we let \(\LRVhyb\) be the language consisting of \(\LPRVun\) together with additional sorts \(\bP\RV_n\) for all $n\geq2$ and all symbols of $\LRV$, where for $n\geq2$ we use $\bP\RV_n$ instead of $\RV_n$. \emph{E.g.}, in \(\LRVhyb\), for $n\geq2$ we have a function symbol $\rv_n:\K\rightarrow\bP\RV_n$ and a ternary relation symbol $\oplus$ on $\bP\RV_n$. 

Let $T^\star$ be a theory of separated pairs $(M,\bP(M))$ of henselian valued fields in the language \(\LRVhyb\). Here, $M$ is the \(\LRVun\)-structure associated to a valued field, \(\bP(M)\) the $\LRV$-structure (interpreted on the respective $\bP \mathbf{S}$'s) of the corresponding valued subfield. (For $n\geq2$, we extend $\rv_n$ from $\bP\K(M)$ to $\K(M)$ trivially, setting $\rv_n(a):=0\in\bP\RV_n(M)$ for any $a\in\K(M)\setminus\bP\K(M)$.) We assume that \(M\) eliminates field quantifiers in \(\LRVun\) and \(\bP(M)\) eliminates fields quantifiers in \(\LRV\). Note that we do not assume that \(\PRV_1\) is stably embedded in \(\RV_1\).

\begin{remark}
In positive characteristic \(p\), since \(\RV_p \isom \K\), eliminating quantifiers from the sort \(\K\) in \(\LRV\) is an empty assumption and it makes more sense to consider pairs of \(\LPRVun\)-structures instead, as in \cref{LPRVun}.
\end{remark}

By a hybrid \(\RV\)-structure, we mean a structure (elementarily equivalent to) \((\RV_1(M),\PRV(M))\), where \(M\models T^\star\) --- with the restriction of the \(\LRVhyb\)-structure. We also denote \(\RVhyb\) the set of sorts \(\{\RV_1,\vg\} \cup \{\PRV_n\mid n\geq2\}\).

\begin{lemma}\label{L:elem-ind}
Let $M_0\preccurlyeq N_0$ be hybrid \(\RV\)-structures. Then $\res(M_0)\lind_{\Pres(M_0)}\Pres(N_0)$ and $\vg(M_0)\cap \Pvg(N_0)=\Pvg(M_0)$.
\end{lemma}

\begin{proof}
Immediate from the elementarity of the extension.
\end{proof}

Let \(M_0\) be a hybrid \(\RV\)-structure, say (elementarily equivalent to) a structure of the form \((\RV_1(M),\PRV(M))\), where \(M\models T^\star\). We say that \(M_0\) is finitely ramified if \(\bP(M)\) is --- \emph{i.e}, \([0,v(\ell)] \cap \Pvg\) is finite for every \(\ell\in\Zz_>0\). In that case, we also assume that \(\Pres(M)\) is perfect. In mixed characteristic $(0,p)$, \(R_\infty := \plim_n \PRes_n(M_0)\) is then a \(p\)-ring  with perfect residue field \(\Pres(M_0)\) --- \emph{cf.} \cite[Ch. II \S 5]{Ser-CL} --- and \(p\) is not a divisor of zero. So it is a complete mixed characteristic discrete valuation ring and a finite extension of \(\W(\Pres(M_0))\) of degree \(\val(p)\), where $\W(k)$ denotes the ring of Witt vectors over $k$. Let \(\pi\) be a uniformizer of \(R_\infty\) --- \emph{i.e}, a generator of the maximal ideal --- and \(P\) its minimal polynomial over \(W(\Pres(M_0))\). 

\begin{definition}\label{ram cst}
Ramification constants refer to the (infinite) tuple, in \(\Pres(M_0)\), of Witt coordinates of the coefficients of a polynomial \(P\) as above.
\end{definition}

\begin{lemma}\label{compatible-section}
Let \(M_0\) be a finitely ramified hybrid \(\RV\)-structure. Assume that $\res^\times(M_0)$ is divisible, or that $\Pvg(M_0)$ is a pure subgroup of $\vg(M_0)$. Then, the following hold:
\begin{enumerate}
\item $\res$ and $\vg$ are purely stably embedded and orthogonal.
\item The theory of \(M\) is determined by the theories of the $\res$-pair (with
a choice of ramification constants), the $\vg$-pair and ramification data ---
\emph{i.e}, the theory stating that, for every \(n\), \(\PRes_n\) has a
uniformizer which is a zero of the polynomial whose Witt coefficients are the
ramification constants.
\end{enumerate}
Moreover, the statements (1) and (2) hold in any $\res$-$\vg$-enrichment of
$M_0$, \emph{i.e}, a \(\res\)-enrichment of a \(\vg\)-enrichment of \(M_0\). 
\end{lemma}

Here, when we say that a definable set is purely stably embedded, we mean that
its induced structure is given by (a definable expansion of) the restriction of
the language to that set. For example, the structure on \(\res\) is that of a
pair of fields and the structure on \(\vg\) is that of a pair of ordered groups.

\begin{proof}
We may assume that \(M_0\) is of the form \((\RV_1(M),\PRV(M))\) for some $\aleph_1$-saturated \(M\models T^\star\). Then, as $\aleph_1$-saturated modules are pure-injective, there is a section of the valuation map restricted to the small valued field $\bP(M)$, inducing 
coherent splittings of the sequences \[1\rightarrow\PRes_n^\times(M_0)\rightarrow\PRV_n^\times(M_0)\rightarrow\Pvg^\times(M_0)\rightarrow0\] for all $n\geq1$. In mixed characteristic, we may assume that the splitting is \emph{normalized}: the chosen uniformizer \(\pi\) is in the image of the splitting, equivalently, \(\ac_n(\pi) = 1\) if \(\ac_n\) is the angular component map induced by the splitting. Indeed, the group \(\Delta\) generated by \(\val(\pi)\) is convex, so the quotient is also ordered and hence torsion free. Since \(\PK^\times(M)/\Val^\times(M)\cdot \pi^\Zz \isom \Pvg(M_0)/\Delta\), the extension \(\bP\Val^\times(M)\cdot \pi^\Zz \leq \PK^\times(M)\) is also pure. Pure-injectivity of \(\bP\Val^\times(M)\) (which is an $\aleph_1$-saturated abelian group) then allows to extend the retraction \(\bP\Val^\times(M)\cdot  \pi^\Zz \to \bP\Val^\times(M)\) sending \(\pi\) to \(1\) to the whole of \(\PK^\times(M)\).

It further follows from the assumptions that the splitting of $1\rightarrow\Pres^\times(M_0)\rightarrow\PRV_1^\times(M_0)\rightarrow\Pvg^\times(M_0)\rightarrow0$ extends to a splitting of the sequence $1\rightarrow\res^\times(M_0)\rightarrow \RV_1^\times(M_0)\rightarrow\vg^\times(M_0)\rightarrow0$. To see this, note first that if $h:\PRV_1^\times(M_0)\rightarrow\Pres^\times(M_0)$ is a retraction of the inclusion map, \emph{i.e.}, $\restr{h}{\Pres^\times(M_0)}=\id_{\Pres^\times(M_0)}$, $h$ extends (uniquely) to a homomorphism $\widetilde{h}:\res^\times(M_0)\cdot \PRV_1^\times(M_0)\rightarrow \res^\times(M_0)$ which is the identity on $\res^\times(M_0)$, since $\res^\times(M_0)\cap \PRV_1^\times(M_0)=\Pres^\times(M_0)$. It is enough to show that $\widetilde{h}$ may be extended to a homomorphism 
$h':\RV_1^\times(M_0)\rightarrow\res^\times(M_0)$. In case $\res^\times(M_0)$ is divisible, this is clear, since divisible abelian groups are injective. In case $\Pvg(M_0)$ is a pure subgroup of $\vg(M_0)$, we conclude as above.

Note that the additional structure on \(\RVhyb\), beyond the abelian structure, is given by \(\oplus\) and some $\res$-$\vg$-enrichment. As explained in \cref{define-oplus}, \(\oplus\) can be defined using the ring structure on \(\res\) and the \(\PRes_n\) (using the splitting, no further constants are required).
Moreover, \(\PRes_n\) is a finite extension, generated by the zero of a polynomial with coefficients the ramification constants, and, as such, is \(\emptyset\)-interpretable in \(W_n(\Pres)\), which is itself \(\emptyset\)-interpretable in \(\Pres\). So, if we add the splittings, \(\RVhyb\) is (identified to) a $\res$-$\vg$-enrichment of the product of \(\res\) and \(\vg\). In the product structure, (1) and (2) are clear, even for $\res$-$\vg$-enrichments. The result follows, as (1) and (2) are preserved in any reduct of the product structure that carries the whole structure on $\res$ and $\vg$.
\end{proof}

\begin{remark}
\label{PRV st emb}
If \(\Pres\) is (purely) stably embedded in \(\res\) and \(\Pvg\) is (purely)
stably embedded in \(\vg\), then \(\PRV\) is (purely) stably embedded in
\(\RVhyb\). Indeed, this is true with a splitting as in the proof above since
\(\RV\) and \(\PRV\) can be identified to products, and it remains true after
removing the splitting
\end{remark}
    
We now get back to the \(\LRVhyb\)-theory $T^\star$ of separated pairs of valued
fields. Let \(M,N\models T^\star\), where we suppose that  $N$ is \(\lvert
M\rvert^+\)-saturated, and let \(A\leq M\) and \(f : A \to N\) be some
embedding.

\begin{definition}
We say that:
\begin{enumerate}
\item \(A\) is \emph{good} if $\PK(A)\leq \K(A)$ is a separated extension of
valued fields with $$\K(A)\ind^{\vg\res}_{\PK(A)}\PK(M),$$
\item \(f\) is \emph{good} if \(A\leq M\) and \(f(A)\leq N\) are good and
\(f_{\RVhyb}\) is elementary --- for the \(\restr{\LRVhyb}{\RVhyb}\)-structure.
\end{enumerate}
\end{definition}

\begin{proposition}\label{emb sep P}
Assume \(f\) is a good embedding. Then \(f\) extends to a good embedding \(g : M \to N\).
\end{proposition}

\begin{proof} We proceed step by step.

{\bf Step~1.}  \emph{We may extend $f$ to a good map defined on $A\cup\RVhyb(M)$ --- and thus assume that $\RVhyb(A)=\RVhyb(M)$.}

Indeed, this follows from saturation, the fact that $f_{\RVhyb}$ is elementary and that the only symbols in the language involving both \(\K\) and \(\RVhyb\) are maps from \(\K\) to \(\RVhyb\).
\smallskip

{\bf Step~2.}  \emph{We may extend $f$ to a good map defined on (the substructure generated by) $A\cup\bP(M)$ --- and thus assume that $\PK(A)=\PK(M)$.}

Indeed, by $\K$-quantifier elimination in the $\LRV$-theory of the small valued field $\bP(M)$, the map $\restr{f}{\bP(A)}$ extends to an (elementary) $\LRV$-embedding $g:\bP(M)\rightarrow\bP(N)$. As $\K(A)\ind^{\vg\res}_{\PK(A)}\PK(M)$ and $\RVhyb(A)=\RVhyb(M)$, by \cref{F:characterize-vd}, $f\cup g$ induces a good embedding of $A\cup\bP(M)$ into $N$.

\smallskip

{\bf Step~3.}  \emph{We may extend $f$ to a good embedding of $M$ into $N$.}

Indeed, by $\K$-quantifier elimination in the $\LRVun$-theory of the valued field $M$, the map $f$ extends to an (elementary) $\LRVun$-embedding $\tilde{f}:M\rightarrow N$. By \cref{L:elem-ind}, we get $\tilde{f}(\K(M))\ind^{\vg\res}_{f(\PK(M))}\PK(N)$, so in particular $\tilde{f}(\K(M))\lind_{f(\PK(M))}\PK(N)$ (and thus $\tilde{f}(\K(M))\cap\PK(N)=f(\PK(M))$) by \cref{F:characterize-vd}(1),  showing that 
$\tilde{f}$ is an $\LPRVun$-embedding, with image a good substructure of $N$. Thus $\tilde{f}$ is a good embedding, since $f$ was already defined on the whole of $\RVhyb(M)$.
\end{proof}

\begin{corollary}
\label{RV-pair-red}
The theory $T^\star$ is complete relative to \(\RVhyb\), and $\RVhyb$ is purely
stably embedded in $T^\star$, \emph{i.e.}, the induced structure is that of a
hybrid \(\RV\)-structure.
\end{corollary}

This also holds for any $\RVhyb$-enrichment of the pair of valued fields. (This
is folklore. See, \emph{e.g.}, \cite[Proposition~2.7]{CoElHaJiRi22} for a
proof.)

\begin{proof}
Assume that $M,N\models T^\star$ are models with  \(\RVhyb(M)\equiv \RVhyb(N)\).
The isomorphism between the prime substructures, \emph{i.e.}, the substructures
of $M$ an $N$ generated by $\emptyset$, is easily seen to be a good embedding.
It follows by \cref{emb sep P}, and a back and forth argument, that it is, in
fact, elementary --- \emph{i.e.}, \(M\equiv N\).

Similarly, if \(M\subsel N\) --- in particular a good substructure --- and \(f
: M \to N\) is an elementary embedding --- in particular a good embedding ---
inducing the identity on \(\RVhyb(M)\), then it remains a good embedding --- and
hence an elementary one --- when extended by the identity on \(\RVhyb(N)\).
Thus, $\tp(M/\RVhyb(M))\vdash\tp(M/\RVhyb(N))$; in other words, \(\RVhyb\) is
stably embedded. Finally any \(\restr{\LRVhyb}{\RVhyb}\)-elementary map on
\(\RVhyb\) is good and hence \(\LRVhyb\)-elementary, so \(\RVhyb\) is pure.
\end{proof}


\begin{remark}
\label{PK-stably-embedded}
With a proof similar to the argument above, we also see that if \(\PRV\) is
(purely) stably embedded in \(\RVhyb\), then \(\PK\) is (purely) stably
embedded. Indeed, if \(f : M \to N\) is an elementary embedding which is the
identity on \(\PK\), if \(\PRV\) is stably embedded, we can extend it to a good
embedding by the identity on \(\PRV(N)\). Since \(\K(M) \ind^\ld_\PK(M)
\PK(N)\), we can then further extend this good embedding by the identity on
\(\PK(N)\). This extension can be seen to preserve \(\rv\) by using that
\(\PK(N) \leq \K(N)\) is separated. This proves that \(\PK\) is stably embedded.

Moreover, if \(\PRV\) is purely stably embedded in \(\RVhyb\), any automorphism
\(f : \PK(N) \to \PK(N)\) is good and hence \(\LRVhyb\)-elementary, proving that
\(\PK\) is a pure valued field.
\end{remark}

Combining our work so far, in particular,
\cref{RV-pair-red,compatible-section,PK-stably-embedded,PRV st emb}, we obtain:

\begin{corollary}\label{res-vg pure P}
Let $M\models T^\star$ be such that \(\bP(M)\) is finitely ramified with perfect residue field.
Assume that $\res^\times(M)$ is divisible (which is the case for example if
$M\models\ACVF$) or that $\Pvg(M)$ is a pure subgroup of $\vg(M)$. Then the
theory of \(M\) is determined by ramification data and the theories of \(\res\)
(with ramification constants) and \(\vg\). Moreover $\res$ and $\vg$ are purely
stably embedded and orthogonal and \(\RVhyb\) is purely stably embedded.

Furthermore, if \(\Pres\) is purely stably embedded in \(\res\) and \(\Pvg\) is
purely stably embedded in \(\vg\), then \(\PRV\) and \(\PK\) are also purely
stably embedded.
\end{corollary}

This remains true in any $\res$-$\vg$-enrichment of $M$ and with angular
components.

\begin{remark}\label{LPRVun}
If we further assume that \(\bP(M)\) eliminates
field quantifiers in \(\LRVun\) --- \emph{e.g.}, if it is algebraically closed
or algebraically maximal Kaplansky of equicharacteristic --- then all the above
results can easily be adapted to pairs of \(\LRVun\)-structures (with no need
for the rather exotic hybrid \(\RV\)-structures).
\end{remark}

\subsubsection{Characteristic 0 Laurent series fields}
Let $F$ be a  field of characteristic 0, and let $K:=F((t))$. In what follows, we are interested in the pair of valued fields $(\alg{K},K)$. Let us first  deal with the pair of value groups. Let \(\cL_\mathrm{og}\) be the language of ordered groups and DOAG be the theory of non-trivial divisible ordered abelian groups. Let also \(\cL_{\mathrm{Pres}}\) be the language \(\cL_{\mathrm{og}}\) enriched with a constant \(1\) and unary predicates for divisibility by integers. Let PRES be the \(\cL_{\mathrm{Pres}}\)-theory of \(\Zz\).

\begin{notation}
    Let $T_{\Qq,\Zz}$ be the theory of all structures $(\Gamma,\Delta)$ with
$\Gamma\models \mathrm{DOAG}$, $\Delta\models \mathrm{PRES}$ and such  for any $\gamma\in\Gamma$ there is a largest $\delta =: \lfloor\gamma\rfloor\in\Delta$ with $\delta\leq\gamma$, considered in the language $\cL_{\Qq,\Zz}$ given by $\cL_{\mathrm{og},\bP}$ together with $\cL_{\mathrm{Pres}}$ on the predicate $\bP$ and the function $\lfloor\cdot\rfloor$.
\end{notation}

The quantifier elimination result we state in part (3) of the following lemma has already been obtained by Weispfenning (\cite{Wei99}). (We thank Matthias Aschenbrenner for having brought this to our attention.) We decided to include our proof for convenience of the reader.

\begin{lemma}\label{lem:discrete-floor}\mbox{}
\begin{enumerate}
    \item Let $M=(\Gamma,\Delta)\models T_{\Qq,\Zz}$. Then the map $\gamma\mapsto (\lfloor\gamma\rfloor,\gamma-\lfloor\gamma\rfloor)$ is an $\emptyset$-definable bijection between $\Gamma$ and $\Delta\times[0,1)$, which identifies the $\emptyset$-definable sets in $M$ with the $\emptyset$-definable sets in the product structure $(\Delta,0,+,\leq)\times([0,1),0,\tilde{+},<)$, where $a\tilde{+}b:= a + b - \lfloor a+b\rfloor$ is the group law on $[0,1)$ induced by the natural bijection between $[0,1)$ and $\Gamma/\Delta$.
    \item In $T_{\Qq,\Zz}$, the predicate $\bP$ is stably embedded with induced structure a pure model of $\mathrm{PRES}$, and $[0,1)$ is stably embedded, with induced structure given by \(\cL_{\mathrm{og}}\), so in particular $o$-minimal.
    \item $T_{\Qq,\Zz}$ eliminates quantifiers and is complete.
    
\end{enumerate}
 
\end{lemma}

\begin{proof}
    Let $f:\Gamma\rightarrow \Delta\times[0,1)$ be the bijection given in (1). Clearly, $f$ is $\emptyset$-definable, and the product structure $(\Delta,0,+,<)\times([0,1),\tilde{+},<)$ is $\emptyset$-definable in $M$. Conversely, under this identification $\bP$ corresponds to $f^{-1}(0)$, $<$ on $\Gamma$ corresponds to the lexicographic ordering on $\Delta\times[0,1)$, and the addition on $\Gamma$ may also be recovered, since if $f(\gamma)=(z,a)$ and $f(\gamma')=(z',a')$, then 
    $$\label{eqn-carry}
        f(\gamma+\delta)=\begin{cases}
        (z+z',a\tilde{+}b) & \text{ if $a\leq a\tilde{+}b$}\\
        (z+z'+1,a\tilde{+}b) & \text{ otherwise}.
    \end{cases}
   $$

This proves (1). Part (2) follows directly from (1). 

Let us now show (3). Completeness follows from quantifier elimination, since $(\Zz,\Zz)$ embeds into every model of $T_{\Qq,\Zz}$ a a substructure. To prove quantifier elimination, we first note that the \(\cL_{\mathrm{og}}\)-theory of $([0,1),0,\tilde{+},<)$ has quantifier elimination. (This is well known, and we leave the easy proof to the reader.) Moreover, $\mathrm{PRES}$ has quantifier elimination in  $\cL_{\mathrm{Pres}}$. It is thus enough to establish the following claim:

\begin{claim}
    If $D \subseteq (\Delta\times[0,1))^n$ is defined by an atomic formula in the product structure $\Delta\times[0,1)$, with $\cL_{\mathrm{Pres}}$ on $\Delta$ and \(\cL_{\mathrm{og}}\) on $[0,1)$, then $f^{-1}(D)\subseteq M^n$ is defined by a quantifier free formula (without parameters).
\end{claim} 

To prove the claim, let us denote the projection onto $\Delta$ by $\pi_1$, that onto $[0,1)$ by $\pi_2$. If $\phi$ is of the form $\psi(\pi_1(x_1),\ldots,\pi_n(x_1))$ for some atomic $\cL_{\mathrm{Pres}}$-formula $\psi(\overline{y})$, the statement is clear, as then $f^{-1}(D)$ is defined by the quantifier free formula $\psi(\lfloor x_1\rfloor,\ldots,\lfloor x_n\rfloor)$. Else, $\phi$ is (equivalent to) a formula of the form $z_1\pi_2(x_1)\tilde{+}\cdots \tilde{+}z_n\pi_2(x_n)=0$, with $z_1,\ldots,z_n\in\Zz$, in which case $f^{-1}(D)$ is defined by $\bP(\sum_{i=1}^n z_ix_i)$; or $\phi$ is (equivalent to) a formula of the form $$z_1\pi_2(x_1)\tilde{+}\cdots \tilde{+}z_n\pi_2(x_n)<z'_1\pi_2(x_1)\tilde{+}\cdots \tilde{+}z'_n\pi_2(x_n),$$ with $z_1,z_1',\ldots,z_n,z_n'\in\Zz$, in which case $f^{-1}(D)$ is defined by the quantifier free formula $\sum_{i=1}^n z_ix_i-\lfloor \sum_{i=1}^n z_ix_i\rfloor<\sum_{i=1}^n z'_ix_i-\lfloor \sum_{i=1}^n z'_ix_i\rfloor$.
\end{proof}

In fact, the proof of \cref{lem:discrete-floor} yields the following more general result.

\begin{remark}
\label{rem:QZ-resplendent}
Let $\cL^+\supseteq\cL_{\mathrm{Pres}}$ and $T^+\supseteq \mathrm{PRES}$ be a
complete $\cL^+$-theory with quantifier elimination. Then the corresponding
expansion $T_{\Qq,\Zz}^+$ of $T_{\Qq,\Zz}$ is complete, eliminates quantifiers,
and $\bP$ is purely stably embedded with induced structure given by $\cL^+$.
\end{remark}

Since $T_{\Qq,\Zz}$ admits the complete model $(\Rr,\Zz)$, it is definably
complete. Actually, this also holds for definable complete expansions of
\(\mathrm{PRES}\), as shows the following corollary.

\begin{corollary}\label{Cvg pair}
Assume that the expansion $T^+\supseteq\rm{PRES}$ is definably complete. Then  \(T_{\Qq,\Zz}^+\) is definably complete.
\end{corollary}

\begin{proof}
Let $(\Gamma,\Delta)\models T_{\Qq,\Zz}^+$ and let $D\subseteq \Gamma$ be a definable subset which is bounded and non-empty. Then $\lfloor D \rfloor$ is a definable subset of $\Delta$ which is non-empty and bounded. 
By assumption, it admits a supremum $s$ in $\Delta$, which is then the maximum of $\lfloor D \rfloor$ as the order on $\Delta$ is discrete. 



As the induced structure on $[0,1)$ is $o$-minimal by \cref{lem:discrete-floor}, the induced structure on $[s,s+1]$ is $o$-minimal as well, and so $\sup (D)=\sup (D\cap[s,s+1])$ exists in $[s,s+1]$, proving definable completeness.
\end{proof}

Let us now consider the residue field. By a classical result of Keisler \cite{Kei64}, if $F$ and $F'$ are fields such that $F\equiv F'$, then $(\alg{F},F)\equiv(\alg{F'},F')$. If $F=\alg{F}$ or $\alg{F}$ is real closed, then the axiomatization, in $\cL_{ring,\bP}$, of $(\alg{F},F)$ is clear, and $\bP$ is stably embedded with induced structure that of a ring. In case $T_{f}$ is a complete theory of fields 
whose models are neither algebraically nor real closed, and $F\models T_f$, then the models of the $\cL_{ring,\bP}$-theory $T_{f,a}$ of $(\alg{F},F)$ are precisely the pairs $(M,\bP(M))$ of fields such that $M=\alg{M}$ and $\bP(M)\models T_f$. By \cite[Theorem~4.7]{HilKamRid}, if one  definably expands the theory, adding relation symbols $\ld_n$ and function symbols $\ell_{n,i}$, the theory $T_{f,a}$ eliminates quantifiers relative to $\bP$
. This yields in particular the following.

\begin{fact}
\label{F:stembrespair}
For $T_f=\Th(F)$ a complete theory of fields (in arbitrary characteristic), the
predicate $\bP$ is stably embedded in the $\cL_{ring,\bP}$-theory
$T_{f,a}=\Th(\alg{F},F)$, with induced structure given by $T_f$. 

This also holds for any $\cL$-expansion of $F$, where $\cL\supseteq\cL_{ring}$.
\end{fact}

The following lemma will be used in Section~\ref{gen} (proof of \cref{fin dens H}).

\begin{lemma}\label{UFres pair}
Let \(F\) be some (enriched) field which eliminates \(\exists^\infty\). Then the pair \((\alg{F},F)\) also eliminates \(\exists^\infty\).
\end{lemma}

\begin{proof}
We may suppose that $F$ is neither algebraically nor real closed, as otherwise the result is clear. By the relative quantifier elimination result \cite[Theorem~4.7]{HilKamRid} already mentioned above, if $(M,\bP(M))\equiv(\alg{F},F)$, and $(M,\bP(M))\subsel(\cU,\bP(\cU))$ then for any $a,b\in\cU\setminus\alg{(M\bP(\cU))}$ we have $tp(a/M)=tp(b/M)=:p_{gen}(x)$, so, assuming that $(\cU,\bP(\cU))$ is sufficiently saturated, any element of $\cU$ is the sum of two realizations of $p_{gen}$. By compactness, for any $M$-definable set $D=\psi(M)\subseteq M$ we then have 
\begin{equation}\label{generic-in-pair}
\begin{split}
     D+D=M&\Leftrightarrow \psi(\cU)+\psi(\cU)=\cU\\
     &\Leftrightarrow \psi(\cU)\not\subseteq\alg{(M\bP(\cU))}\Leftrightarrow p_{gen}(x)\vdash \psi(x).
\end{split}
\end{equation}

We will now assume that $M$ is $\aleph_0$-saturated. Let $\phi(x,y)$ be a formula with $x$ a single variable.  Assume that $c$ is a tuple from $M$ such that $\phi(M,c)$ is infinite. By compactness, it is enough to find a formula $\chi(y)\in\tp(c)$ such that for any $c'$ from $M$ satisfying $\chi$ the set $\phi(M,c')$ is infinite. 

If $p_{gen}(x) \vdash \phi(x,c)$, we may find such a $\chi(y)$ using \cref{generic-in-pair}. So assume now that $p_{gen}(x)\not \vdash \phi(x,c)$, and choose $a\not\in M$ realizing $\phi(x,c)$. Then $a\in\alg{(M\bP(\cU))}\setminus \alg{M}$, so in particular $M(a)\nlind_{\bP(M)}\bP(\cU)$. Set $A:=\dcl(Ma)$. If \(\bP(A)=\bP(M)\), by \cite[Lemma~4.1]{HilKamRid}, we have $A\lind_{\bP(M)}\bP(\cU)$, contradicting that $M(a)\nlind_{\bP(M)}\bP(\cU)$. It follows that that there is $a'\in \bP(A)\setminus\bP(M)$,  so we find an $M$-definable function $g:D\rightarrow\bP(M)$ with infinite image. Using \cref{F:stembrespair} and the assumption 
that the theory of $F$ eliminates \(\exists^\infty\), we may thus find a formula $\chi(y)$ as required, stipulating explicitly the existence of a definable function with infinite image in $\bP$.
\end{proof}

Fix some characteristic zero field \(F\) and let $T^\star_{\mathrm{Laur}}$ be
the theory of separated pairs of henselian valued fields \((M,\bP(M))\) with
$(\res(M),\Pres(M))\equiv(\alg{F},F)$ and $(\vg(M),\Pvg(M))\models T_{\Qq,\Zz}$.
Combining \cref{F:stembrespair,lem:discrete-floor} with \cref{res-vg pure P}, we obtain:

\begin{proposition}
\label{prop:Laurent-acl-pair}
The theory $T^\star_{\mathrm{Laur}}$ is complete, the definable sets $\res$,
$\vg$, $\RV$, $\PK$, $\Pres$, $\Pvg$ and $\PRV$ are all purely stably embedded,
and $\res$ and $\vg$ are orthogonal. All these results also hold for
\(\Pres\)-\(\Pvg\)-enrichments, and also if one adds an angular component.
\end{proposition}

\subsubsection{Finitely ramified fields}
We will now prove analogous statements in mixed characteristic. Let \(K\) be a
complete mixed characteristic \(\Zz\)-valued field with perfect residue field
\(F\). We are interested in the pair of valued fields $(\alg{K},K)$. Let
$T^\star_{\mathrm{Witt}}$ be a theory of separated pairs of henselian valued
fields with fixed ramification data such that
$(\res(M),\Pres(M))\equiv(\alg{F},F)$ (with ramification constants) and
$(\vg(M),\Pvg(M))\models T_{\Qq,\Zz}$. In that case, as a consequence of
\cref{res-vg pure P}, we obtain:

\begin{proposition}\label{prop:mixed-section-pair}
The theory $T^\star_{\mathrm{Witt}}$ is complete, the definable sets $\res$, $\vg$, $\RVhyb$, $\PK$, \(\PRV\), $\Pres$ and $\Pvg$ are all purely stably embedded, and $\res$ and $\vg$ are orthogonal. All these results also hold for \(\Pres\)-\(\Pvg\)-enrichments, and also if one adds a coherent system of normalized angular components.
\end{proposition}

\subsubsection{Divisible value group}
Let $F$ be a  field. If $\rm{char}(F)=p>0$, assume that $F$ does not admit a finite extension of degree divisible by $p$  (in particular $F$ is perfect). Let $K:=F((t^{\Qq}))$. In what follows, we are interested in the pair of valued fields $(\alg{K},K)$.

Let $T^\star_{\mathrm{div}}$ be the theory of separated pairs of equicharacteristic algebraically maximal valued fields $(M,\bP(M))$ such that \((\res(M),\Pres(M)) \equiv (\alg{F},F)\) and $\Pvg(M)=\vg(M)\models\mathrm{DOAG}$. Note that the Kaplansky conditions (\emph{cf.} \cref{benign}) are satisfied in this case. Once again, as a consequence of \cref{res-vg pure P}, we have:

\begin{proposition}\label{prop:divisible-pair}
The theory $T^\star_{\mathrm{div}}$ is complete. The definable sets $\res$, $\vg$, $\RV$, \(\PK\), $\Pres$ and \(\PRV\) are all purely stably embedded, and $\res$ and $\vg$ are orthogonal. All these results also hold for \(\Pres\)-\(\vg\)-enrichments, and also if one adds angular components.
\end{proposition}

\subsection{Valued difference fields}\label{val diff}

Let \((K,v,\sigma)\) be an equicharacteristic zero valued field with an automorphism.

\begin{definition}
\begin{enumerate}
\item For every \(P\in K[x_0,\ldots,x_n]\), \(a\in K\), \(d\in K^{n+1}\) and \(\gamma\in \vg(K)^\times\), we say that \((P,a,d,\gamma)\) is in \emph{\(\sigma\)-Hensel configuration} if
\[\val(P(\prol a)) > \min_i \val(d_i) + \sigma^i(\gamma)\]
 and, for all \(x,y\in K\) with \(v(x-a),v(y-a)>\gamma\),
\[\val(P(\prol y) - P(\prol x) - d\cdot \prol (y- x)) > \min_i \val(d_i \sigma^i(y-x)).\]
Here, \(\prol a := (\sigma^i(a))_{0\leq i \leq n}\).
\item We say that \((K,v,\sigma)\) is \emph{\(\sigma\)-henselian} if for every \((P,a,d,\gamma)\) in \(\sigma\)-Hensel configuration, there exists \(c \in K(M)\) such that \(P(\prol c) = 0\) and \[\val(c - a)\geq \max_{i,d_i\neq0} \val(\sigma^{-i}(P(\prol a)) d_i^{-1}).\]
\item A difference field $(k,\sigma)$ is called \emph{linearly closed} if for every linear non constant \(L\in k[x_0,\ldots,x_n]\) and \(c\in k\), there exists \(a\in k\) such that \(L(\nabla(a)) = c\)
\end{enumerate}
\end{definition}

\begin{fact}
Assume that \(\res(K)\) is linearly closed and either:
\begin{itemize}
\item \((K,v)\) is maximally complete;
\item \((K,v)\) is complete and rank one, i.e., the value group is archimedean.
\end{itemize}
Then \((K,v,\sigma)\) is \(\sigma\)-henselian.
\end{fact}

This follows from Newton approximation, \emph{cf.}\ \cite[Proposition\,4.14]{Rid-ADF}. Note that at each step the approximation to the root of a \(\sigma\)-Hensel configuration \((P,a,d,\gamma)\) improves by at least \(\gamma\), \emph{cf.}\ \cite[Lemma\,4.16]{Rid-ADF}, and hence, in rank one, completeness suffices.

Let \(\LRVsig\) be the language \(\LRV\) with three new unary functions \(\sigma_\K : \K \to \K\), \(\sigma_\RV : \RV \to \RV\) and \(\sigma_\vg:\vg\rightarrow\vg\). The expected quantifier elimination result also holds in characteristic zero \(\sigma\)-henselian fields, by \cite[Theorem\,7.3]{DurOna} (see also \cite[p. 41, Theorem\,A]{Rid-ADF}):

\begin{fact}\label{EQ sHen}
Let \(\cL\) be an \(\RV\)-enrichment of \(\LRVsig\) and \(T\) an \(\cL\)-theory containing the theory \(\sHen[0,0]\) of equicharacteristic zero \(\sigma\)-henselian valued fields. Then \(T\) eliminates field quantifiers.
\end{fact}

\begin{remark}\label{max complete}
In \cite{Azg-OmeInc,Pal-Mul,DurOna} an \emph{a priori} weaker notion of
\(\sigma\)-henselian fields is considered. However, both notions hold in
maximally complete fields (with linearly closed residue field) and both allow
proving \cref{EQ sHen}. In equicharacteristic zero, the automorphism extends to
any maximal completion. Moreover, that maximal completion has the same \(\RV\)
as the original field. Thus, it follows from \cref{EQ sHen} that any
equicharacteristic zero field satisfying either notion is elementarily
equivalent to any maximal completion where both notions hold. So these two
notions of \(\sigma\)-henselianity coincide.
\end{remark}

It follows from field quantifier elimination that \(\RV\) is purely stably
embedded. Its induced structure is the expansion of the short exact sequence of
\(\Zz[\sigma]\)-modules \(1\to \res^\times  \to \RV^\times \to \vg^\times\to 0\)
by \(\oplus\) and \(<\) which, by \cref{define-oplus}, can be defined (with
parameters) in a $\vg$-$\res$-enrichment.

In order to obtain model complete theories, one often restricts the behaviour of the automorphism on the value group, \emph{e.g.}, the class of (existentially closed) multiplicative difference valued fields introduced in \cite{Pal-Mul}:

\begin{definition}\label{VFA}
Let \(\VFA\) be the theory of \(\sigma\)-henselian non-trivially valued fields such that:
\begin{enumerate}
\item For every \(P\in\Zz[\sigma]\), either \(P(\vg_{>0}) = \vg_{>0}\), \(P(\vg_{<0}) = \vg_{>0}\) or \(P(\vg) = 0\);
\item We have \((\res,\sigma_\res) \models\ACFA_0\)
\item The embedding of \(\Zz[\sigma]\)-modules \(\res^\times \to \RV\) is pure.
\end{enumerate}
\end{definition}

\begin{remark}
Two multiplicative behaviours of \(\sigma\) are of particular interest.
\begin{enumerate}
\item The \(\omega\)-increasing case --- \emph{i.e.}, for all \(x\in\Val\) and \(n\in\Zz_{>0}\), \(\val(\sigma(x)) \geq n\val(x)\) --- studied in \cite{Azg-OmeInc,DorHru}. One then gets the asymptotic theory of \((\alg{\Ff_p(t)},\val_t,\Frob_p)\), where \(\Frob_p\) is the Frobenius automorphism.
\item The isometric case, studied in \cite{BelMacSca}. In that case, one gets the asymptotic theory of \((\Cc_p,\val_p,\sigma_p)\), where \(\sigma_p\) is an isometric lift of the Frobenius automorphism on \(\res(\Cc_p) = \alg{\Ff_p}\).
\end{enumerate}

Both characterisations follow (see, \emph{e.g.},  \cite{ChHi14}) from the Ax-Kochen-Ershov principle for \(\sigma\)-henselian valued fields and Hrushovski's deep result that \(\ACFA_0\) is the asymptotic theory of \((\alg{\Ff_p},\Frob_p)\), \emph{cf.} \cite{Hru-Frob}. 
\end{remark}

\begin{fact}\label{EQ VFA}
In \(\VFA\),  \(\res\) is a stably embedded pure difference field and \(\vg\) is a stably embedded \(o\)-minimal pure ordered \(\Zz[\sigma]\)-module, and \(\res\) and \(\vg\) are orthogonal. These results also hold if one adds a \(\sigma\)-equivariant angular component.
\end{fact}

\begin{proof}
Condition (3) in \cref{VFA} ensures that in any $\aleph_1$-saturated model of  \(\VFA\) there is a $\sigma$-equivariant angular component map $\ac$. \cite[Theorem~11.8]{Pal-Mul} yields the results if we add such an $\ac$ map to the language, and they obviously go down to the reduct without $\ac$.
\end{proof}


\subsection{Linear structures}

Let us now recall the results of \cite{Hru-GpIm} on linear structures. In the case of valued difference fields, we will need "twisted" versions of these results. As it took us a while to get the arguments clear, we decided to spell them out in detail.

\subsubsection{Independent amalgamation}
We will first recall some material from  \cite[Section\,4]{Hru-GpIm}. We fix a complete stable theory $T$ in some language $\cL$, and we assume that $T$ eliminates quantifiers and imaginaries. Let $\cC_T$ be the category consisting of those $\cL$-structures that are  algebraically closed substructures of a model of $T$, with $\cL$-embeddings (which are $\cL$-elementary in $T$ by assumption) as morphisms. Let $\mathbf{n}=\{0,1,\ldots,n-1\}$, and set $\cP(\mathbf{n})^-:=\cP(\mathbf{n})\setminus\{\mathbf{n}\}$. We consider $\cP(\mathbf{n})^-$ and $\cP(\mathbf{n})$ as categories, with inclusion maps as morphisms. 

Given a functor $A:P\rightarrow\cC_T$, where  $P$ equals $\cP(\mathbf{n})^-$ or $\cP(\mathbf{n})$, if $\iota:w_1\rightarrow w_2$ denotes the inclusion map for two sets $w_1\subseteq w_2\in P$, in what follows we will write $A(w_1)$ for the subset $A(\iota)(A(w_1))$ of $A(w_2)$, thus omitting the map $A(\iota)$ in our notation. This slight abuse of notation should not lead to any confusion.

\begin{definition}
Let $P$ equal $\cP(\mathbf{n})^-$ or $\cP(\mathbf{n})$.
\begin{enumerate}
\item A functor $A:P\rightarrow\cC_T$ is called \emph{independence preserving} if for any $w,w'\in P$ with $w\cup w'\in P$ one has 
$A(w)\ind_{A(w\cap w')}A(w')$ (inside $A(w\cup w')$).
\item A functor $A:P\rightarrow\cC_T$ is called \emph{bounded} if for any $\emptyset\neq w\in P$ one has $A(w)=\acl(\bigcup_{i\in w}A(\{i\}))$.
\item An \emph{$n$-amalgamation problem} in $T$ is a bounded independence preserving functor $A^-:\cP(\mathbf{n})^-\rightarrow\cC_T$. A \emph{solution} of $A^-$ is a bounded independence preserving functor $A:\cP(\mathbf{n})\rightarrow\cC_T$ extending $A^-$.
\end{enumerate}
\end{definition}

\begin{definition}
The theory $T$ is said to have  
\begin{itemize}
\item \emph{$n$-existence} if every $n$-amalgamation problem in $T$ has a solution;
\item \emph{$n$-uniqueness} if whenever $A$ and $A'$ are solutions of the same $n$-amalgamation problem $A^-$ in $T$, then $A$ and $A'$ are isomorphic over $A^-$, \emph{i.e.}, there is an $\cL$-isomorphism $f:A(\mathbf{n})\cong A'(\mathbf{n})$ fixing $A^-(w)$ pointwise for every $w\in\cP(\mathbf{n})^-$.
\end{itemize}
\end{definition}

\begin{remark}
In the terminology of \cite{Hru-GpIm}, these notions corresponds to $n$-existence / $n$-uniqueness of $T$ over every parameter set.
\end{remark}

\begin{remark}\label{Stable-Amal}
It follows from stability and elimination of imaginaries in $T$ (as then types over algebraically closed sets are stationary) that $T$ has 2-existence, 2 uniqueness and 3-existence.
\end{remark}

Let $\sigma$ be a new unary function symbol, $\cL_\sigma:=\cL\cup\{\sigma\}$. Consider the category $\widetilde{\cC}_T$ of $\cL_\sigma$-structures of the form $(A,\sigma)$, where $A\in\cC_T$ and $\sigma\in\aut[\cL](A)$, with $\cL_\sigma$-embeddings as morphisms.

\begin{definition}\label{D:AMAL-sigma}
Let $P$ equal $\cP(\mathbf{n})^-$ or $\cP(\mathbf{n})$.
\begin{enumerate}
\item A functor $A:P\rightarrow\widetilde{\cC}_T$ is called \emph{independence preserving} (\emph{bounded}, respectively), if it is so when composed with the forgetful functor from $\widetilde{\cC}_T$ to $\cC_T$.
\item We say that $\widetilde{\cC}_T$ has \emph{$n$-existence} if every bounded independence preserving functor $A^-:\cP(\mathbf{n})^-\rightarrow\widetilde{\cC}_T$ extends to a bounded independence preserving functor $A:\cP(\mathbf{n})\rightarrow\widetilde{\cC}_T$.
\end{enumerate}
\end{definition} 

It follows from 2-uniqueness and 2-existence in $T$ that $\widetilde{\cC}_T$ has 2-existence.

Let $T_\sigma$ be the $\cL_\sigma$-theory of all $(M,\sigma)\in\widetilde{\cC}_T$ such that $M\models T$. Recall that if $T_\sigma$ admits a model-companion, it is denoted by $T\!A$. If this is the case, we will say that "$T\!A$ exists". (We refer to \cite{ChPi98} for fundamental facts about  $T\!A$.)

\begin{fact}[\mbox{\cite[Proposition~4.7 and Corollary~4.10]{Hru-GpIm}}]
For $T$ as above, the following are equivalent:

\smallskip
\begin{enumerate}
\item[(1)] $T$ has 3-uniqueness.
\item[(2)] $\widetilde{\cC}_T$ has 3-existence.
\end{enumerate}

\smallskip
Moreover, assuming in addition that $T\!A$ exists, the above conditions imply:

\smallskip
\begin{enumerate}
\item[(3)] $T\!A$ eliminates imaginaries.\qed
\end{enumerate}
\end{fact}

It is easy to see that if $T\!A$ exists and it eliminates bounded hyperimaginaries (\emph{e.g.}, when $T$ is superstable, by \cite{Supersimple} combined with \cite[Corollary~3.8]{ChPi98}), then (3) is actually equivalent to (1) and (2). We will not use this in our paper.

\subsubsection{Twisted independent amalgamation}
Let $\tau:\cL\cong\cL'$ be a bijection between two first order languages (sending sorts to sorts, function symbols to function symbols consistently with their arity, similarly for constants and relations). Then $\tau$ extends naturally to a bijection between the set of $\cL$-formulas and the set of $\cL'$-formulas. Given an $\cL$-formula $\phi$, we denote by $\phi^\tau$ its image under this map. If $T$ is an $\cL$-theory, $T^\tau:=\{\phi^\tau\mid\phi\in T\}$ is an $\cL'$-theory. Of course, up to changing the names of the symbols using $\tau$, $T^\tau$ is the "same" theory as $T$.

If $M$ is an $\cL$-structure, we denote by $M^\tau$ the $\cL'$-structure with base set $M$ and interpretations $\left(\Sigma^\tau\right)^{M^\tau}=\Sigma^M$, for any symbol $\Sigma\in\cL$. If $N'$ is an $\cL'$-structure, we call an $\cL'$-isomorphism $\sigma:M^\tau\cong N'$ a \emph{$\tau$-twisted isomorphism} between $M$ and $N'$. Similarly, one defines the notion of a \emph{$\tau$-twisted elementary map} $\sigma:A\rightarrow A'$, where $A\subseteq M$ and $A'\subseteq N'$, \emph{i.e.}, one requires that for any $\cL$-formula $\phi(x)$ and any tuple $a$ from $A$ of the right length, one has $M\models\phi(a)$ if and only if $N'\models\phi^\tau(\sigma(a))$.

\begin{lemma}\label{L:twisted-3-unique}
Let $T$ be a complete stable $\cL$-theory eliminating quantifiers and imaginaries. Assume that $T$ has $n$-uniqueness. Let $A:\cP(\bf{n})\rightarrow\cC_T$ and $A':\cP(\bf{n})\rightarrow\cC_{T^\tau}$ be bounded independence preserving functors. 

Then for any coherent system  $(\sigma_w)_{w\in\cP(\bf{n})^-}$ of $\tau$-twisted elementary bijections $\sigma_w:A(w)\rightarrow A'(w)$ there exists a $\tau$-twisted elementary bijection $\sigma_{\bf{n}}:A({\bf{n}})\rightarrow A'(\bf{n})$ extending $\sigma_w$ for every $w\in\cP(\bf{n})^-$.
\end{lemma}

\begin{proof}
The result follows from $n$-uniqueness of $T^\tau$, since we may consider $A$ as a functor to $\cC_{T^\tau}$, replacing $A(w)$ by $A(w)^\tau$.
\end{proof}

We now consider the special case where $\cL'=\cL$, $\tau$ is a permutation of $\cL$ and  $T$ is a complete $\cL$-theory such that $T=T^\tau$. Let $\widetilde{C}^{(\tau)}_T$ be the category of $\cL_\sigma$-structures 
$(B,\sigma)$ with $B\in\cC_T$ and $\sigma:B\rightarrow B$ a $\tau$-twisted elementary bijection. When $T$ is stable, we use the same terminology as in \cref{D:AMAL-sigma}, for functors $A:P\rightarrow\widetilde{\cC}^{(\tau)}_T$. \Cref{L:twisted-3-unique} then yields:

\begin{corollary}\label{C:per-3-unique}
Let $T$ be a complete stable $\cL$-theory eliminating quantifiers and imaginaries, and let $\tau:\cL\rightarrow\cL$ be a bijection such that $T^\tau=T$. Then $\widetilde{\cC}^{(\tau)}_T$ has 2-existence. If in addition we assume that $T$ has 3-uniqueness, then $\widetilde{\cC}^{(\tau)}_T$ has 3-existence.
\end{corollary}

\begin{proof}
\cref{L:twisted-3-unique} yields that if $T$ has $n$-uniqueness and $n$-existence, then $\widetilde{\cC}^{(\tau)}_T$ has $n$-existence. We may thus conclude by \cref{Stable-Amal}
\end{proof}

Given a complete $\cL$-theory $T$ and a permutation $\tau$ of $\cL$ such that $T^\tau=T$, we let $T^{(\tau)}_\sigma$ be the $\cL_\sigma$-theory whose models are of the form $(M,\sigma)$, where $M\models T$ and where $\sigma$ is a $\tau$-twisted automorphism of $M$.

Assume now in addition that $T$ is stable and eliminates quantifiers and imaginaries. It follows from quantifier elimination in $T$ that  $T^{(\tau)}_\sigma$ is then a $\forall\exists$-theory, and so it has a model-companion  if and only if the e.c.\ models of $T^{(\tau)}_\sigma$ form an elementary class. If this is the case, denote by $T^{(\tau)}\!A$ the model-companion of $T^{(\tau)}_\sigma$. Then the models of $T^{(\tau)}\!A$ are precisely the e.c.\ models of of $T^{(\tau)}_\sigma$. The basic results on $T\!A$, due to Chatzidakis and Pillay (\cite{ChPi98}), generalize to this context in a straightforward manner. We will only state some facts which we will need.

\begin{lemma}\label{F:ChPi-TA}
Let $T$ and $\tau$ be as above, and assume that $T^{(\tau)}\!A$ exists. Then the following holds:
\begin{enumerate}
\item If $(M,\sigma)\models T^{(\tau)}\!A$ and $B\subseteq M$ then 
$$\acl_{(M,\sigma)}(B)=\acl_\sigma(B):=\acl_M(\bigcup_{z\in\Zz}\sigma^z(B)).$$
\item Quantifier reduction: If $(M_i,\sigma_i)\models T^{(\tau)}\!A$ and $B_i\subseteq M_i$ for $i=1,2$, then $B_1\equiv_{\cL_\sigma} B_2$ if and only if there is an $\cL_\sigma$-isomorphism from $\acl_\sigma(B_1)$ to $\acl_\sigma(B_2)$ sending $B_1$ to $B_2$.
\item $T^{(\tau)}\!A$ is simple and 
$$A\ind^{T^{(\tau)}\!A}_E B\text{ if and only if }\acl_\sigma(EA)\ind^{T}_{\acl_\sigma(E)}\acl_\sigma(EB).$$
If $T$ is superstable, then $T^{(\tau)}\!A$ is supersimple.
\item Assume that $\widetilde{\cC}^{(\tau)}$ has 3-existence. (Equivalently, in $T^{(\tau)}\!A$, the independence theorem holds over  $\acl_\sigma$-closed sets.) Then $T^{(\tau)}\!A$ eliminates imaginaries.
\end{enumerate}
\end{lemma}

\begin{proof}
To prove (2), assume that $B=\acl_\sigma(B)$ is a common substructure of two models $(M,\sigma)$ and $(N,\sigma)$ of $T^{(\tau)}\!A$, 
such that $(N,\sigma)$ is $|M|^+$-saturated. We need to show that $(M,\sigma)$ $\cL_\sigma(B)$-embeds into $(N,\sigma)$. As $\widetilde{\cC}^{(\tau)}_T$ has 2-existence, there is an amalgam $(A,\sigma)\in\widetilde{\cC}^{(\tau)}_T$ of $N$ and $M$ over $B$. Enlarging $(A,\sigma)$ if necessary, we may assume that $(A,\sigma)\models T^{(\tau)}\!A$, hence $(A,\sigma)\supsel (N,\sigma)$. In particular, $\tp_{\cL_\sigma}(M/B)$ is finitely satisfiable in $(N,\sigma)$, so this type is realized in $(N,\sigma)$ by saturation, yielding  an $\cL_\sigma(B)$-embedding of $(M,\sigma)$  into $(N,\sigma)$. 

We now prove (1). Let $(M,\sigma)\models T^{(\tau)}\!A$ and let $B\subseteq M$. Clearly, $\acl_{(M,\sigma)}(B)\supseteq\acl_\sigma(B)$. To prove the other inclusion, it suffices to show that if $B=acl_\sigma(B)$, then $B$ is algebraically closed in $(M,\sigma)$. Let $a\in M\setminus B$, and set $A:=\acl_\sigma(Ba)$. Then $(B,\sigma)\subseteq(A,\sigma)$ is an extension in $\widetilde{\cC}^{(\tau)}_T$. 

For $n\in\Zz_{>0}$, using 2-existence in $\widetilde{\cC}^{(\tau)}_T$ and induction, we may construct an extension $ (B,\sigma)\subseteq (C_n,\sigma)\in \widetilde{\cC}^{(\tau)}_T$ such that $C_n$ contains $n$ isomorphic copies $(A_1,\sigma),\ldots,(A_n,\sigma)$ of $(A,\sigma)$ over $B$, which are $\cL$-independent over $B$. Replacing $(M,\sigma)$ by an elementary extension if necessary, using (2) we may assume that $C_n\subseteq M$. Since $B=\acl_{\cL}(B)\cL$, $A_i\cap A_j=B$ for any $i\neq j$, so $(M,\sigma)$ contains $n$ distinct realizations of $\tp_{\cL_\sigma}(a/B)$, by part (2). As $n$ was arbitrary, $a\not\in\acl_{(M,\sigma)}(B)$.

To show (3), we proceed exactly as in the proof of the corresponding result for $T\!A$ (\cite[Corollary~3.8]{ChPi98}). If $A,B,E$ are subsets of a model of $T^{(\tau)}\!A$, we say that $A$ and $B$ are \emph{independent} over $E$ if $\acl_\sigma(EA)\ind^{T}_{\acl_\sigma(E)}\acl_\sigma(EB)$, where $\ind^T$ denotes forking independence in $T$. This relation satisfies the all abstract properties of an independence notion that guarantee, by the Theorem of Kim-Pillay (see, \emph{e.g.}, \cite[Theorem~2.6.1]{Wag-Simple}), that $T^{(\tau)}\!A$ is simple and that non-foking is given by the independence notion in question. This is clear for all properties except the independence theorem (over a model). To establish the latter, one shows that every 3-amalgamation problem $A^-:\cP(\mathbf{3})^-\rightarrow\widetilde{\cC}^{(\tau)}_T$ with $A^-(\emptyset)\models T^{(\tau)}\!A
$ 
has a solution; equivalently, the independence theorem even holds over models of $T^{(\tau)}_\sigma$. The proof of this is identical to the proof of \cite[Theorem~3.7]{ChPi98}. 

The statement about supersimplicity follows as in the proof of \cite[Corollary~3.8]{ChPi98}.

Part (4) is the analog of \cite[Proposition~4.7]{Hru-GpIm}. Weak elimination of imaginaries in $T^{(\tau)}\!A$ follows directly from a formalization of Hrushovski's argument by Montenegro and the second author (see \cite[Proposition~1.17]{MonRid-PpC}). But finite sets are coded in models of $T$, as $T$ eliminates imaginaries by assumption. so they are also coded in the expansion $T^{(\tau)}\!A$ of $T$. Hence $T^{(\tau)}\!A$ eliminates imaginaries.
\end{proof}

\subsubsection{Linear imaginaries} Let us now recall some notions from \cite[Section\,5]{Hru-GpIm}.

\begin{definition}
Let $\mathfrak{t}$ be a theory of fields (possibly with additional structure).  A \emph{$\mathfrak{t}$-linear structure} $M$ is an \(\cL\)-structure with a sort  \(\res\) for a model of $\mathfrak{t}$, and additional sorts $V_i$ ($i\in I$)
denoting finite-dimensional  $\res$-vector spaces, such that the family $\left(V_i\right)_{i\in I}$ is closed under tensor products and duals. Each $V_i$ has (at least) the $\res$-vector space structure. One assumes that $\res$ is stably embedded in $M$ with induced structure given by $\mathfrak{t}$. 

We now fix such a $\mathfrak{t}$-linear structure $M$.

\begin{enumerate}
\item $M$ is said to \emph{have flags} if for any $i$ with $\dim(V_i) > 1$, for some $j, k$ with $\dim(V_j) = \dim(V_i) - 1$, there exists a $\emptyset$-definable exact sequence $0 \longrightarrow V_k  \longrightarrow V_i \longrightarrow V_j \longrightarrow 0$. We will call such a short exact sequence a \emph{flag}.

\item $M$ is said to \emph{have roots} if for any one-dimensional $V = V_i$, and any $m \geq 2$, there
exists a (one-dimensional) $W = V_j$ and a $\emptyset$-definable $\res$-linear isomorphism $f: W^{\otimes m} \cong V$.
\end{enumerate}
\end{definition}

Let us now mention two results from \cite{Hru-GpIm}. The proof of the first one is rather elementary, whereas that of the second one is quite involved.

\begin{fact}[{\cite[Lemma~5.6]{Hru-GpIm}}]\label{F:lin1}
The theory of an $\ACF$-linear structure with flags (in any characteristic) eliminates imaginaries.
\end{fact}

The following fact follows from \cite[Proposition~5.7]{Hru-GpIm} in combination with  \cite[Proposition~4.3 and Corollary~4.10]{Hru-GpIm}.

\begin{fact}\label{F:lin2}
Let $T$ be the theory of an $\ACF_0$-linear structure with flags and roots. Then $T$ has 3-uniqueness.
\end{fact}

Our main interest in linear structures stems from the fact that the
\(\res\)-internal sets in a given model of \(\ACVF\) give rise to such a
structure. For every \(M\models \Hen[0]\) and \(A \subseteq \Geom(M)\), we
define \[\Lin_A := \bigsqcup_{\substack{s\in\Lat(\dcl_0(A))\\ \ell\in\Zz_{>0}}}
s/\ell\Mid s.\]

In equicharacteristic zero, i.e., if \(M\models\Hen[0,0]\), this corresponds
exactly to the collection of vector spaces \(\mathrm{VS_{\res,A}} =
\bigsqcup_{s\in\Lat(\dcl_0(A))} s/\Mid s\) introduced in \cite{HasHruMac-Book}.
In mixed characteristic, however, this is a more complicated structure since it
also consists of (free) \(\Res_\ell\)-modules --- and this more complicated
structure is actually needed in \cref{invariant}. Note that, by \cref{A-points}
and our choice of representation of the geometric sorts, \(\Lin_A(M)\) is the
set of cosets \(c + \ell\Mid s\) where \(s\in\Lat(\dcl_0(A))\) has a basis in
\(M\) and \(c\in s(M)\). 

\begin{lemma}\label{F:lin-St}
Let \(M\models\Hen[0,0]\) and \(A\subseteq\Geom(M)\). Then \(\Lin_A(M)\), with its \(\cL_0(A)\)-induced structure, is an $\Th(\res(M))$-linear structure with flags. Moreover, if \(\vg(M)\) is divisible, then  \(\Lin_A(M)\) has roots.
\end{lemma}


\begin{proof}
We may assume that \(\dcl_0(A)\cap\Geom(M)\subseteq A\). The fact that the residue field $\res$ is stably embedded in \(\Hen[0,0]\), with induced structure that of a pure field, is well known --- and follows from the existence of splittings as in \cref{compatible-section}.

Now let $V,W$ be two sorts from $\Lin_A(M)$, \emph{i.e.}, vector spaces over $\res$ of the form $V=a/\Mid a$, $W=b/\Mid b$ for some $a\in \Lat_m(A)$ and $b\in \Lat_n(A)$ --- with bases in \(M\).
Then $a\otimes_{\Val} b$ is canonically isomorphic to an element $c$ from $\Lat_{m\cdot n}(A)$, so we may identify $V\otimes_{\res}W$ with 
$c/\Mid c$, which is a sort from $\Lin_A(M)$. Similarly, $\dual{a}=\Hom[\Val](a,\Val)$ can be identified with $\{z\in \K^n\mid \forall v\in a,\,\sum z_iv_i\in\Val\}\in \Lat_m(A)$, so $\dual{V}=\Hom[\res](V,\res) \isom \dual{a}/\Mid \dual{a}$ is a sort from $\Lin_A$ as well. 

\smallskip

\emph{\textbf{Flags:}} For $a\in \Lat_n(A)$ define $a_1:=a\cap (\K\times \{0\}^{n-1})$. Then the projection onto 
the first coordinate identifies $a_1$ with an element of $\Lat_1(A)$. Let $\pi:a\rightarrow \K^{n-1}$ be induced from the projection on the last $n-1$ coordinates. Then 
\[0\rightarrow\ker(\pi)=a_1\rightarrow a\rightarrow\pi(a)\rightarrow 0\] is an $A$-definable 
exact sequence of free $\Val$-modules, and $\pi(a)\in \Lat_{n-1}(A)$ --- this follows from the fact 
that $\pi(a)$ is a finitely generated torsion free $\Val$-submodule of $\K^{n-1}$ of rank $n-1$.
Reducing modulo $\Mid$, we conclude.

\smallskip

\emph{\textbf{Roots:}} Assume \(\vg(M)\) is divisible. Let $n\geq 1$, and let $V$ be a one-dimensional sort from $\Lin_A$. 
Then $V=\gamma\Val/\gamma\Mid$ for some $\gamma\in\Gamma(\dcl_0(A))$. 
Consider $V_n:=\delta\Val/\delta\Mid$, for $\delta=\frac{\gamma}{n}$. The map 
\[x/\gamma\Mid\mapsto y/\delta\Mid \otimes\ldots\otimes y/\delta\Mid : V\rightarrow V_n^{\otimes n},\]
where $y^n=x$, is well-defined and an $A$-definable isomorphism of $\res$-vector spaces defined over $A$. In particular, $\Lin_A$ has roots.
\end{proof}

The result above actually holds of the \emph{stable part} \(\bigsqcup_{s\in\Lat(\dcl_0(A))} s/\Mid s\) in all characteristic (provided \(\res\) is stably embedded).

\begin{corollary}\label{cor:StA-3unique}
Let \(M\models\ACVF_{0,0}\) and \(A\subseteq\Geom(M)\). Then $\Lin_A$ satisfies 3-uniqueness.\qed
\end{corollary}

\subsubsection{Twisted linear imaginaries}
\begin{lemma}\label{L:ec-twist}
Let $\mathfrak{t}$ be a stable theory of fields, and let $T$ be the theory of a $\mathfrak{t}$-linear structure such that $T$ eliminates quantifiers. Let $\tau$ be a permutation of the language with $T=T^\tau$ such that $\tau$ fixes all the formulas on the sort $\res$. Suppose $\mathfrak{t}A$ exists. Then $T^{(\tau)}_\sigma\cup\mathfrak{t}A$ is the model-companion of $T^{(\tau)}_\sigma$. In particular, this holds for $\mathfrak{t} = \ACF$.
\end{lemma}

\begin{proof}
Let $(M,\sigma)\models T^{(\tau)}_\sigma$. Then, as $M\models T$, for any sort $V$ from $\cL$ there is an $M$-definable surjection 
$f:\res(M)\rightarrow V(M)$. For any \(N \supsel_{\cL} M\), $f$ then also defines a surjection from $\res(N)$ onto $V(N)$, hence $N=\dcl_{\cL}(M\res(N))$. Thus, any extension of \(\sigma\) to a $\tau$-twisted automorphism on \(N\) is uniquely determined by its restriction to \(\res(N)\). It follows that $(M,\sigma)$ is an e.c.\ model of $T^{(\tau)}_\sigma$ if and only if $(\res(M),\restr{\sigma}{\res(M)})$ is an e.c.\ model of $\mathfrak{t}_\sigma$. This yields the statement of the lemma.
\end{proof}

As a special case of \cref{L:ec-twist}, we get the following.

\begin{remark}
Let $\mathfrak{t}$ be a stable theory of fields, and let $T$ be the theory of a $\mathfrak{t}$-linear structure, such that $T$ eliminates quantifiers. Suppose $\mathfrak{t}A$ exists. Then $T\!A$ exists and is given by $T_\sigma\cup\mathfrak{t}A$. In particular, this holds for $\mathfrak{t} = \ACF$.
\end{remark}

\begin{definition}
Let $\res$ be a stably embedded sort in a theory $T$. An $A_0$-definable set $D$ is said to be \emph{internally $\res$-internal} 
 (over $A_0$), if there is a tuple $d\in D$ and an $A_0d$-definable surjection $f:Y\rightarrow D$, where $Y\subseteq \res^n$ for some $n$.
\end{definition}

\begin{lemma}\label{L:InternallyInternal}
Let $\res$ be a stably embedded sort in a theory $T$, and let $D$ be $A_0$-definable and internally $\res$-internal (over $A_0$). Then $\res\cup D$ is stably embedded (over $A_0$).
\end{lemma}

\begin{proof} Set $D':=\res\cup D$. It follows from the assumptions that $D'$ is $A_0$-definable and internally $\res$-internal over $A_0$. Let $f$ be an $A_0d$-definable surjection as in the definition, with $d\in D'$. Taking the preimage under $f\times\cdots\times f$, one sees that  any $\cU$-definable subset $X$ of $D'^m$ is $\res(\cU)A_0d$-definable, by stable embeddedness of $\res$. In particular, $X$ is 
$A_0D'(\cU)$-definable, proving stable embeddedness of $D'$ (over $A_0$).
\end{proof}

\begin{proposition}\label{EI Lin}
Let \(M\models\VFA\) and \(A\subseteq\Geom(M)\). Then $\Lin_A$ is stably embedded in $\VFA$ and its \(A\)-induced structure eliminates imaginaries.
\end{proposition}

\begin{proof}
Stable embeddedness follows from stable embeddedness of \(\res\) --- \emph{cf.}  \cref{EQ VFA} --- and the fact that \(\Lin_A\) is internally $\res$-internal (by naming a basis for every sort).

Now, let $T$ be the theory of $\Lin_A(M)$, with its \(\cL_0(A)\)-induced structure. By \cref{F:lin1} and \cref{F:lin-St}, $T$ eliminates imaginaries. Let $\tau$ be the permutation of $\cL_0(A)$ induced by $\sigma$. Then $\tau$ fixes all the formulas on the sort $\res$, and we have $T^\tau=T$. It follows from \cref{cor:StA-3unique} and \cref{F:ChPi-TA}(4) that $T^{(\tau)}\!A$ eliminates imaginaries.

Also, by \cref{EQ VFA} \((\res(M),\sigma_\res)\) is a stably embedded pure model of \(\ACFA\), and hence \(\Lin_A(M) \models T^{(\tau)}\!A\) by \cref{L:ec-twist}. Since the \(A\)-induced structure on \(\Lin_A\) is a definable expansion of its \(\ACF\)-linear structure with a twisted automorphism, elimination of imaginaries follows --- \emph{e.g.}, by \cite[Lemma~5.4]{Hru-GpIm}.
\end{proof}

\subsubsection{Real linear imaginaries}

We conclude these preliminaries with a study of \(\RCF\)-linear structures.

\begin{definition}
An \(\RCF\)-linear structure with flags is said to be \emph{oriented} if for
every sort \(V\) of dimension one, each of the two half lines are
\(\emptyset\)-definable.
\end{definition}

\begin{proposition}\label{RCF lin EI}
Any oriented \(\RCF\)-linear structure with flags eliminates imaginaries.
\end{proposition}

\begin{proof}
Let us first prove a few preliminary results. Let \(M\) be a sufficiently saturated and homogeneous oriented \(\RCF\)-linear structure with flags. First, note that if \(0\to W\to V\to U\to 0\) is an \(\emptyset\)-definable flag,  then any translate of \(W\) in \(V\) is ordered by \(a < b\) if \(a-b\) is in a fixed half line of \(W\).

\begin{claim}\label{RCF lin rig}
\(M\) is rigid: for every \(A\subseteq M\), \(\acl(A) = \dcl(A)\).
\end{claim}

\begin{proof}
Let $X$ be a non-empty finite $A$-definable set such that all elements of $X$ have the same type over $A$. We need to show that $X$ is a singleton. Using tensors, we may assume that \(X\) is contained in some sort \(V\). We proceed by induction on \(\dim(V)\). Let  \(0\to W\to V\to U\to 0\)  be an \(\emptyset\)-definable flag for \(V\). By induction, we may assume that \(X\) projects to a singleton \(b\in U\), \emph{i.e.}, \(X\) is contained in a translate $a+W$ of \(W\) in \(V\) and we can conclude since $a+W$  inherits an  $\emptyset$-definable total order from the ordered group structure on $W$.
\end{proof}

\begin{claim}\label{RCF lin UEI}
Let \(X \subseteq c + W\subseteq V\) be definable for some \(\emptyset\)-definable flag \(0\to W\to V\to U\to 0\) and some \(c\in V\). Then \(X\) is coded.
\end{claim}

\begin{proof}
Since \(\res\) is $o$-minimal and there is a definable order preserving bijection between \(c+W\) and \(\res\), \(X\) is a finite union of points and intervals and hence it is coded by its (finite) border.
\end{proof}

Let \(K = \alg{\res(M)}\) and \(K\tensor M\) be the structure whose sorts are the sorts \(V\) of \(M\) interpreted as \(K\tensor_{\res(M)} V(M)\), with the field structure on \(\res\), the \(\res\)-vector space structure on each \(V\) and the tensor, dual and flag structure. Then \(K\tensor M\) is an \(\ACF\)-structure with flags and for every \(N \subsel M\), and tuple \(c\in M\), since all of the vector spaces have bases in \(N\), \(\dcl(Nc) \subseteq \acl_0(Nc)\) where \(\acl_0\) (respectively \(\dcl_0\)) denotes the algebraic (respectively definable) closure in \(K\tensor M\). Note also that in \(K\tensor M\), we have \(\res(\dcl_0(M)) = \res(M)\) and since each of the vector spaces has a basis in \(M\), \(\dcl_0(M) \subseteq M\). Since \(K\tensor M\) eliminates imaginaries by \cref{F:lin1}, it follows that any \(M\)-definable set in \(K\tensor M\) has a code consisting of elements in \(M\).

By \cite[Remark~3.2.2]{HasHruMac-ACVF}, to prove elimination of imaginaries in \(M\), it suffices to code every definable function \(f : V \to S\), where $S$ is a sort. We will proceed by induction on the dimension of \(V\).
Let \(0\to W\to V\to U\to0\) be an \(\emptyset\)-definable flag for \(V\). Let \(F\) be the Zariski closure of the graph of \(f\) in \(K\tensor M\) --- any choice of basis induces a Zariski topology on \(V\), but this topology is independent of the choice of coordinates. For every \(c\in V(M)\), since \(\dcl(Nc) \subseteq \acl_0(Nc)\), for every \(N\subsel M\), the fiber \(F_c\) of \(F\) above \(c\) is a finite set containing \(f(c)\). As was noted above, \(F\) has a code in \(M\). Note that any \(\sigma\in\aut(M/\code{f})\) can be extended to an automorphism of \(K\tensor M\) fixing \(F\). It follows that \(F(M)\) is defined over \(\code{f}\) and hence it has a code in \(M\cap \code{f}\). Moreover, by compactness and \cref{RCF lin rig}, in \(M\), we find \(\code{F}\)-definable maps \((h_i)_{i<n}\) such that for all \(c\in V(M)\), \(F_c(M) = \{h_i(a)\mid i<n\}\).

Now, fix \(a\in U(M)\) and \(c\in V\) above \(a\). Let \(f_a\) be the restriction of \(f\) to the fiber \(c + W\) above \(a\) in \(V\). Let \(X_{a,i} = \{x \in c+W \mid f_a(x) = h_i(x)\}\). By \cref{RCF lin UEI}, this set is coded in \(M\). Let \(g(a)\) denote the tuple consisting of codes of the \(X_{i,a}\). The function \(g\) is \(\code{f}\)-definable with domain \(U\), so, by induction, \(g\) is coded in \(M\). This concludes the proof since \(f\) is \((\code{F},\code{g})\)-definable.
\end{proof}

\begin{proposition}\label{oriented}
Let \((K,<,v)\) be an ordered field with a non-trivial convex valuation and \(A\subseteq\Geom(A)\). Then \(\Lin_A(K)\) is oriented.

In particular, if \(v\) is henselian and \(\res(K)\models\RCF\), then \(\Lin_A(K)\) is stably embedded and its \(A\)-induced structure eliminates imaginaries.
\end{proposition}

\begin{proof} Dimension one sorts in \(\Lin_A\) are of the form \(\gamma\Val/\gamma\Mid\) for some \(\gamma\in\vg(K)\). But this quotient inherits the order on \(\gamma\Val\); so it is oriented. The rest of the proposition follows from \cref{L:InternallyInternal}, \cref{RCF lin EI} and \cref{F:lin-St}.
\end{proof}

\begin{remark}\label{def order}
\begin{itemize}
\item Any \(K\equiv \Rr((\Qq))\), being real closed, admits a unique field ordering which is definable (without parameters).
\item Any \(K\equiv \Rr((t))\), admits exactly two field orderings, depending on the sign of a choice of uniformizer \(\pi\). Both orders are definable (using an imaginary parameter for a half line in \(\RV_{1,\val(\pi)} := \{\xi \in\RV_1\mid \val(\xi) = \val(\pi)\}\), in particular any element of \(\RV_{1,\val(\pi)}\)).
\end{itemize}
\end{remark}

\section{\texorpdfstring{$C$}{C}-minimal definable generics}\label{gen}

We will now consider generalisations of \cite[Theorem~8.7]{Rid-VDF}. We fix the
following notation for \cref{gen,invariant}.

\begin{notation}\label{notation3}
Let  \(\cL_0 = \Ldiv\) and \(T_0\) be the \(\cL_0\)-theory \(\ACVF\). Let $\cL\supseteq \cL_0$ and \(T\) be a (complete) $\cL$-theory of valued fields. Let $M\models T$ be sufficiently saturated and homogeneous and \(M_0 = \alg{M}\models T_0\). Note that since \(\ACVF\) eliminates quantifiers, we will implicitly assume that every \(\cL_0\)-formula is quantifier free. We will denote by \(\TP^0_x(M)\) the set of (quantifier free) \(\cL_0(M)\)-types (in \(M_0\)) in variables \(x\) and whenever \(\Psi(x;t)\) is a set of \(\cL_0\)-formulas, \(\TP^{\Psi}_{x}(M)\) will denote the set of \(\Psi\)-types over \(M\) --- that is, maximal consistent sets (in \(M_0\)) of formulas \(\psi(x;a)\) and \(\neg\psi(x,a)\) with \(\psi\in\Psi\) and \(a\in M^t\).
\end{notation}

Note that, unless explicitly specified, we do not make any assumption on the
characteristic in \cref{gen}.

\subsection{Main results} In this section we prove the following two density results:

\begin{theorem}\label{def dens}
Assume
\begin{itemize}[leftmargin=50pt]
\item[\Cball] \(T\) is definably spherically complete;
\item[\Cvg] The full induced theory on \(\vg\) is definably complete;
\item[\UFres] The full induced theory on \(\res\) eliminates \(\exists^\infty\);
\item[\UFvg] The full induced theory on \(\vg\) eliminates \(\exists^\infty\).
\end{itemize}
Then, for every strict pro-\(\cL(A)\)-definable \(X\subseteq\K^x\), with \(x\) countable and \(A = \acleq(A) \subseteq \eq{M}\models\eq{T}\), there exists an \(\cL_0(\Geom(A))\)-definable \(p\in\TP^0_x(M)\) consistent with \(X\).
\end{theorem}
In other terms, there exists \(N\supsel M\) and \(a\in X(N)\) such that \(\qftp(a/M)\) is \(\cL_0(\Geom(A))\)-definable. Recall (see, \emph{e.g.}, \cite[Section~2.2]{HruLoe}) that a set is strict pro-definable if it is the limit of a small directed system of definable sets with surjective transition maps. In other terms, it is a \(\star\)-definable set whose projection on any finite set of variables is definable.

\begin{proof}
This is a particular case of \cref{str def dens}
\end{proof}

\begin{remark}
\begin{itemize}
\item Any (non zero) definably complete ordered abelian group \(\Gamma\) is
elementarily equivalent to either \(\Zz\) or \(\Qq\). Indeed, \(\Gamma\) cannot
have a proper non trivial definable convex subgroup and is therefore
elementarily equivalent to a subgroup \(H\) of \((\Rr,+,<)\). If it is not
elementarily equivalent to \(\Zz\) or \(\Qq\), \(H\) is a dense non divisible
subgroup of \(\Rr\). For any \(\gamma\in\Gamma\) non divisible by
\(n\in\Zz_{>0}\), the cut at \(\gamma/n\) yields a counter-example.
\item Hypotheses \Cball{} and \Cvg{} are necessary for the conclusion to hold.
Indeed, the conclusion implies that any \(\fL(M)\)-definable chain \(C\) of
balls is \(\cL_0(M)\)-definable : taking a generic translate, on can ensure that
\(\bigcap_{b\in C} b\) does not contain any \(\cL_0(M)\)-definable chain of
balls, hence any \(\cL_0(M)\)-definable type consistent with this translate of
\(\bigcap_{b\in C} b\) must be the generic of this intersection. Then
\(\bigcap_{b\in C} b\) is a ball, proving both \Cball{} and \Cvg{}.

\item Hypothesis \UFvg{} does not allow for discrete value groups. Note however
that the conclusion of the theorem fails in \(p\)-adic fields. So the hypotheses
\Cball{}, \Cvg{} and \UFres{} cannot be sufficient.
\item As \cref{inv dens} illustrates, by restricting to a mild class of
enrichments of \(\ACVF_\forall\), one can trade hypothesis \UFvg{} for purely
algebraic conditions and a weaker conclusion.
\end{itemize}
\end{remark}

Let \(\Hen[0]\) be the \(\cL_0\)-theory of characteristic zero henselian valued fields.

\begin{theorem}[\emph{cf.} \cref{inv dens H}]\label{inv dens}
Let \(T\) be a \(\res\)-\(\vg\)-enrichment of \(\Hen[0]\), such that:
\begin{itemize}[leftmargin=50pt]
\item[\Cvg] The (full) induced theory on \(\vg\) is definably complete;
\item[\Ram] For every \(n\in\Zz_{>0}\), the interval \([0,\val(n)]\) is finite and \(\res\) is perfect;
\item[\Infres] The residue field \(\res\) is infinite;
\item[\UFres] The (full) induced theory on \(\res\) eliminates
\(\exists^\infty\).
\end{itemize}
Then, for every strict pro-\(\cL(A)\)-definable \(X\subseteq\K^x\), with \(A = \acleq(A) \subseteq \eq{M}\models\eq{T}\), there exist an \(\aut(M/\Geom(A))\)-invariant \(p\in\TP^0_x(M)\) consistent with \(X\).
\end{theorem}

Note that in this setting \(\res\) is stably embedded so the full induced structure coincides with the \(\emptyset\)-induced structure.

\begin{remark}
\begin{itemize}
\item Contrary to \cref{def dens}, \cref{inv dens} requires finite ramification in mixed characteristic. Even if \cref{inv dens} does not apply to characteristic zero non-Archimedean local fields either, \emph{cf.} the stronger \cite[Remark\ 4.7]{HruMarRid}.
\item Under the hypotheses of \cref{inv dens}, locally, we do find definable types: for any \emph{finite} set \(\Psi(x;t)\) of \(\cL_0\)-formulas, we can find an  \(\cL_0(\Geom(A))\)-definable \(p\in\TP^\Psi_x(M)\) consistent with \(X\), cf. \cref{fin dens H}. This local statement does not hold in characteristic zero non-Archimedean local fields.
\item In both theorems, hypothesis \UFres{} is an artefact of our proof. It is necessary to prove certain intermediate results. However, we do not know if it is necessary to prove either theorems. Moreover, these theorems are the only reason hypothesis \UFres{} appears in the imaginary Ax-Kochen-Ershov principle, \emph{cf.} \cref{AKE EI}.
\end{itemize}
\end{remark}

Given these observations, the following questions are quite natural:

\begin{question}
\begin{enumerate}
\item Can the density of either invariant or definable types --- \emph{i.e.}, the conclusion of either theorem --- be proved without assuming \UFres?
\item Under the hypotheses of \cref{inv dens}, can we find an \(\cL_0(\Geom(A))\)-definable type \(p\)?
\item Can the hypotheses of \cref{inv dens} be weakened to also encompass characteristic zero non-Archimedean local fields?
\end{enumerate}
\end{question}

\subsection{The uniform arity one case}
We start by giving a succinct (and slightly more general) presentation of
terminology and results from \cite[Section 6 and 7]{Rid-VDF}.
The types in \cref{inv dens,def dens} are found by finding, for every unary set,
a close-fitting intersection of balls (uniformly over realizations of a type
found at an earlier step). To obtain anything definable, we need to localize to
definable families of (finite sets of) balls. The main technical issue is then
to find large enough families such that the approximation (and later induction
on arity) goes through while keeping it small enough that it stays definable;
this is achieved with the notion of good presentation (\cref{good pres}). 

Once these have been introduced, the main goal of this section is to give two
versions of the approximation process (\cref{rel 1,rel 1 H}). We invite a first
reader to assume in \cref{X tree} and onwards that \(p\) is an arity zero type
(equivalently a realized type) and that \(F_\lambda\) is the collection of all
balls (open and closed) to get an idea of the base arity one case with fixed
parameters.

\begin{definition}
\label{fballs}
We define \(\balls\) to be the (\(\cL_0\)-definable) set of balls (closed or
open) in models of \(T_0\) --- the field itself is the open ball of radius
\(-\infty\) and points are closed balls of radius \(+\infty\). For every
\(r\in\Zz_{> 0}\), \(\fballs{r}\) is the set of finite (potentially empty)
subsets of cardinality at most \(r+1\) of \(\balls\) of the same radius
and either all open or all closed --- in particular, there is no nesting among
the elements of some \(B\in{r}\). Let also \(\fballs{<\infty} := \bigcup_{r> 0}
\fballs{r}\).

Similarly, we denote \(\fpoints{r} \subseteq \fballs{r}\) the set of finite
subsets of cardinality at most \(r+1\) of \(\K\) and \(\fpoints{<\infty} :=
\bigcup_{r> 0} \fpoints{r} \subseteq \fballs{<\infty}\).
\end{definition}

For every finite set of balls \(B\), we define \(\points{B} := \bigcup_{b\in
B}b\) and for every finite sets of balls \(B_1\) and \(B_2\), we write \(B_1\leq
B_2\) if \(\points{B_1} \subseteq \points{B_2}\). For every
\(b_1,b_2\in\balls\), we also define \(d(b_1,b_2) := \inf\{\val(x_1-x_2)\mid
x_i\in b_i\}\). Note that this is not a metric on the space of balls since
\(d(b_1,b_1) = \rad(b_1)\), the radius of \(b_1\). Finally, for
\(B_1,B_2\in\fballs{r}\), we define \(D(B_1,B_2) := \{d(b_1,b_2) \mid b_i\in
B_i\}\). We enumerate \(D(B_1,B_2) \subseteq \vg\) in increasing order. Let
\(d_i(B_1,B_2)\) be the \(i\)-th element of this enumeration. So that
\(d_i(B_1,B_2)\) is defined for all \(i < r^2\), for every \(i\) above the
cardinality of \(D(B_1,B_2)\), we set \(d_i(B_1,B_2)\) to the the maximal
element. This choice of an enumeration (with repetitions) of \(D(B_1,B_2)\) does
not actually matter, as long as it is chosen uniformly.

Let us now fix a set \(\Psi(x;t)\) of \(\cL_0\)-formulas, an \(\cL_0\)-definable
set \(\Lambda\), an integer \(r\) and an \(\cL_0\)-definable family \(F =
(F_\lambda)_{\lambda\in\Lambda}\) of functions \(F_\lambda : \K^x \to
\fballs{r}\). We now wish to give sufficient conditions on \(\Psi\) and \(F\)
that will allow us to proceed with certain classical unary constructions in
valued fields, uniformly over realisations of \(\Psi\)-types. In particular, it
will allow is to describe (local) types in \(n+1\)-variables as generics of
balls parametrized by \(n\)-variables.

\begin{definition}
Let \(p \in \TP^{\Psi}_{x}(M)\).
\begin{enumerate}
\item We say that \(p\) is adapted to \(F\) if, for each of the following
statement, \(p\) implies either this statement or its negation:
\begin{itemize}
\item \(F_\lambda(x)\square \bigcup_{i < r} F_{\mu_i}(x)\) 
where \(\lambda,\mu_i \in\Lambda(M)\) and \(\square\in\{=, \subseteq, \subset, \leq, <\}\);
\item \(\points{F_\lambda}(x) = \points{F_{\mu_1}}(x) \cap \points{F_{\mu_2}}(x)\), where  \(\lambda,\mu_i \in\Lambda(M)\);
\item Every ball in \(F_\lambda(x)\) is closed;
\item \(\rad(F_\lambda(x))\square d_i(F_{\mu_1}(x),F_{\mu_2}(x))\), where  \(\lambda,\mu_i \in\Lambda(M)\), \(\square\in\{=, \leq\}\) and \(i < r^2\);
\end{itemize}
\item We say that \(F\) is closed under intersections over \(p\) if for every \(\lambda,\mu\in\Lambda(M)\), there exists \(\epsilon\in\Lambda(M)\) with \(p(x)\vdash \points{F_\lambda}(x)\cap\points{F_\mu}(x) = \points{F_\epsilon}(x)\) --- and we further assume that there exists \(\eta\in\Lambda(M)\) such that \(F_\eta(x) =\{\K\}\).
\item We say that \(F\) is closed under complement over \(p\) if for every \(\lambda,\mu\in\Lambda(M)\), with \(p(x)\vdash F_{\mu}(x)\subseteq F_{\lambda}(x)\), there exists \(\epsilon\in\Lambda(M)\) with \(p(x)\vdash F_\epsilon(x)=F_\lambda(x) \setminus F_\mu(x)\).
\end{enumerate}
\end{definition}

Note that, in the above definition, the \(\fL\)-structure on \(M\) does not matter.

\begin{remark}
The family of all constant functions to \(\balls\) (over any type), is an
important example of the above properties. This simple family suffices to prove
\cref{def dens,inv dens} for \(X\subseteq \K^1\). Dealing with higher arity
definable sets, however, requires non-constant functions.
\end{remark}

Let now \(p \in \TP^{\Psi}_{x}(M)\) be adapted to \(F\), and let us assume that
\(F\) is closed under intersections and complement over \(p\).

\begin{definition}
 Let \(\lambda\in\Lambda(M)\). We say that \(F_\lambda\) is \emph{irreducible over \(p\)}, if for every \(\mu\in\Lambda(M)\), \(p(x)\vdash F_\mu(x)\subseteq F_\lambda(x)\) implies \(p(x)\vdash F_\mu(x) = \emptyset \vee F_\mu(x) = F_\lambda(x)\).
\end{definition}

We define \(\Lambda_p(M) := \{\lambda\in\Lambda(M) \mid F_\lambda\text{ irreducible over }p\}\).

\begin{lemma}
For every \(\lambda\in\Lambda(M)\), there exists finitely many
\(\mu_i\in \Lambda_p(M)\) with \(p(x)\vdash F_{\lambda}(x)=\bigcup_i
F_{\mu_i}(x)\).
\end{lemma}

\begin{proof}
Since \(p\) is adapted to \(F\) it suffices to check that the lemma holds at one
realization \(a\) of \(p\). We can then proceed by induction on the cardinality
of \(F_\lambda(a)\), using closure under complement to conclude.
\end{proof}

\begin{lemma}
\label{ball tree}
For every \(\lambda,\mu\in\Lambda_p(M)\), we have
\[p(x)\vdash \points{F_{\lambda}}(x)\cap \points{F_{\mu}}(x)=\emptyset \vee F_{\lambda}(x)\leq F_{\mu}(x) \vee F_{\mu}(x)\leq F_{\lambda}(x).\]
\end{lemma}

\begin{proof}
Again, it suffices to check for one \(a\models p\). We may assume that the balls
in \(F_\lambda(a)\) have smaller (or equal) radius than those in \(F_\mu(a)\)
and if the radii are equal and the balls in \(F_\lambda(a)\) are open, so are
the balls in \(F_\mu(a)\). By closure under intersection, we find
\(\epsilon\in\Lambda(M)\) such that \(\points{F}_\lambda(a) \cap
\points{F_\mu}(a) = \points{F_{\epsilon}}(a)\). By hypothesis on the radii,
\(F_\epsilon(a) \subseteq F_\lambda(a)\) and hence, by irreducibility, either
\(F_\epsilon(a) = \emptyset\) or \(F_\epsilon(a) = F_\lambda(a)\).
\end{proof}

We will later need some further hypotheses (\emph{cf.} \cref{quant}) on \(\Psi\) and \(F\) leading to the following definition:

\begin{definition}\label{good pres}
Let \(\Psi(x;t)\) be a set of \(\cL_0\)-formulas, and \(F = (F_\lambda)_{\lambda\in\Lambda} : \K^x \to \fballs{r}\) be a definable family of functions --- for some \(\cL_0\)-definable \(\Lambda\) and integer \(r\). We say that \((\Psi,F)\) is a good presentation if, for any \(p\in\TP^{\Psi}_{x}(M)\)
\begin{enumerate}
\item \(p\) is adapted to \(F\);
\item \(F\) is closed under intersection and complement over \(p\);
\item \(F\) has large balls over \(p\), \emph{i.e.}, for every \(\lambda,\mu\in\Lambda(M)\) and \(i\in\Zz_{\geq 0}\):
\begin{itemize}
\item If \(p(x)\vdash F_{\lambda}(x)\neq \K\), there is \(\eta\in\Lambda(M)\) such that \(p(x)\vdash \rad(F_{\eta}(x)) = d_i(F_{\lambda}(x),F_{\mu}(x)) \wedge{}\)"\(F_{\eta}(x)\) is closed"\({}\wedge F_{\lambda}(x)\leq F_{\eta}(x)\);
\item If \(p(x)\vdash\) "\(F_{\lambda}(x)\) is open"\(\,\vee\,
\rad(F_{\lambda}(x)) < d_i(F_{\lambda}(x),F_{\mu}(x))\), there is
\(\eta\in\Lambda(M)\) such that, \(p(x)\vdash \rad(F_{\eta}(x)) =
d_i(F_{\lambda}(x),F_{\mu}(x)) \wedge{}\)"\(F_{\eta}(x)\) is open"\({}\wedge
F_{\lambda}(x)\leq F_{\eta}(x)\).
\end{itemize}
\end{enumerate}

Let \(\Delta(xy;s)\) with \(\card{y}=1\) be a set of \(\cL_0\)-formulas. We say
that \((\Psi,F)\) is a \emph{good presentation for \(\Delta\)} if \((\Psi,F)\)
is a good presentation and every \(M\)-instance of \(\Delta\) is a Boolean
combination of \(M\)-instances of \(\Psi\) and formulas
\(y\in\points{F_\lambda(x)}\) with \(\lambda\in\Lambda(M)\).

Let \(G := (G_{\omega})_{\omega\in\Omega} : \K^x\to \fballs{\ell}\) be an
\(\cL_0\)-definable family of functions. If, moreover, for every
\(\omega\in\Omega(M)\), there exists \(\lambda\in\Lambda(M)\) such that
\(G_\omega = F_\lambda\). we say that \((\Psi,F)\) is a \emph{good presentation
for \((\Delta,G)\)}.
\end{definition}

An important point is that \emph{finite} good presentations always exist:

\begin{proposition}\label{Ex GP}
Let \(\Delta(xy;s)\) be a finite set of \(\cL_0\)-formulas with \(|y| = 1\) and \((G_{\omega})_{\omega\in\Omega} : \K^x\to \fballs{\ell}\) be \(\cL_0\)-definable. Then there exists a finite set of \(\cL_0\)-formulas \(\Psi(x;t)\) and an \(\cL_0\)-definable \(F := (F_\lambda)_{\lambda\in\Lambda}: \K^ x \to \fballs{r}\) such that \((\Psi,F)\) is a good presentation for \((\Delta,G)\).
\end{proposition}

We only sketch the proof --- details of the precise encodings can be found in \cite[Proposition~6.14, 6.15, 6.18 and 7.12]{Rid-VDF}.

\begin{proof}
The existence of \(\Psi\) and \(F\) such that any instance of \(\Delta\) is an Boolean combination of instances of \(\Psi\) and \(y \in \points{F}_\lambda(x)\) follows by compactness from the Swiss cheese decomposition. Enlarging \(F\), we may assume it contains \(G\) and that condition (3) hold. At any point, enlarging \(\Psi\), we may assume that condition (1) holds. Since the intersection of two balls is either empty or one of these balls, \(F\) can be closed under intersection by considering the family of \(r+1\)-fold intersections of \(F\). Closure under complement can be obtained by considering the finite Boolean algebra generated by the subsets, appearing in \(F\), of any given \(F_\lambda\) (over some realisation of \(p\)). They are generated in (uniformly) finitely many steps and hence can be considered as the elements of one single family. This concludes the proof since the previous two steps preserve condition (3).
\end{proof}

\begin{remark}
A good presentation \((\Psi(x;t),F)\) remains a good presentation as \(\Psi\) grows. So, given a set \(\Psi(xy;t)\) of \(\cL_0\)-formulas and an \(\cL_0\)-definable \(F := (F_\lambda)_{\lambda\in\Lambda} : \K^x \to \fballs{r}\), we say that \((\Psi(xy;t),F)\) is a good presentation if there exists \(\Phi(x;t)\subseteq\Psi\) such that  \((\Phi(x;t),F)\) is a good presentation.
\end{remark}

We now fix a good presentation \((\Psi(xy;t),F)\), with \(\card{y} = 1\), and \(p\in\TP^{\Psi}_{xy}(M)\). Let \((\Psi,F)\) be the set of \(\cL_0\)-formulas \(\Psi\cup\{y\in\points{F_\lambda}(x)\}\) in variables \(xy\) and parameters \(t\lambda\). Let \(\TP^{\Psi,F}_{xy}(M)\) denote the space of \((\Psi,F)\)-types over \(M\) (in \(M_0\)).

\begin{definition}
For every \(\cL(M)\)-definable maps \(f,g : X \to Y \) and every partial
\(\cL(M)\)-type \(q\) concentrating on \(X\), we say that \(f\) and \(g\) have
the same \emph{\(q\)-germ}, and we write \(\germ{q}{f} = \germ{q}{g}\), if
\(q(x) \vdash f(x) = g(x)\).
\end{definition}

\begin{definition}
For every \(\lambda,\mu\in\Lambda(M)\), we write \(\lambda \leq_p \mu\) whenever
\[p(xy) \vdash F_\lambda(x) \leq F_\mu(x).\]
\end{definition}

Note that \(\leq_p\) is an \(\cL(M)\)-definable pre-order. Recall that elements
of any \(B\in \fballs{r}\) cannot be nested. It follows that, for any two
\(B_1,B_2\in\fballs{r}\), \(B_1 = B_2\) if and only if \(\points{B_1} =
\points{B_2}\). So the equivalence relation associated to \(\leq_p\) is equality
of \(p\)-germs. We therefore write \(\lambda <_p \mu\) whenever \(\lambda \leq_p
\mu\) and they have distinct \(p\)-germs.

Moreover, when restricted to \(\Lambda_p\), by \cref{ball tree},
there is a largest element, \(\K\), and the \(\leq_p\)-upwards closure of any
\(\lambda\in \Lambda_p \sminus [\emptyset]_p\) is totally ordered:
\((\Lambda_p\sminus \germ{p}{\emptyset},\leq_p)\) is a tree.

\begin{definition}
Let \(E\subseteq \Lambda_p(M)\). The \emph{generic type of \(E\) above \(p\)} is
\begin{align*}
\gen{E}{p}(xy) &:= p(xy)\\
&\cup\{y\in \points{F_\mu}(x)\mid \mu\in E\}\\
&\cup\{y\nin \points{F_\lambda}(x)\mid  \lambda\in\Lambda(M)\wedge \forall \mu\in E\ \lambda <_p \mu\}.
\end{align*}
\end{definition}

This is the partial type of realizations of \(p\) such that \(y\) is in
\(\bigcap_{\mu \in E}\points{F_\mu}(x)\), but in no strict subset of the form
\(\points{F_\lambda}(x)\). Provided \(\gen{E}{p}\) is consistent it generates a
complete \((\Psi,F)\)-type that we also denote
\(\gen{E}{p}\in\TP^{\Psi,F}_{xy}(M)\). Further assuming that \(p\) is
\(\cL(M)\)-definable (as a \(\Psi\)-type), then \(\Lambda_p(M)\) is a
\(\cL(M)\)-definable set, and if \(\cE\subseteq \Lambda_p\) is an
\(\cL(M)\)-definable then the type \(\gen{\cE(M)}{p}\) is an
\(\cL(M)\)-definable \((\Psi,F)\)-type, provided it is consistent. We denote it
\(\gen{\cE}{p}\).

\begin{definition}
Let \(\pi(x)\) be a partial \(\cL(M)\)-type and \(A\subseteq \eq{M}\). We say
that \(\pi\) is \emph{\(\cL(A)\)-quantifiable over \(\cL\)} if, for every
\(\cL\)-formula \(\phi(x;t)\), there exists an \(\cL(A)\)-formula \(\theta(t)\)
such that \(\{b\in M^t \mid \pi(x)\vdash \phi(x;b)\} = \theta(M)\). When it
exists, we write \(\fa{\pi} x\,\phi(x;t) := \theta(t)\) and \(\ex{\pi}x\,
\phi(x;t) = \neg(\fa{p}x\,\neg\phi(x;t))\).
\end{definition}

\begin{remark}
\begin{enumerate}
\item Such a type is often also a "definable partial type" in the literature.
There is however some ambiguity on the terminology, \emph{cf.}
\cite[Remark~7.2.(ii)]{Rid-VDF}, hence the present distinct choice of
terminology.
\item If, for some set of \(\cL_0\)-formulas \(\Delta(x;y)\), \(p(x)\) is a
complete \(\cL(A)\)-quantifiable \(\Delta\)-type over \(M\), then it is
\(\cL(A)\)-definable, as a \(\Delta\)-type --- that is \(\forall_p
x\,\phi(x;t)\) exists for every \(\phi(x;y)\in\Delta\). As we will see in
\cref{quant}, under certain hypotheses on \(T\) and \(\Delta\), the converse
also holds.
\end{enumerate}
\end{remark}





We can now prove the crucial step in proving \cref{def dens,inv dens}: the relative arity one case. Let us now assume that \(p\) is $\cL(A)$-quantifiable over \(\cL\), where \(A = \acleq(A) \subseteq \eq{M}\), and consistent with some \(\cL(A)\)-definable $X\subseteq \K^{xy}$.

\begin{definition}\label{X tree}
For every \(\lambda,\mu\in\Lambda_p\), let \(\lambda \treeq \mu\) hold whenever
\[\fa{p}{xy}\,y \in (X_x \cap\points{F_\lambda}(x)) \to y\in\points{F_\mu}(x),\]
where \(X_x = \{y\mid xy \in X\}\) denotes the fiber above \(x\).
\end{definition}

The relation \(\treeq\) is an \(\cL(A)\)-definable preorder on \(\Lambda_p\) and
we denote \(\treequiv\) the associated equivalence relation. Since \(\leq_p\)
refines \(\treeq\) on \(\Theta_p := \Lambda_p\sminus (\emptyset/{\treequiv})\),
this is also a tree with root \(\K/\treequiv\) and \(\treequiv\)-classes are
\(\leq_p\)-convex.

\begin{lemma}
\label{discrete}
For every \(\lambda\in\Theta_p\), if the generic
\(\gen{\lambda/{\treequiv}}{p}\) of \(\lambda/{\treequiv}\) over \(p\) is not
consistent with \(X\), then \(\lambda/{\treequiv}\) has finitely many
$\treeq$-daughters \((\mu_i/{\treequiv})_{0\leq i < n} \in
\acleq(A\code{\lambda/{\treequiv}})\) in \(\Theta_p/{\treequiv}\). Moreover,
\(n\geq 2\) and \(p(xy)\vdash y\in X_x \cap \points{F_\lambda}(x) \to
\bigvee_{i< n} y\in \points{F_{\mu_i}}(x)\).
\end{lemma}

\begin{proof}
Let us assume that \(\gen{\lambda/{\treequiv}}{p}\) is not consistent with
\(X\), \emph{i.e.} \(p(xy)\cup \{(x,y)\in X\} \cup \{y\in F_\lambda(x)\} \cup
\{y\nin F_\nu(x) \mid \nu \tree \lambda\}\) is not consistent. By compactness,
there exist \((\nu_i)_{0\leq i < m} \in \Theta_p(M)\) such that \(\nu_i \tree
\lambda\) and \(p(xy)\vdash y\in X_x \cap \points{F_\lambda}(x) \to \bigvee_{i<
m} y\in \points{F_{\nu_i}}(x)\). The existence of the \(\mu_i\) now follows from
the facts that any \(\mu \treeq \lambda\) is \(\treeq\)-comparable to one of the
\(\nu_i\) and that since the \(F_{\nu_i}\) are irreducible, the subtree with
root \(\lambda\) and leaves \((\nu_i)_{i<m}\) embeds in the lattice of subsets
of \(\{0,\ldots,m-1\}\), which is finite --- we refer the reader to
\cite[Claim~8.4]{Rid-VDF} for details. Finally, if \(n =1\), we would have
\(\lambda \treeq \mu_0\), contradicting that \(\mu_0\) is a daughter of
\(\lambda\).
\end{proof}

Let \(\overline{\Theta}_p := \{\mu\in\Theta_p\mid \fa{p}xy\) "\(F_{\mu}(x)\) is
closed"\(\}\) and, for every \(\lambda\in\Lambda\), let \(Y_{\lambda} :=
\{\mu\in\overline{\Theta}_p \mid \fa{p}xy\,\rad(F_\mu(x)) =
\rad(F_\lambda(x))\}\).

\begin{lemma}\label{disj tree}
One of the following holds:
\begin{itemize}
\item There exists a \(\lambda\in\Theta_p\) such that \(\lambda/{\treequiv}\in A\) and \(\gen{\lambda/\equiv}{p}\) is consistent with \(X\).
\item There exists \(\lambda/{\treequiv}\in A\) with \(\germ{p}{Y_\lambda} :=
\{\germ{p}{F_\mu} \mid \mu\in Y_\lambda\}\) finite of arbitrarily large
cardinality.
\end{itemize}
\end{lemma}

\begin{proof}
Assume that \(X\) is consistent with no \(\gen{\lambda/{\treequiv}}{p}\), where \(\lambda/{\treequiv}\in A\). Then, by \cref{discrete}, \(\Theta_p/\treequiv\) admits an initial finitely strictly branching discrete tree --- that is every, element has at least two daughters --- with every branch infinite. Note that, for every \(\lambda\in\Theta_p\) with \(\lambda <_p \K\), by the large ball property, there is \(\mu\in\Lambda\) with \(\lambda\leq_p\mu\) and \(\fa{p}{xy}\) "\(F_\mu(x)\) is closed"\({}\wedge\rad(F_\mu(x)) = \rad(F_\lambda(x))\). We may assume that \(F_\mu\) is irreducible over \(p\). Then \(\lambda = \mu\) or \(\lambda\) is the unique \(\leq_p\)-daugther of \(\mu\). Note also that, by the large ball property, \(\overline{\Theta}_p \cap \K/{\treequiv} \neq \emptyset\). It follows that \(\overline{\Theta}_p/\treequiv\) also admits an initial finitely branching discrete tree, denoted \(\Xi_p\), with every branch infinite.

Note that, for any two \(\mu,\nu\in Y_\lambda\), since \(\fa{p}xy\,\rad(F_\mu(x)) = \rad(F_\nu(x))\), we have that \(\germ{p}{F_\mu} = \germ{p}{F_\nu}\) implies \(\mu\treequiv\nu\), which implies that \(\fa{p}{xy}\,\points{F_\mu}(x) \cap \points{F_\nu}(x) \neq \emptyset\), which, by irreducibility, implies that \(\germ{p}{F_\mu} = \germ{p}{F_\nu}\), so these three statements are equivalent. In particular, the identity induces a bijection between \(\germ{p}{Y_\lambda}\) and \(Y_\lambda/{\treequiv}\).

We now build, by induction, \(\lambda_i\in\Lambda\) such that \(Y_{\lambda_i}/{\treequiv}\subseteq \Xi_p\) and \(\card{Y_{\lambda_i}/{\treequiv}} = \card{\germ{p}{Y_{\lambda_i}}}\) is finite and strictly increasing. Start with any \(\lambda_0 \in \overline{\Theta}_p\cap\K/{\treequiv}\). Then \(Y_{\lambda_0} = \germ{p}{F_\lambda}\) and \(Y_{\lambda_0}/{\treequiv} = \lambda_0/{\treequiv} = \K/{\treequiv}\). If \(\lambda_i\) is built, let \((\mu_j)_{j < m}\) enumerate all the \(\treeq\)-daughters --- in \(\Xi_p\) --- of the elements in \(Y_{\lambda_i}/{\treequiv}\). Let \(j_0\) be such that, for all \(j\), \(\fa{p}{xy}\,\rad(F_{\mu_{j_0}}(x)) \leq \rad(F_{\mu_{j}}(x))\). For every \(\nu\in Y_{\mu_{j_0}}\), by the large ball property, we find \(\lambda\in Y_{\lambda_i}\) such that \(\nu\leq_p\lambda\). Since \(\fa{p}{xy}\ \rad(F_\nu(x)) = \rad(F_{\mu_{j_0}}(x)) \leq \rad(F_{\mu_{j}}(x))\), we cannot have \(\nu\tree \mu_j\) and hence \(\nu/{\treequiv}\) is either in \(Y_{\lambda_i}/{\treequiv}\) or it is one of the \(\mu_j/{\treequiv}\). So \(Y_{\mu_{j_0}}/{\treequiv}\subseteq\Xi_p\) is finite.

Furthermore, for every \(\mu_j\), by the large ball property, there exists \(\nu \in Y_{\mu_{j_0}}\) such that \(\mu_j\leq_p\nu\). It follows that, for every element of \(Y_{\lambda_i}/{\treequiv}\), either it or all of its (more than one) daughters appear in \(Y_{\mu_{j_0}}\). In particular, all the sisters of \(\mu_{j_0}/{\treequiv}\) appear, and hence \(\card{Y_{\mu_{j_0}}/{\treequiv}} > \card{Y_{\lambda_i}/{\treequiv}}\). Thus, we can choose \(\lambda_{i+1} = \mu_{j_0}\).
\end{proof}

We can now eliminate the second option in \cref{disj tree} by imposing a uniform bound on the size of finite instances of \((Y_\lambda)_{\lambda\in\Lambda}\):

\begin{corollary}\label{rel 1}
Assume:
\begin{itemize}[leftmargin=50pt]
\item[\UFradgloc{p}{F}] For every \(\cL(M)\)-definable family \((Y_z)_z\) of subsets of \(\germ{p}{F_{\Lambda_p}} := \{\germ{p}{F_\lambda}\mid\lambda\in\Lambda_p\}\) such that for all \(z\) and \(\germ{p}{F_\lambda}, \germ{p}{F_{\mu}} \in Y_z\), \(p(xy)\vdash\) "\(F_\mu(x)\) is closed"\({}\wedge \rad(F_\lambda(x)) = \rad(F_\mu(x))\), there exists \(n\in\Zz_{>0}\) such that, for all \(z\), \(\card{Y_z}<\infty\) implies \(\card{Y_z} \leq n\).
\end{itemize}

Then there exists an $\cL(A)$-definable \(\cE\subseteq \Lambda_p\) such that \(\gen{\cE}{p}\) is consistent with \(X\).\qed
\end{corollary}

However the family \(\germ{p}{Y_\Lambda}\) is not any definable family in
\(\germ{p}{F_{\Lambda_p}}\). It has certain geometric properties that reflect
that of \(X\). In particular, with further hypotheses on \(X\), we can dispense
with \(\UFradgloc{p}{F}\) altogether, as we will see in \cref{rel 1 H}.

We now wish to apply the construction above in the pair \((M_0,M)\) which is naturally an \(\cL_{0,P}\) structure enriched with the \(\cL\)-structure on \(M\). To be precise and avoid an unnecessary conflict of notation:
\begin{notation}
Let \(\cL_1\) be some expansion of \(\cL_0\) and \(T_1\) some \(\cL_1\)-theory of valued fields. In the following lemma, we apply the above with \(T\) the theory of the pair \(M := (\alg{M_1},M_1)\), where \(M_1\models T_1\) is sufficiently saturated and homogeneous, in the language \(\cL := \LP\) consisting of \(\cL_{0,P}\) structure enriched with the \(\cL_1\)-structure on \(\bP\) --- so \(M_0 = \alg{M_1}\).
 \end{notation}

Let us now introduce some useful terminology from \cite{CluHalRK}:

\begin{definition}\label{preparation}
Fix \(n\in\Zz_{> 0}\) invertible in
\(\K(M)\). For any ball \(b\), we define \(\lball{b}{n}:=\{a + n^{-1}(a-c) \mid
a,c \in b\}\). It is a ball of radius \(\rad(b)-\val(n)\) around \(b\), open if
\(b\) is, closed otherwise. For a set of balls $B$, we set $B[n]:=\{b[n]:b\in
B\}$.
\begin{enumerate}
\item An \(\cL(M_1)\)-definable set \(X\subseteq\K(M_1)\) is
\emph{\(n\)-prepared} by some finite set \(C\subseteq\K(M_0)\) if for every ball
\(b\in \balls(M_0)\) with \(\lball{b}{n}\cap C = \emptyset\), either \(b\cap
X(M_1) = \emptyset\) or \(b\cap X(M_1) = b(M_1)\).
\item We say that some \(\cL_0(M_1)\)-definable \(G : \K^n \to \fpoints{r}\)
\emph{\(n\)-prepares} \(X\subseteq\K^{n+1}(M_1)\) if, for every
\(x\in\K(M_1)^n\), \(G(x)\) \(n\)-prepares \(X_x\).
\item We say that \(X\subseteq\K^{n+1}(M_1)\) is \emph{\(n\)-prepared} by \(F\)
if there exists \(\lambda\in\Lambda(M_1)\) such that \(F_\lambda\) has values in
\(\fpoints{r}\) and \(n\)-prepares \(X\). 
\end{enumerate}
\end{definition}

\begin{remark}\label{Hen0 prep}
By field quantifier elimination (\emph{cf.} \cref{EQ rv}), if \(M_1\) is a pure henselian field of characteristic zero, any \(\cL(M_1)\)-definable \(X\subseteq \K\) is \(p^{\ell}\)-prepared, for some \(\ell\), by the finite set of roots of polynomials that appear in the (field quantifier free) definition of \(X\), where \(p\) is either \(1\) or the residue characteristic when it is positive.
\end{remark}

Let also \(A_1 = \acleq_1(A_1)\subseteq \eq{M_1}\), \(X\subseteq \K^{xy}(M_1)\)
be \(\cL_1(A_1)\)-definable, \(A = A_P = \acleq(A_1)\) and
\(p\in\TP^{\Psi}_{xy}(M_0)\) be \(\LP(A_P)\)-definable and consistent with
\(X\).

\begin{lemma}\label{rel 1 H}
Let \(\Phi(x;t)\subseteq\Psi\) be such that \((\Phi,F)\) is a good presentation for \(\Psi\). Assume that there is some \(n\in \Zz_{> 0} \cap \K^\times(M)\) such that:
\begin{itemize}[leftmargin=50pt]
\item[\Prep{X}{F,n}] The set \(X\) is \(n\)-prepared by \(F\);
\item[{\Ram[n]}] The interval \([0,\val(n)]\subseteq \vg(M_1)\) is finite and \(\res(M_1)\) is perfect;
\item[\Infres] The residue field \(\res(M_1)\) is infinite.
\end{itemize}
Then, there exists an $\LP(A_P)$-definable \(\cE\subseteq \Lambda_p(M_0)\) such that \(\restr{\gen{\cE}{p}}{M_0}\) is consistent with \(X\).
\end{lemma}

\begin{proof}
Let \(\rho\) be such that \(F_\rho(x)\) \(n\)-prepares \(X_x\) for all \(x\). By \cref{disj tree}, applied in \(M = (\alg{M_1},M_1)\), either the conclusion of the \namecref{rel 1 H} holds, or we can find \(\lambda/{\treequiv}\in A\) with \(\card{\germ{p}{Y_\lambda}} > r\). Then some element of \(Y_\lambda\), say \(F_\lambda\), does not contain any point of \(F_\rho(x)\). Replacing \(\lambda\) with any \(\mu\treeq\lambda\) in \(\Xi_p\) such that \(\card{[\mu,\lambda]} \geq \card{[0,\val(n)]}\), we may further assume that \(F_\rho(x)\cap\points{\lball{F_\lambda(x)}{n}} = \emptyset\) --- where, for any \(B\in\fballs{<\infty}\), \(B[n] := \{b[n]\mid b\in B\}\) --- and that \(\lambda \tree \K\).

By \cref{discrete}, we have \(p(xy)\vdash y\in X_x \cap \points{F_{\lambda}}(x) \to \bigvee_{i<n} y\in \points{F_{\mu_i}}(x)\), where the \((\mu_i/{\treequiv})_{i<n}\) are the daughters of \(\lambda/{\treequiv}\). By compactness, there exists some \(\psi(xy)\in p\) such that  \(q := \restr{p}{\Phi} \vdash \forall y\ \psi(xy) \wedge y\in X_x \cap \points{F_{\lambda}}(x) \to \bigvee_{i<n} y\in \points{F_{\mu_i}}(x)\). Since \((\Phi,F)\) is a good presentation for \(\Psi\), there are \(\kappa\) and \((\mu_i)_{n\leq i < m} \in\Lambda_p\) with \(\models\psi(xy)\leftrightarrow y\in\points{F_{\kappa}}(x)\sminus(\bigcup_{n \leq i < m} \points{F_{\mu_i}}(x))\). In particular, \(p(xy)\vdash y\in\points{F_{\kappa}}(x)\). It follows that \(\kappa \treequiv \K\) and hence \(\lambda \treeq \kappa\). So, we have \[q \vdash \points{F_\lambda}(x) \cap X_x \subseteq \bigcup_i \points{F_{\mu_i}}(x).\]

Since \(\lambda\ntreequiv \emptyset\), there exists \(ac\models p\) in some \(N_1\supsel M_1\) such that \(c\in X_a \cap\points{F_\lambda}(a)\). Let \(b_0\in \balls(\alg{N_1})\) be the ball of \(F_\lambda(a)\) containing \(c\). Since \(F_\rho(a)\cap \lball{b_0}{n}= \emptyset\), we have \(b_0\cap X_a(N_1) = b_0(N_1)\). It follows that \(b_0(N_1) \subseteq \bigcup_i\points{F_{\mu_i}}(a)\). By construction, \(b_0(N_1)\) is not covered by any single ball in \(\bigcup_i F_{\mu_i}(a)\). So \(\{\val(x-y) \mid x,y\in b_0(N_1)\}\) has a minimal element (realised by some \(x, y\) in distinct balls of \(\bigcup_i F_{\mu_i}(a)\)). Let \(b\in \balls(N_1)\) be the smallest closed ball containing \(b_0(N_1)\). Then \(b(N_1) = b_0(N_1)\) is covered by finitely many of its maximal open subballs, contradicting hypothesis \Infres.
\end{proof}

\subsection{Quantifiable types}
To use the above constructions in an induction, we need a number of results on
quantifiable types. The first one is that \(\gen{\cE}{p}\) is itself
quantifiable when \(p\) is. Recall our general setup (\cref{notation3}) for this
section.

For any finite set \(B\) of balls, let \(\res_B\) be
the set of maximal open subballs of the balls \(b \in B\) and \(\resf_B :
\points{B} \to \res_B\) be the projection.

\begin{lemma}[cf. {\cite[Corollary~6.9]{Rid-VDF}}]\label{quant}
Let \((\Psi(xy;t),(F_\lambda(x))_{\lambda\in\Lambda})\) be a good presentation. Let $p\in\TP^{\Psi}_{xy}(M)$ be $\cL(A)$-quantifiable over \(\cL\), where \(A\subseteq \eq{M}\). Assume:
\begin{itemize}[leftmargin=50pt]
\item[\UFresgloc{p}{F}] For any \(\cL(M)\)-definable \((Y_z)_z \subseteq \germ{p}{F_\Lambda}\) and \(\lambda\in\Lambda(M)\) such that, for all \(z\) and \(\germ{p}{F_\mu} \in Y_z\), \(p(xy) \vdash F_\mu(x) \subseteq \res_{F_{\lambda}(x)}\), there exists \(n\in\Zz_{\geq0}\) such that, for all \(z\), \(\card{Y_z}<\infty\) implies \(\card{Y_z} \leq n\).
\end{itemize}
Then, any \(\cL(A)\)-definable \(q \in \TP_{xy}^{\Psi,F}(M)\) containing \(p\) is \(\cL(A)\)-quantifiable over \(\cL\).
\end{lemma}

\begin{proof}
Let \(\cE := \{\lambda\in\Lambda_p\mid q(xy)\vdash y\in \points{F_\lambda}(x)\}\). Then \(\cE\) is \(\cL(A)\)-definable and \(q = \gen{\cE}{p}\). If $\cE$ does not have an \(\leq_p\)-minimal element which is closed, for any \(\cL\)-formula \(\phi(xy;s)\) and \(e\in M^s\), \(q(xy) \vdash \phi(xy;e)\) if and only if there exists \(\lambda\in\cE(M)\) and \(\mu\in\Lambda_p(M)\) with \(\lambda <_p \cE\) and \(q(xy) \vdash y\in \points{F_\lambda}(x)\sminus\points{F_\mu}(x) \to \phi(xy;e)\) --- cf. \cite[Proposition~6.4]{Rid-VDF}. So let us assume that $\cE$ has an \(\leq_p\)-minimal element \(\lambda_0\) which consists of closed balls. If $p(xy)\vdash \rad(F_{\lambda_0}(x)) = +\infty$, then \(q(xy) \vdash \phi(xy;e)\) if and only if \(p(xy)\vdash \points{F_{\lambda_0}}(x) \to \phi(xy;e)\) --- cf. \cite[Proposition~6.6]{Rid-VDF}. If $p(xy)\vdash \rad(F_{\lambda_0}(x)) \neq +\infty$, let \[Y_s := \{\mu\in\Lambda_p \mid \fa{p}xy\ F_{\mu}(x)\subseteq \res_{F_\lambda(x)} \wedge \ex{p}xy\ \phi(xy;s)\wedge y\in\points{F_{\mu}}(x)\}.\] Let \(n\) be a uniform bound on the cardinality of finite \(\germ{p}{Y_s}\), as in \UFresgloc{p}{F}.

\begin{claim}\label{quanti cl ball}
For every \(e\in M^s\), \(q\) is consistent with \(\phi(xy;e)\) if and only if, for every \((\mu_i)_{i< n}\in\Lambda_p\), with \(\fa{p}xy F_{\mu_i}(x) \in \res_{F_{\lambda_0(x)}}\), \(\ex{p}xy\ \phi(xy;e)\wedge y\in\points{F_{\lambda_0}}(x)\sminus\bigcup_{i < n} \points{F_{\mu_i}}(x)\).
\end{claim}

\begin{proof}
Assume \(q\) is not consistent with \(\phi(xy;e)\), then, by compactness, there exists \((\mu_i)_{i< m} \in\Lambda_p\) such that \(\mu_i <_p \cE\) and \(\fa{p}xy\ \phi(xy;e)\wedge y\in\points{F_{\lambda_0}}(x)\to\bigvee_{i<m} y\in\points{F_{\mu_i}}(x)\). By the large ball property, we may assume \(\fa{p}xy\ F_{\mu_i}(x) \in \res_{F_{\lambda_0(x)}}\). Choosing a minimal \(m\), we may also assume that, \(\ex{p}xy\ \phi(xy;e)\wedge y\in\points{F_{\mu_i}}(x)\). In particular, \(\mu_i\in Y_e\).

By definition of \(Y_s\), for every \(\mu\in Y_e(M)\), we find \(ac\models p\) such that \(\phi(ac;e)\) and \(c\in \points{F_\mu}(a) \subseteq \points{F_{\lambda_0}}(a)\). So there is an \(i\) such that \(c\in \points{F_{\mu_i}}(a)\). By irreducibility, \(F_{\mu}(a) = F_{\mu_i}(a)\). It follows that \(\germ{p}{Y_e}\) is finite and thus \(m\leq \card{\germ{p}{Y_e}} \leq n\).
\end{proof}

Since \(q \vdash \phi(xy;e)\) if and only if \(q\) is not consistent with \(\neg\phi(xy;e)\), \cref{quanti cl ball} allows us to conclude.
\end{proof}

We also need a better understanding of the interpretable set \(\germ{p}{F_\Lambda}\). Note that it is, a priori, \(\cL(M)\)-interpretable which is exactly the kind of sets elimination of imaginaries aims at describing. However, if \(p\) happens to be the restriction to \(M\) of a global \(\cL_0(M)\)-definable type, then \(\germ{p}{F_\Lambda}\) naturally embeds in an \(\cL_0(M)\)-interpretable set. The goal of the following lemmas is to give (necessary) hypotheses under which any definable \(p\) verifies that condition. Valued vector spaces will play an important role:

\begin{definition}
Let \((K,\val)\) be a valued field and \(V\) be a \(K\)-vector space. A \emph{valuation on \(V\)} is a map \(v: V\to X\) where \(X\) is an ordered set with a maximal element \(\infty\) and an action \(+\) of \(\Gamma\), respecting the order, such that:
\begin{itemize}
\item \(v(0) = \infty\);
\item For all \(x,y\in V\), \(v(x + y) \geq \min \{v(x),v(y)\}\);
\item For all \(a\in K\) and \(x\in V\), \(v(a\cdot x) = \val(a) + v(x)\).
\end{itemize}

We say that a family \((x_i)_{i\in I} \in V\) is \emph{separating} if for every
finite \(I_0\subseteq I\) and every \((a_i)_{i\in I_0} \in K\), \(v(\sum_{i\in
I_0} a_i x_i) = \min_{i\in I_0} (\val(a_i) + v(x_i))\).
\end{definition}

The following lemma owes much to Johnson's computation of the canonical basis of
definable types in \(\ACVF\), \emph{cf.}\ \cite[Section~5.2]{Joh-EIACVF}. Let
\(\phi_d(x;yz) := \val(\sum_{|I|<d} y_I x^I) \geq \val(\sum_{|I|<d} z_I x^I)\).

\begin{proposition}\label{compl alg}
Assume:
\begin{itemize}[leftmargin=50pt]
\item[\CV] For every \(n\in \Zz_{\geq 1}\), every \(\cL(M)\)-definable valuation \(v\) on \(\K^n\) has a separating basis;
\item[\Cvg] \(T\) has definably complete value group.
\end{itemize}
Then, for every \(A = \dcleq(A)\subseteq \eq{M}\), \(\cL(A)\)-definable \(p\in\TP_{\phi_d}(M)\) and algebraic extension \(K(M)\leq L\), any \(q\in \TP_{\phi_d}(L)\), extending \(p\) and finitely satisfiable in \(M\!\), is \(\cL_0(\Geom(A))\)-definable.
\end{proposition}

\begin{proof}
For every field $F$ with \(K = \K(M)\leq F\leq L\), we define a valuation \(v\) on the \(F\)-vector space \(V_d(F) := F[x]_{\leq d}\) of polynomials in variables \(x\) over \(F\) of degree at most \(d\) by \(v(P(x)) \leq v(Q(x))\) if \(\val(P(x))\leq \val(Q(x)) \in q\). The valuation \(v\) on \(V_d(K)\) is \(\cL(A)\)-definable. By hypothesis \CV, it has a separating basis \((P_i)_i \in V_d(K)\).

\begin{claim}\label{sep alg}
\((P_i)_i\) is a separating basis of \(V_d(L)\) over \(L\).
\end{claim}

\begin{proof}
We may assume that \(K\leq L\) is finite. By \cite[Remark\ 2.7]{Joh-DpFinI}, the
valuation on \(L\) (interpreted in \(K\)) is then \(\cL(M)\)-definable and
hence, by hypothesis \CV, also has a separating basis \((c_j)_j\in L\) over
\(K\). Let \(b_ i = \sum_j b_{i,j} c_j \in L\), where \(b_{i,j}\in K\). If
\(v(\sum_{j} (\sum_i b_{i,j}P_i)c_j) >  \min_j v(\sum_i b_{i,j}P_i) +
\val(c_j)\), since \(q\) is finitely satisfiable in \(M\), we have \(v(\sum_{j}
(\sum_i b_{i,j}P_i(a))c_j) >  \min_j v(\sum_i b_{i,j}P_i(a)) + \val(c_j)\) for
some \(a\in M\), contradicting that \((c_j)_j\) is separating over \(K\). So
\begin{align*}
v(\sum_i b_i P_i) &= v(\sum_{j} (\sum_i b_{i,j}P_i)c_j)\\
&= \min_j v(\sum_i b_{i,j}P_i) + \val(c_j)\\
&= \min_{i,j} \val(b_{i,j}) + \val(c_j) + v(P_i)\\
&\leq \min_i \val(b_i) + v(P_i)\\
&\leq v(\sum_i b_i P_i) \qedhere
\end{align*}
\end{proof}

We now define the \(L\)-Archimedean equivalence on \(v(V_d(L))\): \(v(P)
\sim_L^\infty v(Q)\) holds if there exists \(c\in L^\times\), with \(\val(x)\geq
0\), such that \(- \val(c) + v(P) \leq v(Q) \leq \val(c) + v(P)\). One can check
that \(\card{v(V_d(L))/\sim_L^\infty} \leq \card{v(V_d(L))/\val(L)} \leq
\dim_L(V_d(L)) + 1 < \infty\). We also define the \(L\)-infinitesimal
equivalence on \(v(V_d(L))\): \(v(P)\sim_L^0 v(Q)\)  holds if for every \(\gamma
\in \val(L)_{>0}\) we have \(- \gamma + v(P) < v(Q) < \gamma + v(P)\). Note that
two elements of the same \(\val(L)\)-orbit cannot be \(\sim^0_L\)-equivalent
unless they are equal. It follows that \(\sim_L^0\)-classes are finite.

Let \(C\) be any \(K\)-Archimedean class and \(\overline{C}\) denote its upwards
closure. Then \(V_C := v^{-1}(\overline{C}) \leq V_d(K)\) is an
\(\cL(A)\)-definable sub-\(K\)-vector space.

\begin{claim}[{\cite[Lemma~4.3]{Joh-EIACVF}}]
\(V_C\) has a basis of elements in \(A\). 
\end{claim}

\begin{proof}
Some coordinate projection \(V_C \subseteq K^l \to K^m\) restricts to an
isomorphism on \(V_C\). The preimage of the standard basis of \(K^m\) then has
the required properties.
\end{proof}

Let \(C_0\) be the successor of \(C\) in \(V_d(K)/\sim_K^\infty\). Since
\(V_{C_0}\subset V_C\), any basis of \(V_C\) has an element outside \(V_{C_0}\).
In particular \((V_C \sminus V_{C_0})(A)\neq\emptyset\) and we find
\(\gamma_C\in C(A) \neq\emptyset\). Then the whole (finite) \(K\)-infinitesimal
class of \(\gamma_C\) is in \(A\). Let \(i\) be such that \(v(P_i) \in C\). By
\Cvg, the set \(\{\gamma\in\val(K)\mid \gamma + v(P_i) \leq \gamma_C\}\) has a
supremum \(\gamma_i\). Multiplying \(P_i\) by some constant \(c\in K\) with
\(\val(c) = -\gamma_i\), we may assume that \(\gamma_i = 0\) in which case
\(v(P_i)\) is \(K\)-infinitesimally close to \(\gamma_C \in A\). Since the
\(K\)-infinitesimal class of \(\gamma_C\) is finite, it follows that \(v(P_i)\)
is also in \(A\). Since every \(\val(K)\)-orbit is contained in some
\(K\)-Archimedean class, we now have that for any \(i\), \(v(P_i) \in A\) and
for any \(j\), if \(v(P_i) \sim_K^\infty v(P_j)\), then \(v(P_i) \sim_K^0
v(P_j)\).

Note that, since \(\val(L)\) is in the convex hull of \(\val(K)\), \(\sim_{L}^\infty\) extends \(\sim_K^\infty\). Also, if \(v(K)\) is dense, then \(\sim_{L}^0\) extends \(\sim_{K}^0\). However, if \(\val(K)\) is discrete then \(\sim_K^0\) reduces to equality. In particular, we also have that, if \(v(P_i) \sim_{L}^\infty v(P_j)\), then \(v(P_i) \sim_{L}^0 v(P_j)\).

For every \(i\), \(d\), let \(M_{i,d}(L) = \{P\in V_d(L) \mid v(P) \geq v(P_i)\}\). Note that, by \cref{sep alg}, \(\sum \lambda_j P_j \in M_{i,d}(L)\) if and only if:
\begin{itemize}
\item \(\lambda_j = 0\) for every \(j\) with \(v(P_j) < v(P_i)\) and \(v(P_j) \nsim_K^\infty v(P_i)\);
\item \(\lambda_j \in \Mid\) for \(j\) with \(v(P_j) < v(P_i)\) and \(v(P_j) \sim_K^0 v(P_i)\);
\item \(\lambda_j \in \Val\) for \(j\) with \(v(P_i) \leq v(P_j)\) and \(v(P_j) \sim_K^0 v(P_i)\).
\end{itemize}
So \(M_{i,d}\) is (quantifier free) \(\cL_0(\Geom(A))\)-definable. Since \(q\) is \(\cL_0(\bigcup_{i,d}\code{M_{i,d}})\)-definable, it is indeed \(\cL_0(\Geom(A))\)-definable.
\end{proof}


If \(p\in\TP^0(M)\) the existence (and uniqueness) of such a \(q\) follows, on general grounds, from the finite satisfiability of \(p\):

\begin{lemma}\label{compl stationary}
Let \(p\in\TP^0(M)\) be finitely satisfiable in \(M\). Then any two realisations of \(p\) have the same \(\cL_0(\acl_0(M))\)-type. In particular, the unique extension of \(p\) to \(\acl_0(M)\) is finitely satisfiable in \(M\).
\end{lemma}

\begin{proof}
Fix any \(c\in \acl_0(M)\), \(\phi(xy)\) an \(\cL_0\)-formula and \(\psi(y)\) an \(\cL_0(M)\)-formula that algebrises \(c\). Then \(\forall y\ [\psi(y) \to (\phi(x_1y) \leftrightarrow \phi(x_2y)\))] defines an \(\cL_0(M)\)-definable equivalence relation with finitely many classes. Then the \(E\)-class of any \(a \in N\supsel M\) realising \(p\) has an element \(e\in M\). It follows that \(p(x)\vdash x E e\). In particular \(p(x) \vdash \phi(xc)\) whenever \(\phi(ec)\) holds.

Let \(q\) be the unique extension of \(p\) to \(\acl_0(M)\). Then \(p\vdash q\) and hence \(q\) is finitely satisfiable in \(M\).
\end{proof}

Following \cite{Bau-Sep}, we can prove that {\CV} follows from definable spherical completeness. Up to definability, this is a standard result. But we include its proof, on the one hand, for the sake of completeness, and, on the other, to show that the proof can indeed been done definably.

\begin{lemma}
Assume:
\begin{itemize}[leftmargin=50pt]
\item[\Cball] \(T\) is definably spherically complete : any \(\cL(M)\)-definable chain of balls has a non-empty intersection.
\end{itemize}
Then any (finite dimensional) \(\cL(M)\)-interpretable valued \(\K\)-vector space \((V,v)\) has a separating basis.
\end{lemma}

\begin{proof}
Let us proceed by induction on \(n+1 := \dim(V)\).  In particular, we may assume that we have found a separating family \((y_i)_{0\leq i< n} \in V\).

\begin{claim}\label{sep max}
For every \(x \in V\), \(\{v(x - \lambda y) \mid \lambda\in \K^{n}\}\) has a maximal element.
\end{claim}

\begin{proof}
For every \(\lambda,\mu\in \K^n\), we have \(v(x - \mu y) \geq v(x - \lambda y) =: \gamma\) if and only if \(\min_i \{\val(\mu_i - \lambda_i) + v(y_i)\} = v((\mu-\lambda)y) \geq \gamma\). For every \(i < n\) and \(\lambda \in \K^{n}\), let \(B_{i,\lambda} := \{\mu\in\K\mid v(x-\lambda_{\neq i} y_{\neq i} - \mu y_i) \geq v(x-\lambda y)\} = \{\mu\in\K\mid\val(\mu-\lambda_i) + v(y_i)\geq v(x - \lambda y)\}\). They form a chain for inclusion.

If there is a minimal \(B_{i,\lambda_0}\), pick any \(\lambda_i \in B_{i,\lambda_0}\). If there is no minimal \(B_{i,\lambda}\), for every \(\gamma \in \val(\K)\), let \(b_{i,\gamma}\) be the closed ball of radius \(\gamma\) containing some \(B_{i,\lambda}\), if it exists, or \(\K\) otherwise. Since the chain of \(B_{i,\lambda}\) does not have a minimal element, any \(B_{i,\lambda}\) contains a \(b_{i,\gamma}\) that itself contains a \(B_{i,\mu}\). By definable spherical completeness, we find \(\lambda_i \in \bigcap_\gamma b_{i,\gamma} = \bigcap_{\lambda} B_{i,\lambda}\). Then \(\lambda_i\) has the property that, for any \(\mu \in \K^n\), \(v(x-\mu y)\leq v(x - \mu_{\neq i} y_{\neq_i} - \lambda_i y_i)\). It follows that \(v(x - \lambda y)\) is maximal.
\end{proof}

Let \(x\) be linearly independent from the \(y_i\). By the \cref{sep max}, we may assume that \(v(x) = \max\{v(x - \lambda y) \mid \lambda\in \K^n\}\). Then, for every \(\mu\in\K\) and \(\lambda\in\K^n\), we have \(v(\mu x + \lambda y) \leq v(\mu x)\) and thus, \(v(\mu x + \lambda y) = \min\{v(\mu x),v(\lambda y)\} = \min_i\{v(\mu x),v(\lambda_i y_i)\}\).
\end{proof}

\begin{remark}
Note that given any basis of \(V\), in the above lemma, we actually construct a
separating basis whose base change is upper triangular.
\end{remark}

\subsection{Counting germs}
The last ingredient in this section is to reduce the (seemingly horrendous) hypotheses \UFresgloc{p}{F} and \UFradgloc{p}{F} to something more tractable. We start by generalizing \cite[Section~7]{Rid-VDF}. Let \(M\models T\), \(M_0:=\alg{M}\models T_0\) and \(Q\) be either \(\vg\) or \(\res\).

\begin{lemma}\label{reparam}
Let \((f_\lambda)_{\lambda\in\Lambda} : \K^n \to Q\) be an \(\cL_0\)-definable family and \(c\in \K(N_0)^n\), where \(N_0\supsel M_0\). There exists an \(\cL_0(M)\)-definable family \((g_{\rho})_{\rho\in R}:\K^n\to Q^{[<\infty]}\), with \(R\subseteq Q^m\), such that for all $\lambda\in\Lambda(M_0)$, there exists $\rho\in R(M_0)$ with $f_\lambda(c)\in g_{\rho}(c)$.

In particular, if \(p\in\TP^0(M_0)\) is \(\cL_0(M)\)-definable, there is a \(\cL_0(M)\)-definable finite-to-one map \(\{\germ{p}{f_\lambda} \mid \lambda\in\Lambda(M_0)\}\to Q^m\), for some \(m\in \Zz_{>0}\).
\end{lemma}

\begin{proof}We start with the non-uniform version of the result:

\begin{claim}
\label{gen res}
For every \(N_0\models \ACVF\), \(A\leq\K(N_0)\) and finite tuple \(c\in
\K(N_0)\), there exists a finite tuple \(a\in A\), such that \(Q(\acl_0(Ac))
\subseteq \acl_0(Q(A)ac)\).
\end{claim}

\begin{proof}
If \(\card{c} = 1\), let \(a_0\in \K(\acl_0(A))\) be such that \(\val(c - a_0)\)
is maximal --- if it exists, otherwise the extension \(A\leq A(c)\) is immediate
and we take \(a_0 = 0\). Then \(\RV(A(c)) \subseteq \dcl_0(\rv(A)\rv(c-a_0))\).
It follows that \(Q(\acl_0(Ac)) = \acl_0(Q(A(c))) \subseteq
\acl_0(Q(\acl_0(A))ca_0) = \acl_0(Q(A)c a_0)\). If \(a\in A\) is such that
\(a_0\in\acl_0(a)\), we indeed have \(Q(\acl_0(Ac)) \subseteq \acl_0(Q(A)ac)\).

If \(c = de\) with \(\card{e}=1\), we proceed by induction:
\[Q(\acl_0(Ade))\subseteq \acl_0(Q(\acl_0(Ad))be)\subseteq \acl_0(Q(A)acbe),\]
with \(a,b \in A\).
\end{proof}

By \cref{gen res}, and compactness in a saturated model of the pair
\((N_0,M_0)\), there exists an \(\cL_0(M_0)\)-definable \(g\) as above. The
union of its conjugates over \(M\) has the same properties and is
\(\cL_0(M)\)-definable.

Now, if \(p\in\TP^0(M_0)\) is \(\cL_0(M)\)-definable, then for any
\(\lambda\in\Lambda(M_0)\), let \(Y_{\lambda} :=  \{\rho \mid \fa{p}{x}\
f_\lambda(x) \in g_\rho(x)\}\) and \(h(\germ{p}{f_\lambda}) :=
\code{Y_\lambda}\in Q^m\). Note that \(p(x)\models f_\mu(x) \in \bigcap_{\rho\in
Y_\mu} g_\rho(x)\) which is a finite set. It follows that there are at most
finitely many germs \(\germ{p}{f_\mu}\) associated to a given \(Y_\lambda\), in
other words, \(h\) is finite-to-one.
\end{proof}

\begin{lemma}\label{res to resg}
Assume:
\begin{itemize}[leftmargin=50pt]
\item[\UFres] For any \(\cL(M)\)-definable \((Y_z)_z \subseteq \res\),  there exists \(n\in\Zz_{\geq0}\) such that, for all \(z\), \(\card{Y_z}<\infty\) implies \(\card{Y_z} \leq n\);
\end{itemize}
Then, for every \(\cL_0(M)\)-definable \(p\in \TP^{0}_x(M_0)\) and \(\cL_0\)-definable \((F_{\lambda})_{\lambda\in\Lambda} : \K^x \to \fballs{r}\), \UFresgloc{p}{F} holds, uniformly in \(\lambda\).
\end{lemma}

\begin{proof} The core of the proof is the following almost internality result:

\begin{claim}\label{alm int} For every \(\lambda\in\Lambda(M)\), there exists an
\(\cL_0(M)\)-definable finite-to-one map \(g_\lambda : X_\lambda :=
\{\germ{p}{F_\mu}\mid \mu\in\Lambda(M_0)\) and \(p(x) \vdash\) "\(F_\mu(x)\) are
maximal open balls of \(F_{\lambda}(x)\)"\(\} \to \res^m\), for some
\(m\in\Zz_{>0}\).
\end{claim}

\begin{proof}
Let $c\models p$. In $\alg{M(c)} \models \ACVF$, any closed ball of
\(F_\lambda(c)\) has at least two (infinitely many, in fact) distinct maximal
open subballs. So there exists \(G_i(c) \in  \fballs{<\infty}\), for \(i :=
1,2\), two \(\cL_0(Mc)\)-definable sets picking at least one maximal open balls
in each of the ball of \(F_\lambda(c)\) and such that \(G_1(c)\cap G_2(c) =
\emptyset\). For every \(\mu\) with \(\germ{p}{F_\mu} \in X_\lambda\), let
\(f_\mu(x) := \{(b-b_1)/(b_2 - b_1) \mid b\in F_\mu(x),\ b_i \in G_i(x)\) and
\(b,b_1,b_2\) are in the same ball of \(F_\lambda(x)\} \in \fres{<\infty}\).
Note that \(F_\mu(c) \in\acl_0(f_\mu(c)G_1(c)G_2(c))\). Using symmetric
functions, we identify \(\fres{<\infty}\) with some \(\res^n\).

By \cref{reparam}, we find an \(\cL_0(M)\)-definable finite to one map \(h : \{\germ{p}{F_\mu} \mid \germ{p}{F_\mu}\in X_\lambda\}\to \res^m\). Then \(\germ{p}{F_\mu}\in\acl_0(\germ{p}{G_1}\germ{p}{G_1}\germ{p}{f_\mu})\subseteq\acl_0(Mh(\mu))\) and \(h(\mu)\in\dcl_0(M\germ{p}{f_\mu})\subseteq\dcl_0(M\germ{p}{F_\mu})\).
\end{proof}

Since \(\res(\dcl_0(M))\) is the perfect closure of \(\res(M)\), by compactness, composing with a power of the Frobenius automorphism, we may assume that \(g_\lambda(X_\lambda(M)) \subseteq \res(M)\) and that \(g_\lambda\) in uniform in \(\lambda\). The (uniform) bound in \UFresgloc{p}{F} now follows from \UFres.
\end{proof}

\begin{lemma}\label{UFballs}
Assume {\UFres} and:
\begin{itemize}[leftmargin=50pt]
\item[\UFvg] For any \(\cL(M)\)-definable \((Y_z)_z \subseteq \vg\),  there exists \(n\in\Zz_{\geq 0}\) such that, for all \(z\), \(\card{Y_z}<\infty\) implies \(\card{Y_z} \leq n\).
\end{itemize}
Then for every \(\cL_0(M)\)-definable \(p\in\TP_x^{0}(M_0)\), \(\cL_0\)-definable \((F_\lambda)_{\lambda\in\Lambda}: \K^x \to \fballs{r}\), and \(\cL(M)\)-definable \((Y_z)_{z} \subseteq \germ{p}{F_\Lambda}\), there exists \(n\in\Zz_{>0}\) such that \(\card{Y_z}<\infty\) implies \(\card{Y_z} \leq n\).
\end{lemma}

In particular, {\UFradgloc{p}{F}} holds.

\begin{proof}
Let \(r_\lambda(x) := \rad(F_\lambda(x))\). By \cref{reparam}, there exists an \(\cL_0(M)\)-definable finite-to-one map \(g: \germ{p}{r_\Lambda} \to\Gamma^m\). Composing by division by a fixed integer, we may assume that \(g(\germ{p}{{r_\Lambda}}(M)) \subseteq \Gamma(M)\). It now follows that there is a bound on finite \(\{\germ{p}{r_\lambda} \mid \germ{p}{F_\lambda} \in Y_z\}\). So, cutting each \(Y_z\) in finitely many pieces (and getting rid of the infinite ones), we may assume that \(\germ{p}{\rad(F_\lambda)}\) is constant and the balls are of the same type, as \(\germ{p}{F_\lambda}\) ranges through \(Y_z\). Similarly, we may assume that the set of distances between balls in \(F_\lambda(x)\) and \(F_\mu(x)\), with \(\germ{p}{F_\lambda}, \germ{p}{F_\lambda} \in Y_z\) has size bounded by some integer \(k\). We now proceed by induction on \(k\).

Let \(\gamma_z(x)\) be the smallest such distance, \(G_{\lambda,z}(x)\) be the set of closed balls of radius \(\gamma_z(x)\) around \(F_\lambda(x)\) and \(Z_{z} := \{\germ{p}{G_{\lambda,z}}\mid \germ{p}{F_\lambda}\in Y_z\}\). Then the set of distances between balls in \(Z_z\) has size at most \(k-1\) and we find a bound by induction. In particular, removing some more infinite \(Y_z\), we find an \(H_z : \K^n \to\fballs{<\infty}\) such that every maximal open ball of \(H_z(x)\) contains at most one ball of \(F_\lambda(x)\) as \(\germ{p}{F_\lambda}\) varies through \(Y_z\). The bound now follows from \cref{res to resg}.
\end{proof}

\subsection{The higher arity case}
We can now proceed with the induction:

\begin{proposition}\label{str def dens}
Assume {\CV}, {\Cvg}, \UFres{} and \UFvg{}. Let \(X\subseteq\K^x\) be strict pro-\(\cL(A)\)-definable, where \(A = \acleq(A)\subseteq \eq{M}\) and $x$ is countable. Let \(\Delta(x;t)\) be a finite set of \(\cL_0\)-formulas, \(p\in\TP_x^\Delta(M)\) be \(\cL(A)\)-quantifiable over \(\cL\) and consistent with \(X\) and \(z\subseteq x\). Then, there exists an \(\cL_0(\Geom(A))\)-definable \(q\in\TP^0_z(M_0)\) such that \(\restr{q}{M}\) is consistent with \(p\) and \(X\).
\end{proposition}

\begin{proof}
We proceed by induction on \(\card{z}\). In particular, we may assume that for any set \(\Delta(x;t)\) of \(\cL_0\)-formulas, \(p\in\TP_x^\Delta(M)\) which is \(\cL(A)\)-quantifiable over \(\cL\) and finite strictly smaller \(w\subset z\), \(\restr{p}{w}\) can be extended to an \(\cL_0(\Geom(A))\)-definable \(q \in \TP^0_w(M_0)\).

\begin{claim}\label{1 form}
Let \(\Delta(x;t)\) be a finite set of \(\cL_0\)-formulas, \(p\in\TP_x^\Delta(M)\) be \(\cL(A)\)-quantifiable over \(\cL\) and consistent with \(X\), finite \(z\subseteq x\) and \(\Phi(z;s)\) be a finite set of \(\cL_0\)-formulas. Then, there exists a finite set \(\Theta(z;t)\) containing \(\Phi\) and \(q\in\TP^{\Delta,\Theta}_x(M)\) which is \(\cL(A)\)-quantifiable over \(\cL\) and consistent with \(p\) and \(X\).
\end{claim}

\begin{proof}
We proceed by induction on \(\card{z}\). Assume \(z = w y\) with \(\card{y} =
1\) (where \(w\) might be the empty tuple). By \cref{Ex GP}, we find a finite
good presentation \((\Psi(w;t),F(w))\) for \(\Phi\). By induction, we find
\(\Xi(w;u)\supseteq \Psi\) and \(q\in\TP_{x}^{\Delta,\Xi}(M)\)  which is
\(\cL(A)\)-quantifiable over \(\cL\) and consistent with \(p\) and \(X\). Since
\(w\subset z\), as stated in the first paragraph of the proof, \(\restr{q}{w}\)
extends to a complete \(\cL_0(\Geom(A))\)-definable \(\cL_0(M_0)\)-type.

By \cref{rel 1} and \cref{UFballs}, we now find an \(\cL(A)\)-definable \(r\in \TP_{x}^{\Delta,\Xi,F}(M)\) which is consistent with \(q\) and \(X\). By \cref{quant,res to resg}, \(r\) is \(\cL(A)\)-quantifiable over \(\cL\).
\end{proof}

Let \((\phi_i(z;t_i))_{i\in\omega}\) enumerate all \(\cL_0\)-formulas. By \cref{1 form}, we find \(\Theta_i\) containing \(\phi_i\) and \(q_i\in\TP^{\Delta,\Theta_{\leq i}}_{z}(M)\), which is \(\cL(A)\)-quantifiable over \(\cL\) and consistent with \(p\cup \bigcup_{j<i} q_j\) and \(X\). Then \(\bigcup_i q_i\in\TP^0_z(M)\) is \(\cL(A)\)-definable and consistent with \(p\) and \(X\). By \cref{compl alg} and \cref{compl stationary}, \(q\) extends to a complete \(\cL_0(\Geom(A))\)-definable \(\cL_0(M_0)\)-type.
\end{proof}

This result is already non-trivial when \(X = \K^x\) and \(T = \ACVF\):

\begin{corollary}\label{def compl}
Let \(\Psi(x;t)\) be a set of \(\cL_0\)-formulas and \(A = \acl_0(A) \leq M_0\). Any \(\cL_0(A)\)-quantifiable \(p\in\TP_x^\Psi(M_0)\) can be extended to an \(\cL_0(A)\)-definable \(q\in \TP^0_x(M_0)\).\qed
\end{corollary}

If we do not assume \UFvg, and try to replace the use of \cref{rel 1} by that of \cref{rel 1 H}, the above induction fails. We can, nevertheless, recover a local version of the result:

\begin{proposition}\label{fin dens H}
Let \(n\in \Zz_{>0}\cap \K^\times(M)\) and \(X\subseteq\K^x\) be \(\cL(A)\)-definable, where \(A = \acleq(A)\subseteq \eq{M}\). Assume {\CV} in both \(M\) and the pair \((M_0,M)\), {\Cvg}, {\UFres}, \Infres, \Ram[n]. Also assume that:
\begin{itemize}[leftmargin=50pt]
\item[{\Prep{\pi(X)}{n}}] For every projection \(Y\subseteq \K^{zy}\) of \(X\), with \(\card{y} = 1\), \(N\supsel M\) and \(a\in N^y\), \(Y_a\) is \(n\)-prepared by some finite \(\cL_0(Ma)\)-definable set \(C\subseteq\K\).
\end{itemize}
Then, for every finite set \(\Psi(x,t)\) of  \(\cL_0\)-formulas, there exists an \(\cL_0(\Geom(A))\)-definable \(p\in\TP_x^\Psi(M)\) consistent with \(X\).
\end{proposition}

\begin{proof}
Let \(A_P = \acleq_{\LP}(A)\). We say that \(\Theta(zy,s)\), where \(|y|=1\), is a hereditarily good presentation if \(\Theta\) is of the form \(\Phi(z,t)\cup \{y\in F_\lambda(z)\}\) for some good presentation \((\Phi,F)\) where \(\Phi\) is itself an hereditarily good presentation. 

\begin{claim}\label{hgp quant}
Let \(\Theta(x,s)\) be a hereditarily good presentation and \(q\in\TP_x^\Theta(M_0)\) be \(\cL_0(\Geom(A_P))\)-definable. Then for \(\cL \in\{\cL_0,\LP\}\), \(q\) is \(\cL(\Geom(A_P))\)-quantifiable over \(\cL\) --- in particular it has a complete \(\cL_0(\Geom(A_P))\)-definable extension to \(\TP^0_x(M_0)\).
\end{claim}

\begin{proof}
We proceed by induction on \(|x|\). Since \UFres{} holds both in \(M_0\) and, by \cref{UFres pair}, in the pair \((M_0,M)\), quantifiability follows from \cref{quant,res to resg} --- applied respectively to \(M_0\) and to the pair \((M_0,M)\). The existence of a complete definable extension follows by \cref{def compl} applied in \(M_0\).
\end{proof}

We now prove, by induction on \(x = zy\), the existence of \(p\in\TP_x^\Psi(M_0)\) which is \(\cL_0(\Geom(A_P))\)-definable and consistent with \(X\). By compactness, there exists \((G_\omega)_{\omega\in\Omega} : \K^z \to\fpoints{<\infty}\) such that the family \((X_z)_z\) is \(n\)-prepared by \(G\). Let \(d\in\Zz_{>0}\) bound the degree of any polynomial appearing in \(\Psi\) and \(G\). By \cref{Ex GP}, and induction, we find a finite hereditarily good presentation \((\Theta(z,s),F(z))\) for \(\phi_d(x,uv) := \val(\sum_{|I|<d} u_I x^I) \geq \val(\sum_{|I|<d} v_I x^I)\). By induction, there exists \(q\in\TP_z^\Theta(M_0)\) which is \(\cL_0(\Geom(A_P))\)-definable and consistent with \(X\). By \cref{hgp quant}, \(q\) is \(\LP(\Geom(A_P))\)-quantifiable over \(\LP\). By \cref{rel 1 H}, there exists an \(\LP(A_P)\)-definable \(p\in\TP^{\Theta,F}_x(M_0)\) consistent with \(X\). By hypothesis, \CV{} holds in \((M_0,M)\), and so does \Cvg{},  by \cref{Cvg pair}. By \cref{compl alg}, it follows that \(\restr{p}{\phi_d}\) is \(\cL_0(\Geom(A_P))\)-definable --- and hence so is \(\restr{p}{\Psi}\).

Let \(a\in A_P\) be the canonical basis of \(\restr{p}{\phi_d}\). Since \(M_0 = \alg{M} \subseteq \acl_0(M)\), we have  \(a\in\acl_0(M)\). Let \(c\in\dcl_0(M)\) be a code of the finite \(\cL_0(M)\)-orbit of \(a\) --- which is included in its finite \(\LP(A)\)-orbit. Let \(f\) be \(\cL_0\)-definable such that \(c\in f(M)\) and \(e\in \eq{M}\) be a code of \(f^{-1}(c)\). The \(\LP(A)\)-orbit of \(c\) consists of finite subsets of the \(\LP(A)\)-orbit of \(a\) and is therefore finite. Hence, so is the \(\cL(A)\)-orbit of \(e\); \emph{i.e.}, \(e\in \acleq(A) = A\). It follows that \(\restr{p}{\phi_d,M} \subseteq \bigcap_{\sigma\in\aut(M_0/M)}\sigma(\restr{p}{\phi_d})\) is \(\cL(A)\)-definable. By \cref{compl alg}, it is in fact \(\cL_0(\Geom(A))\)-definable — and hence so is \(\restr{p}{\Psi,M}\).
\end{proof}

This local result does imply the existence of a global invariant type:

\begin{corollary}\label{inv dens H} Let \(n\in \Zz_{>0}\cap \K^\times(M)\) and
\(X\subseteq\K^x\) be strict pro-\(\cL(A)\)-definable, where \(A =
\acleq(A)\subseteq \eq{M}\). Assume {\CV} in both \(M\) and the pair
\((M_0,M)\), {\Cvg}, {\UFres}, \Infres, \Ram[n] and:
\begin{itemize}[leftmargin=50pt]
\item[{\Prep{\pi(X)}{n}}] For every projection \(Y\subseteq \K^{zy}\) of \(X\)
onto finitely many coordinates, with \(\card{y} = 1\), every \(N\supsel M\) and
every \(a\in\K^{z}(N)\), \(Y_a\) is \(n^\ell\)-prepared by some
\(\cL_0(Ma)\)-definable set \(C\subseteq \K(N)\), for some \(\ell\in\Zz_{\geq
0}\).
\end{itemize}
Then there exists an \(\aut(M/\Geom(A))\)-invariant \(p\in\TP^0_x(M)\) consistent with \(X\).
\end{corollary}

Note that \(\Ram[n]\) implies \(\Ram[n^\ell]\) for every \(\ell\in\Zz_{\geq 0}\). Also if \(M\) is a finitely ramified henselian field, \CV{} holds in both \(M\) and \((M_0,M)\), since both theories have maximally complete models, and \Prep{\pi(X)}{n} holds by \cref{Hen0 prep}.

\begin{proof}
For every (finite) set \(\Psi(x,t)\) of \(\cL_0\)-formulas, the set of \(p\in\TP^0_x(M)\) which are consistent with \(X\) and whose \(\Psi\)-type is \(\aut(M/\Geom(A))\)-invariant is closed. It is non-empty, by \cref{fin dens H}. By compactness, the intersection of all these sets, which coincides with the set of \(\aut(M/\Geom(A))\)-invariant \(p\in\TP^0_x(M)\) consistent with \(X\), is also non-empty.
\end{proof}

\begin{remark}
If \(T\) is a \(\res\)-\(\vg\)-enrichment of finitely ramified henselian fields,
then, by \cref{res-vg pure P}, the pair \((M_0,M)\) is elementarily equivalent
to one where both \(\K\) and \(\PK\) are maximally complete --- namely the pair
of the maximal completion of \(M\) inside the maximal completion of its
algebraic closure. Hence {\Cball} --- and therefore {\CV}--- holds both in \(M\)
and in the pair \((M_0,M)\).
\end{remark}

\section{Invariant completions}\label{invariant}

\begin{notation}
In this section, let \(T\) be an \(\RV\)-enrichment of the theory of characteristic zero henselian fields.
\end{notation}

\subsection{Main results}
Our goal in this section is to describe the behaviour of global \(\cL\)-types whose underlying \(\cL_0\)-type is invariant. A crucial point is that \cref{EQ rv} can be reformulated in the following manner: for every \(A\substr M\models T\),
\[\qftp(A)\cup\tp(\RV(A))\vdash\tp(A).\]

Therefore, the main point of this section is to better understand $\tp(\RV(A))$ and then deduce properties of $\tp(A)$. In particular, we will show that $\RV(A)$ is generated by a small canonical set. This will allow us to conclude that a global type whose underlying quantifier free type is invariant is itself invariant over \(\RV\) (\emph{cf.} \cref{RV inv}). However, a better control of the parameters requires more auxiliary sorts. Recall that \[\Lin_A := \bigsqcup_{\substack{s\in\Lat(\dcl_0(A))\\ \ell\in\Zz_{>0}}} s/\ell\Mid s.\]

In this section, we prove:

\begin{theorem}\label{inv}
Assume
\begin{itemize}[leftmargin=50pt]
\item[\Infres] The residue field \(\res\) is infinite.
\end{itemize}
Let \(M\subsel N\models T\) sufficiently saturated and homogeneous, \(A \subseteq\Geom(M)\) and \(a\in\K(N)\) such that \(\qftp(a/M)\) is \(\aut(M/A)\)-invariant. Then the type \(\tp(a/M)\) is \(\aut(M/A\RV(M)\Lin_A(M))\)-invariant.
\end{theorem}

\subsection{Invariance and stably embedded sets}

Note that we consider invariance over large subsets of our model --- that happen to be the points of some stably embedded definable sets. This gives rise to some subtle issues and two notions of invariance. When \(D = \bigcup_i D_i\) is ind-\(\cL\)-definable, we denote by \(\eq{D}\) the ind-\(\cL\)-definable union of all \(\cL\)-interpretable sets \(X\) that admit an \(\cL\)-definable surjection \(\prod_j D_{i_j} \to X\).

\begin{definition}Let \(M\) be an \(\cL\)-structure, \(C \subseteq M\), \(D\) be a (ind-)\(\cL\)-definable set and \(p\) be a partial \(\cL(M)\)-type. 
We say that \(p\):
\begin{itemize}
\item is \emph{\(\aut(M/C)\)-invariant} if for every \(\sigma\in\aut(M/C)\), \(p\) and \(\sigma(p)\) are equivalent.
\item has \emph{\(\aut(M/C)\)-invariant \(D\)-germs} if it is \(\aut(M/C)\)-invariant and so is the \(p\)-germ of every \(\cL(M)\)-definable map \(f : p \to \eq{D}\);
\item is \emph{\(\aut(M/D)\)-invariant} if it has \(\aut(M/D(M))\)-invariant \(D\)-germs.
\end{itemize}
\end{definition}

We will only apply these notions for \(M\) saturated, \(p\) a complete $\Delta$-type for some set of \(\cL\)-formulas $\Delta$ and \(C\) equal to the $M$-points of a stably embedded (ind-)$\cL$-definable set, \(D\) stably embedded --- an ind-$\cL$-definable set \(D = \bigcup_i D_i\) is stably embedded if any definable \(X \subseteq \prod_j D_{i_j}\) is definable with parameters from \(D\).

\begin{remark}
\begin{enumerate}
\item A type might be \(\aut(M/D(M))\)-invariant but not \(\aut(M/D)\)-invariant. For example, let \(M\models\ACVF\) and \(b\) be a closed ball of \(M\) without any \(\acl(\code{b}\RV_1(M))\)-definable subballs. Then any two \(a_1,a_2\in b(M)\) have the same type over \(\acl(\code{b}\RV_1(M))\). However, for every \(x\in b\), \(\rv_1(x - a_1) = \rv_1(x-a_2)\) implies that the \(a_i\) are in the same maximal open subball of \(b\). It follows that the generic of \(b\) over \(M\) is \(\aut(M/\code{b})\)-invariant but not \(\aut(M/\code{b}\RV_1)\)-invariant.
\item A type \(p\in\TP(M)\) is \(\aut(M/C)\)-invariant if and only if for every realization \(a\models p\) in a sufficiently homogeneous \(N\supsel M\), any \(\sigma\in\aut(M/C)\) extends to an element of \(\aut(N/Ca)\).
\item On the other hand, a type \(p\in\TP(M)\) has \(\aut(M/C)\) invariant \(D\)-germs, where \(D\) is stably embedded, if and only if for every \(a\models p\) in a sufficiently saturated \(N\supsel M\), any \(\sigma\in\aut(M/C)\) extends to an element of \(\aut(N/CD(N)a)\) --- \emph{cf.} the proof of \cref{inv trans}.
\item The space of types with \(\aut(M/C)\)-invariant \(D\)-germs is closed: for any \(\sigma\in\aut(M/C)\) and \(\cL(M)\)-definable map \(f : p \to \eq{D}\), no types in the open set "\(\germ{p}{f} \neq \germ{p}{f^\sigma}\)" has \(\aut(M/C)\)-invariant \(D\)-germs.
\end{enumerate}
\end{remark}

Let us now recall the following folklore result on stable embeddedness which 
states that we can recover the usual characterisation of types and hence of definable closure (equivalently internality) from invariance over a stably embedded definable set:



\begin{lemma}\label{ste aut}
Let \(M\) be saturated sufficiently large, \(D\) be (ind-)\(\cL\)-definable stably embedded and \(e\in M\). If \(e\) is fixed by every \(\sigma\in\aut(M/D(M))\), then \(e\in\dcl(D(M))\).
\end{lemma}

\begin{proof}
Let $e'\in M$ be such that \(e\equiv_{D(M)} e'\). By \cite[Lemma~10.1.5]{TenZie} --- more precisely its extension \emph{mutatis mutandis} to stably embedded ind-definable sets --- we can find \(\sigma\in\aut(M/D(M))\) such that \(e' = \sigma(e) = e\). Since \(D\) is stably embedded, there exists a small \(A\subseteq D(M)\) such that, \(\tp(e/A)\vdash \tp(e/D(M))\). So both types have a single realisation in \(M\), \emph{i.e.}, $e\in\dcl(A) \subseteq \dcl(D(M))$.
\end{proof}

One advantage of the stronger notion of invariance is transitivity:

\begin{lemma}\label{inv trans}
Let \(M\subsel N\) be \(\cL\)-structures with \(N\) saturated and sufficiently large, \(C \subseteq M\) (potentially large), \(D\) be an (ind-)\(\cL\)-definable stably embedded set, \(p\in\TP(M)\) have \(\aut(M/C)\)-invariant \(D\)-germs, \(a\models p\) in \(N\) and \(q\in\TP(N)\) be \(\aut(N/CD(N)a)\)-invariant. Then \(\restr{q}{M}\) is \(\aut(M/C)\)-invariant.

If, moreover, \(q\) has \(\aut(N/CD(N)a)\)-invariant \(E\)-germs, for some (ind-)\(\cL\)-definable set \(E\), then \(\restr{q}{M}\) has \(\aut(M/C)\)-invariant \(E\)-germs.
\end{lemma}

\begin{proof}
Fix \(\sigma \in \aut(M/C)\). Since \(p\) is \(\aut(M/C)\)-invariant, the
automorphism \(\sigma\) extends to a partial \(\cL\)-elementary isomorphism
\(\tau : M(a) \to M(a)\) fixing \(a\). Since \(\sigma\) fixes the germs of every
definable map from \(p\) to \(\eq{D}\), \(\tau\) induces the identity on
\(\eq{D}(\dcl(M(a)))\). It follows, since \(D\) is stably embedded that \(Ca\)
and \(\tau(Ca)\) have the same type over \(D(N)\) and hence that \(\tau\)
extends to an element of \(\aut(N/CD(N)a)\) --- \emph{cf.}
\cite[Lemma~10.1.5]{TenZie}. This automorphism \(\tau\) fixes \(q\). It follows
that \(\restr{q}{M}\) is fixed by \(\restr{\tau}{M} = \sigma\).

If, moreover, \(q\) has \(\aut(N/CD(N)a)\)-invariant \(E\)-germs, then \(\sigma = \restr{\tau}{M}\) fixes the \(\restr{q}{M}\)-germ of any \(\cL(M)\)-definable function into \(\eq{E}\).
\end{proof}

The core of our proof of \cref{inv} is the following variation on transitivity:

\begin{lemma}\label{inv crit}
Let \(M\subsel N \models T\), \(C\subseteq M\) potentially large, \(a\in\K^x(N)\) a (potentially infinite) tuple and \(\rho:\K^x\to \RV\) be pro-\(\cL_0(M)\)-definable. Assume that \(\rv_\infty(M(a)) \subseteq \dcl_0(C\rho(a))\) and that \(p := \qftp(a/M)\) and \(\germ{p}{\rho}\) are \(\aut(M/C)\)-invariant. Then \(\tp(a/M)\) has \(\aut(M/C)\)-invariant \(\RV\)-germs.
\end{lemma}

\begin{proof}
Pick \(\sigma\in\aut(M/C)\). Let \(N_0 \models\ACVF\) containing $N$ be
saturated and sufficiently large. By invariance of \(p\), the automorphism
\(\sigma\) extends to a partial \(\cL_0\)-isomorphism \(\tau : M(a) \to M(a)\)
fixing \(a\). Note that \(\tau(\rho(a)) = \rho^\sigma(a) = \rho(a)\) and hence
\(\restr{\tau}{(\rv_\infty(M(a)))}\) is the identity. By quantifier elimination
in \(\ACVF\) (in \(\LRV\), which implies a strong form of stable embeddedness
for \(\RV\)), $\tau$ extends to a partial elementary map which is the identity on
$\RV(N_0)$, which further extends to some element of
\(\aut(N_0/C\RV(N_0)a)\), also denoted \(\tau\) --- \emph{cf.}
\cite[Lemma~10.1.5]{TenZie}. By \cref{EQ rv}, \(\tp(Ma) = \tp(\sigma(M)a)\),
\emph{i.e.}, \(\sigma(p) = p\). Moreover, any \(\cL(Ma)\)-definable \(X\subseteq
\RV^n\) is \(\cL(\rv_\infty(M(a)))\)-definable and hence \(X(N) = \tau(X(N)) =
X^\tau(N)\), equivalently, \(\sigma\) fixes the \(p\)-germ of any
\(\cL(M)\)-definable function into \(\eq{\RV}\).
\end{proof}

\subsection{Computing leading terms}\label{comput rv}

In view of \cref{inv crit}, given any $A$-invariant type $\tp_0(a/M)$, we want to find a pro-$\cL_0(M)$-definable map $\rho$ such that $\rho(a)$ \(\dcl_0\)-generates $\rv_\infty(M(a))$ and \(\germ{p}{\rho}\) is $\aut(M/A)$-invariants. When $A \subsel M$ and \(A\) is sufficiently large, this is done in \cref{high dim}. As previously stated, for general small \(A\), dealing with closed balls forces us to also consider maps into certain \(A\)-definable \(\res\)-vector spaces. The goal then becomes to build a "nice" model of $T$ containing $A$ and proceed by transitivity.

The technical core of the proof consists in a generalisation to relative arity one of the classical description of \(1\)-types in henselian fields, cf. \cref{open,closed,finite}.\medskip

Let us start with three leading term computations that we will need later. 
\begin{lemma}\label{calcul}
Let \(M\models\ACVF\). Let \(L\leq K = \K(M)\) be a subfield and \(R\subseteq \rv_1(K)\) contain \(\rv_1(L)\). Let also \(b\in\balls(K)\), \(g\in\K(\dcl_0(L R))\), \(c\in K\) and \(P := \prod_{i< d}(x-e_i) \in\K(\dcl_0(LR))[x]\).
\begin{enumerate}
\item\label{out} If \(c,g\in b\) and \(e_i\nin b\), for all \(i\), then \(\rv_1(P(c)) = \rv_1(P(g))\in 
\dcl_0(R)\).
\item\label{in} If \(c\nin b\) and \(g, e_i\in b\), for all \(i\). then \(\rv_1(P(c)) = \rv_1(c-g)^d\).
\item\label{sep} Assume that \(b\) is closed, that \(c,g,e_i\in b\), for all \(i\), and that the maximal open 
subball of \(b\) around \(c\) does not contain \(g\) nor any \(e_i\). Then \[\rv_1(P(c)) = \bigoplus_{i\leq d}\rv_1(P_i(g))\rv_1(c-g)^{i}\in\dcl_0(R\rv_1(c-g)),\]
where \(P(y+x) = \sum_i P_i(y) x^i\). In particular, the sum is well-defined.
\end{enumerate}
\end{lemma}

In fact, computation (2) is a particular case of computation (3) --- consider
the smallest closed ball containing \(c\), \(g\) and the \(e_i\). Recall that if
\(a_1\) and \(a_2\) are in some ball \(b\) that does not contain \(c\),
\(\rv_1(c - a_1) = \rv_1(c-a_2)\) --- \emph{cf}.~\cref{rv ball}.

\begin{proof}
Let us first prove (\ref{out}). We have \(\rv_1(P(c)) = \prod_i\rv_1(c-e_i) = \prod_i\rv_1(g-e_i) = \rv_1(P(g))\in \RV_1(\dcl_0(LR)) \subseteq \dcl_0(\rv_1(L) R) = \dcl_0(R)\), where the inclusion follows from quantifier elimination for \(\ACVF\) in \(\LRV\). As for (\ref{in}), we have \(\rv_1(P(c)) = \prod_i\rv_1(c-e_i) = \rv_1(c-g)^d\). 
Finally, in the case of (\ref{sep}), let \(Q(x) = P((c-g)x + g)/(c-g)^d\) and \(\sum_i Q_i(y)x^i = Q(x+y)\). The roots \((e_i-g)/(c-g)\) of \(Q\) are in \(\Val\). Thus \(Q\in\Val[x]\), \(Q_i(0) = P_i(g)(c-g)^{i-d}\in\Val\) and \(\val(Q(1)) = 0\).  We have:
\begin{align*}
\rv_1(P(c))
&= \rv_1(c-g)^d\resf(Q(1))\\
&= \rv_1(c-g)^d(\sum_i \resf(Q_i(0)))\\
&= \rv_1(c-g)^d(\bigoplus_i \rv_1(Q_i(0)))\\
&= \bigoplus_{i\leq d}\rv_1(P^{(i)}(g))\rv_1(c-g)^{i}\\
&\in \RV_1(\dcl_0(LR\rv_1(c-g)))\\
&\subseteq \dcl_0(R\rv_1(c-g)),
\end{align*}
where the third equality follows from the fact that \(\val(\sum_i Q_i(0)) = \val(Q(1)) = 0 \leq \min_i \{\val(Q_i(0))\}\leq \val(\sum_i Q_i(0)) \).
\end{proof}

\begin{remark}\label{infinity use}
In mixed characteristic, we will be applying this result to the least equicharacteristic zero coarsening, yielding a computation for \(\rv_\infty\) and not just \(\rv_1\).
\end{remark}

Essentially every computation of leading terms reduces to the above cases by the following lemma.

\begin{lemma}%
\label{red irr}
Let \(M\models\ACVF\), \(L\leq K = \K(M)\), \(c\in K\) and \(\rv_\infty(L)\leq R
= \RV(\dcl_0(R)) \leq\rv_\infty(K)\). The following are equivalent:
\begin{enumerate}
\item \(\rv_\infty(L(c))\subseteq R\);
\item For every \(P\in\K(\dcl_0(LR))[x]\) which is monic and irreducible over
$\K(\dcl_0(LR))$, we have \(\rv_\infty(P(c))\in R\).
\end{enumerate}
Moreover if \(P\in\K(\dcl_0(LR))[x]\) is irreducible, then its roots are either
all inside or outside any \(B\in\fballs{<\infty}(\dcl_0(LR))\).
\end{lemma}

\begin{proof}
By (1), \(\rv_\infty(P(c))\in\RV(\dcl_0(LRc)) \subseteq\RV(\dcl_0(\rv_\infty(L(c)) R)) = R\). The converse is a consequence of the fact that \(\rv_\infty\) is a multiplicative morphism and any polynomial over \(L\) is a product of (an element of \(L\) and) monic irreducible polynomials over \(\K(\dcl_0(LR))\).

As for the moreover statement, fix some \(B \in\fballs{<\infty}(\dcl_0(LR))\). Let \(E\) be the set of roots of \(P\) that belong to \(B\) and \(Q = \prod_{e\in E} (x-e) \in\K(\dcl_0(LR))[x]\). If \(P\) is irreducible, then \(Q = P\) or \(Q=1\).
\end{proof}

One last important ingredient --- also ubiquitous in the development of motivic
integration (\emph{e.g.}, in \cite{HruKaz}) --- is the fact that, in
characteristic zero, finite sets of points (and of balls in equicharacteristic
zero) can be canonically parametrised by \(\RV\). Recall the definition of
\(\fballs{r}\) from \cref{fballs} and that of $b[n]$ and $B[n]$ from
\cref{preparation}.

\begin{lemma}\label{param finite}
For every \(r\in\Zz_{>0}\), there exists \(m\in\Zz_{>0}\) such that for every characteristic zero valued field \(L\) and \(B\in \fballs{r}(L)\) with \(\card{\lball{B}{m}} = r\), there exists an \(\cL_0(\code{B})\)-definable injection \(\nu : B \to \RV^n\).
\end{lemma}

\begin{proof}
We proceed by induction on \(r\). If \(r=1\), take \(m=1\) and \(\nu\) to be constant equal to \(1\in\RV_1\). If \(\card{B}>1\), we may assume that \(\card{\lball{B}{m}} = r\) for all \(m\), the \namecref{param finite} will follow by compactness. Also, assuming that \(L\models\ACVF\) is sufficiently saturated and homogeneous, it suffices to find an \(\aut(L/\code{B})\)-invariant injection \(\nu : B \to (\RV_\infty)^n\). Indeed, since \(B\) is finite some projection to \(\RV^n\) is already injective and it must be definable. Finally, let \(\gamma := \max\{\val(b_1-b_2)\mid b_i\in B\) distinct\(\}\). Since \(\card{\lball{B}{m}} = r\) for all \(m\), \(\gamma<\rad(B) + \val(\Zz)\) and \(B\) can be injected in the set of open \(\cval\)-balls of radius \(\gamma/\Delta_\infty\). So we may assume that the residue characteristic of \(L\) is zero.

Let \(B'\) be the set of closed balls of radius \(\gamma\) around the balls of \(B\). By construction, we have \(\card{B'} < \card{B} = r\). For every \(b'\in B'\), let \(B_{b'} := \{b\in B \mid b\subseteq b'\}\). Note that, by hypothesis, \(\resf_{b'}(b)\in\Res_{b'} = \{\)maximal open subballs of \(b'\}\)  uniquely determines \(b\) inside $B$. Let \(c_{b'}\in\Res_{b'}\) denote the average of the \(\resf_{b'}(b)\) as \(b\) ranges over \(B_{b'}\). By induction, we find an \(\cL_0(\code{B})\)-definable injection \(\mu : \{c_{b'} \mid b'\in B'\} \to \RV^n\). For every \(b\in B\), let \(\nu(b) := (\rv_1(b - c_{b'}),\mu(c_{b'}))\) where \(b \subseteq b'\in B'\). Then \(\nu : B \to \RV^{n+1}\) is an \(\cL_0(\code{B})\)-definable injection.
\end{proof}

\begin{lemma}\label{pts balls}
For every \(r\in \Zz_{>0}\) there exists an \(m\in\Zz_{>0}\) such that for every characteristic zero valued field \(L\) and every \(B\in \fballs{r}(L)\) with \(\card{\lball{B}{m}} = r\), there exists \(g\in\fpoints{r}(L)\) with exactly one point inside each ball of \(\lball{B}{m}\).
\end{lemma}

\begin{proof}
Let us start with a weaker version of the result:

\begin{claim}
For every \(b\in\balls(\acl_0(L))\), \(b(\acl_0(L)) \neq\emptyset\).
\end{claim}

\begin{proof}
We may assume that \(L\) is algebraically closed. If \(v(L)\neq 0\), \(L\models\ACVF\) and hence, by model completeness, \(b(L) \neq\emptyset\). If \(\val(L) = 0\), \(\rad(b)\in \vg(\dcl_0(L)) = \{0\}\). If \(0\in b\), we are done. Otherwise, \(\val(b) = 0\) and \(b\subseteq\Val\). So \(b\) is open and it is (interdefinable with) a residue element. But \(\res(\acl_0(L)) = \resf(L)\) and thus \(b(L)\neq\emptyset\).
\end{proof}

\begin{claim}
For every \(B\in \fballs{r}(\dcl_0(L))\), there  exists \(m\in\Zz_{>0}\) and \(g\in\fpoints{r}(\dcl_0(L))\) such that, if \(\card{\lball{B}{m}} = r\), there is exactly one point of \(g\) inside each ball of \(\lball{B}{m}\).
\end{claim}

\begin{proof}
We may assume that \(B\) is irreducible over \(L\) --- \emph{i.e.}, for any non-empty \(\cL_0(L)\)-definable \(C\subseteq B\), \(C = B\). For every \(b\in B\), let \(d\in b(\acl_0(L))\). Let \(D\) be \(\cL_0(L)\)-definable and irreducible over \(L\) containing \(d\) and let \(g_b\in\acl_0(L)\) be the average of \(D\cap b\). Then \(g_b\in b[m]\), where \(m := \card{D\cap b}\).  Let \(g\) be finite \(\cL_0(L)\)-definable set irreducible over \(L\) containing \(g_b\). Since \(\card{\lball{B}{m}} = r = \card{B}\), we get \(g \cap b[m] = \{g_b\}\).  By irreducibility, each ball of \(\lball{B}{m}\) contains exactly one element of \(g\).
\end{proof}

The lemma follows by compactness.
\end{proof}

Let \(\pballs{r}{x}\) denote the (ind-\(\cL_0\)-definable) set of \(\cL_0\)-definable maps \(F:\K^x \to \fballs{r}\) and \(\pballs{<\infty}{x}\) denote the (ind-\(\cL_0\)-definable)  set \(\bigcup_r\pballs{r}{x}\). Similarly we denote \(\ppoints{r}{x}\) the  (ind-\(\cL_0\)-definable) set of \(\cL_0\)-definable maps \(F:\K^x \to \fpoints{r}\) and \(\ppoints{<\infty}{x}\) the  (ind-\(\cL_0\)-definable) set \(\bigcup_r\ppoints{r}{x}\).

\begin{notation}
We fix \(M\subsel N\models T\), \(A\subseteq \eq{M}\), \(a\in\K^x(N)\) a
potentially infinite tuple and \(c\in\K(N)\) a single element. Assume that
\(p(xy) = \qftp(ac/M)\) is \(\aut(M/A)\)-invariant and let \(q := \qftp(a/M)\).
For every \(F,G\in\pballs{<\infty}{x}(M)\), we write \(F\leq_q G\) if
\(q(x)\vdash\points{F}(x)\subseteq\points{G}(x)\). Finally, let \(E :=
\{F\in\pballs{<\infty}{x}(M)\mid p(x,y)\vdash y\in \points{F}(x)\}\).
\end{notation}

In the following \Cref{finite,open,closed}, we will describe how \(\RV(M(ac))\) is generated depending on the shape of \(E\).

\begin{lemma}[Finite sets]
\label{finite}
Assume that \(E\) has a least element \(f\) for \(\leq_q\) and that
\(f\in\ppoints{r}{x}\). Then, there exists a pro-\(\cL_0(M)\)-definable map
\(\rho : \K^{xy} \to \RV^n\), whose \(p\)-germ is \(\aut(M/A) \)-invariant such
that \(\rv_\infty(M(ac))\subseteq\dcl_0(\rv_\infty(M(a))\rho(ac))\).
\end{lemma}

\begin{proof}
Let \(\rho(ac) = \nu_a(c)\), where \(\nu_a : f(a) \to \RV^n\) is the
\(\cL_0(f(a))\)-definable injection of \cref{param finite}. By invariance of
\(p\), for every \(\sigma\in\aut(M/A)\), \(c\in f^\sigma(a)\cap f(a)\). Since
\(f\) is the least element of \(E\), we have \(f^\sigma(a) = f(a)\) and hence
\(\germ{q}{f}\) --- and thus \(\germ{p}{\rho}\) --- is \(\aut(M/A)\)-invariant.
Moreover, since \(c\in\dcl_0(Ma\rho(ac))\), we have \(\rv_\infty(M(ac))
\subseteq \RV(\dcl_0(Ma\rho(ac))) \subseteq \dcl_0(\rv_\infty(Ma)\rho(ac))\).
\end{proof}

We now assume that \(E\cap\ppoints{<\infty}{x} = \emptyset\).

\begin{lemma}
\label{sep finite}
There exists a pro-\(\cL_0(M)\)-definable map
\(\nu : \K^{xy} \to \RV^\xi\), with \(\xi\) potentially infinite, whose
\(p\)-germ is \(\aut(M/A)\)-invariant such that, for every \(F\in E\), for some
\(m\in\Zz_{>0}\) which only depends on \(\card{F(a)}\), the ball \(b\in
\lball{F(a)}{m}\) containing \(c\) is \(\cL_0(Ma\nu(ac))\)-definable, and
$\nu(ac)\in\acl_0(Ma)$.
\end{lemma}

\begin{proof}
For every \(r\in\Zz_{>0}\), let \(m_r\in\Zz_{>0}\) be as in \cref{param finite}. Let \(F\in E\cap \pballs{r}{x}\) be irreducible over \(q\) and such that \(\card{\lball{F}{m_r}} =r\) --- if such an \(F\) does not exist let \(\nu_r(x) = 1\). By irreducibility, for every \(G\in E\cap \pballs{r}{x}\) with \(G\leq_q F\), every ball in \(F(a)\) contains exactly one ball of \(G(a)\). In particular, neither \(\gamma= \max\{\val(b_1-b_2)\mid b_i \in F(a)\) distinct\(\}\) nor \(B_{r}(a)\), the set of open balls of radius \(\gamma + \val(m_r)\) around balls of \(F(a)\), depend on the choice of \(F\). It follows that \(\germ{q}{B_r}\) is \(\aut(M/A)\)-invariant. By construction, inclusion induces an injection \(F(a) \to B_r(a)\) and that $\card{B_r(a)}=\card{B_r(a)[m_r]}$.

Ss in \cref{param finite}, let \(\nu_r : B_r(a) \to \RV^n\) be an \(\cL_0(\code{B_r(a)})\)-definable injection. Let \(\nu_r(ac) = \nu_r(b)\) where \(c\in b\in B_r(a)\). Note that $\nu_r(ac)\in\acl_0(Ma)$. The element of \(F(a)\) containing \(c\) is uniquely determined by \(b\), and hence by \(\nu_r(ac)\); and \(\germ{p}{\nu_r}\) is \(\aut(M/A)\)-invariant by construction.

Let us now fix any \(F\in E\) that we can assume irreducible. Let \(M = \max\{m_s \mid s\leq\card{F(a)}\}\). The sequence \(\card{\lball{F(a)}{M^k}}\geq 1\) is decreasing, bounded by \(\card{F(a)}\) and hence, there exists \(k\leq \card{F(a)}\) such that \(\card{\lball{\lball{F(a)}{M^{k}}}{M}} = \card{\lball{F(a)}{M^{k+1}}} = \card{\lball{F(a)}{M^{k}}}\). Let \(r := \card{\lball{F}{M^k}}\). By the previous paragraphs and the choice of $M$, the ball \(b\in \lball{F(a)}{M^k}\) containing \(c\) is \(\cL_0(Ma\nu_r(ac))\)-definable. It follows that \(\nu = (\nu_r)_{r\geq1}\) has the required properties.
\end{proof}

We now wish to consider the case where, either \(E\) induces a strict intersection in the least equicharacteristic zero coarsening \(\val_\infty\) (case (1)), or \(c\) is generic over \(Ma\) in a finite set of open \(\val_\infty\)-balls (case (2)):

\begin{lemma}[Open and strict balls]\label{open}
Assume that one of the following holds:
\begin{enumerate}
\item For all \(F\in E\), there exists \(G\in E\) with \(G[m] <_q F\), for any \(m\in\Zz_{>0}\);
\item There exists an \(r\in\Zz_{>0}\) such that for every \(F\in E\) and \(m\in\Zz_{>0}\), there exists an open \(G\in E\cap\pballs{r}{x}\) with \(\lball{G}{m}\leq_q F\).
\end{enumerate}
Then, there is a pro-\(\cL_0(M)\)-definable map \(\rho : \K^{x y} \to \RV^\xi\) whose \(p\)-germ is \(\aut(M/A)\)-invariant and such that \(\rv_\infty(M(ac))\subseteq\dcl_0(\rv_\infty(M(a))\rho(ac))\).
\end{lemma}

Note that the cases (1) and (2) are not mutually exclusive.

\begin{proof}
Let \(\nu : \K^{xy}\to\RV^\xi\) be as in \cref{sep finite}.

\begin{claim}\label{sep open}
For every \(F\in E\), the ball \(b\in F(a)\) containing \(c\) is in \(\dcl_0(Ma\nu(ac))\).
\end{claim}

\begin{proof}
Assume that there exist \(G\in E\) with \(G[m] \leq_q F\), for every \(m\in\Zz_{>0}\). Then, by \cref{sep finite} applied to \(G\), the ball \(b'\in \lball{G(a)}{m}\) containing \(c\) is \(\cL_0(Ma\nu(ac))\)-definable. The claim follows since \(b'\subseteq b\).

Otherwise, by case (2), we can find a minimal \(r\) such that for every \(F\in E\) and \(m\in\Zz_{>0}\), there exists an open \(G\in E\cap \pballs{r}{x}\) with \(\lball{G}{m}\leq_q F\). Then for \(m\) sufficiently large, depending on \(r\), by \cref{sep finite}, the ball \(b'\in\lball{G}{m}\) containing \(c\) is \(\cL_0(Ma\nu(ac))\)-definable. The claim follows since \(b'\subseteq b\).
\end{proof}

If there does not exist \(g\in\ppoints{<\infty}{x}(M)\) such that \(\emptyset <_q g\leq_q E\), let \(\rho(ac) = \nu(ac)\). If such a \(g\) exists, we may assume that it is irreducible and, then, the cardinality of the \(F\in E\) irreducible over \(q\) is bounded by \(\card{g(a)}\). Let \(F\in E\) be irreducible over \(q\) of maximal cardinality \(r\) and let \(b(ac)\in F(a)\) contain \(c\). Note that the partition of \(g(a)\) induced by \(F\), and in particular \(h(ac) := g(a) \cap b(ac)\), does not depend on \(F\). Moreover, since there is some (irreducible) \(G\in E\cap\pballs{r}{x}\) with \(\lball{G}{\card{h(ac)}} \leq_q F\), the average of \(h(ac)\) is in \(b(ac)\). So, replacing \(h(ac)\) by its average, we may assume that \(h(ac)\) is a singleton. Then,  \(h(ac)\in\dcl_0(g(a)b(ac))\subseteq\dcl_0(Ma\nu(ac))\) by \cref{sep open}. Let \(\rho(ac) = (\nu(ac),\rv_\infty(c-h(ac)))\).

For every \(\sigma\in\aut(M/A)\), applying the previous argument to $g$ and $F^{\sigma^{-1}}\in E$, we have \(h(ac) \in b^{\sigma^{-1}}(ac)\),  and hence, by invariance of \(p\), \(h^\sigma(ac) \in b(ac)\). Note that the smallest (closed) ball containing \(h(ac)\) and \(h^\sigma(ac)\) is algebraic over \(Ma\) and let \(D(a)\) be its finite orbit over \(Ma\). We have \(D\leq_q E\). If \(\lball{D}{m}\in E\), for some \(m\in\Zz_{>0}\), then there is \(G\in E\) open (irreducible) such that \(\lball{G}{m}\leq_q \lball{D}{m}\) and hence \(D >_q G \in E\). But then either \(h(ac) \nin G(ac)\) or \(h^\sigma(ac)\nin G\). By invariance of \(p\), we may assume that \(h(ac) \nin G(ac)\), contradicting the fact that \(h(ac)\) does not depend on the choice of \(F\). Thus $D(a)[m]\not\in E$ and so \(c\nin \points{\lball{D(a)}{m}}\). It follows that \(\rv_\infty(c-h(ac)) = \rv_\infty(c-h^{\sigma}(ac))\). We have just proved the \(\aut(M/A)\)-invariance of \(\germ{p}{\rho}\).

Now, to prove that \(\rv_\infty(M(ac))\subseteq\dcl_0(\rv_\infty(M(a))\rho(ac))\), by \cref{red irr}, it suffices to prove that, for every irreducible \(P\in\K(\dcl_0(Ma\nu(ac)))[x]\), \(\rv_\infty(P(c)) \in \dcl_0(\rv_\infty(M(a))\rho(ac))\). Recall that, by \cref{sep finite}, $\nu(ac)\in\acl_0(Ma)$. Let \(z(a)\) be the finite set, irreducible over \(M(a)\), containing the set \(Z\) of roots of \(P\). If \(z\leq_q E\), then, as above \(c\) avoids \(\lball{d(ac)}{m}\), where \(d(ac)\) is the smallest closed ball around \(Z\cup\{h(ac)\}\). By \cref{calcul}.(\ref{in}), taking into account \cref{infinity use}, \(\rv_\infty(P(c)) = \rv_\infty(c - h(ac))^d \in \dcl_0(\rho(ac))\). Otherwise, there is some \(F\in E\) such that \(Z\cap \points{F(a)} = \emptyset\). Then, by \cref{calcul}.(\ref{out}), \(\rv_\infty(P(c)) \in\dcl_0(\rv_\infty(M(a))\nu(ac))\).
\end{proof}

\begin{remark}\label{open disj}
There are actually two distinct possible behaviours in \cref{open}:
\begin{itemize}
\item If there does does not exist \(g\in\ppoints{<\infty}{x}\) such that \(g\leq_q F\) for every \(F\in E\), then \(\rv_\infty(M(ac))\subseteq \dcl_0(\rv_\infty(M(a))\nu(ac))\subseteq\acl_0(\rv_\infty(M(a)))\);
\item If such a \(g\) exists, then \(\val(M(ac)) \not\subseteq \acl_0(\val(M(a)))\).
\end{itemize}
\end{remark}

The last remaining case to consider is when \(c\) is generic over \(Ma\) in some closed \(\val_\infty\)-ball. For every \(B\in \fballs{<\infty}\), we define \(\Res_{B, m} := \{b'\subseteq \points{B}\mid b'\) open ball of radius \(\rad(B)+\val(m)\}\) and \(\Res_{B,\infty} = \plim_m \Res_{B, m}\). For every \(x\in \points{B}\), let \(\resf_{B, m}(x)\) denote the unique element of \(\Res_{B, m}\) containing \(x\) and \(\resf_{B,\infty} : \points{B} \to \Res_{B,\infty}\) be the induced map.

\begin{lemma}[Closed balls]\label{closed}
Assume that there exists an \(F\in E\) such that for every \(g\in \ppoints{<\infty}{x}(M)\) with \(g\leq_q F\), \(c \nin \points{\resf_{F(a),\infty}(g(a))}\). Let \(b\in F(a)\) contain \(c\), \(\xi\in\RV^n\) such that \(b\in\dcl_0(Ma\xi)\) and \(G \in\resf_{b,\infty}(\dcl_0(Ma\xi))\). Then \(\rv_\infty(M(ac))\subseteq\dcl_0(\rv_\infty(M(a))\xi\rv_\infty(\resf_{F(a),\infty}(c) - G(a\xi)))\).
\end{lemma}

In later applications of this lemma, we will take \(\xi = \nu(ac)\) as given by \cref{sep finite}.

\begin{proof}
Note that for any \(m\in\Zz_{>0}\), the hypothesis on \(F\) remains true of \(F[m]\). So, replacing \(F\) by some \(F[m]\), with \(\card{F[m](a)}\) minimal, we may assume that \(\card{F[m](a)}\) is constant. By \cref{pts balls}, we can now find \(f\in\ppoints{<\infty}{x}(M)\) such that \(f(a)\) has exactly one point in every ball of \(F(a)\). By hypothesis, \(c\nin \resf_{F(a),\infty}(f(a))\). Let \(h\in\dcl_0(Ma\xi)\) denote the unique element of \(f(a)\cap b\). Since \(\rad(F(a))/\Delta_\infty = \val_\infty(c-h) = \val_\infty(c - G(a,\xi)) \leq \val_\infty(G(a,\xi) - h)\), we have \(\rv_\infty(c - h) = \rv_\infty(\resf_{F(a),\infty}(c) - G(a\xi)) \oplus \rv_\infty(G(a\xi) - h)\). So it suffices to prove that \(\rv_\infty(M(ac))\subseteq\dcl_0(\rv_\infty(M(a))\xi\rv_\infty(c - h))\).

By \cref{red irr}, it further suffices to prove that, for every irreducible \(P\in\K(\dcl_0(Ma\xi))[x]\), \(\rv_\infty(P(c)) \in \dcl_0(\rv_\infty(M(a))\xi\rv_\infty(c-h))\). If, for every \(m\in\Zz_{>0}\), no root of \(P\) is in \(\lball{b}{m}\), then, by \cref{calcul}.(\ref{out}), \(\rv_\infty(P(c)) \in \dcl_0(\rv_\infty(M(a))\xi)\). Otherwise, let \(m\in\Zz_{>0}\) be such that every root of \(P\) is in \(\lball{b}{m}\). Since \(\K(\acl_0(Ma\xi))\subseteq \acl_0(Ma)\), let \(z(a)\) be finite irreducible over \(M(a)\) containing the roots of \(P\). By hypothesis, \(c\nin \resf_{\lball{F(a)}{m}}(z(a))\). By \cref{calcul}.(\ref{sep}), \(\rv_\infty(P(c)) \in \dcl_0(\rv_\infty(M(a))\rv_\infty(c-h))\).
\end{proof}

\begin{notation}
Let \(\hA\subseteq \K(M)\) contain a realisation of every \(\cL(A)\)-type and assume that \(M\) is sufficiently saturated and homogeneous.
\end{notation}

We can now wrap up the relative arity one case:

\begin{proposition}\label{rel one}
There exists a pro-\(\cL_0(\hA)\)-definable map \(\rho: \K^{xy}\to \RV^{\xi}\) such that \(\rv_\infty(M(ac))\subseteq \dcl_0(\rv_\infty(M(a))\rho(ac))\).
\end{proposition}

\begin{proof}
Note first that any \(\aut(M/A)\)-invariant \(p\)-germ of \(\cL_0(M)\)-definable functions is represented by an \(\cL_0(\hA)\)-definable function --- it suffices to consider a realisation in \(\hA\) of the type of the parameters over \(A\).

If \(E\cap\ppoints{<\infty}{x} \neq \emptyset\), we apply \cref{finite}. So let us assume that \(E\cap\ppoints{<\infty}{x} = \emptyset\). If for every \(F\in E\), there exists \(G\in E\) with \(\lball{G}{m} <_q F\), for every \(m\in\Zz_{>0}\), we are in case (1) of \cref{open} and we can conclude. So we may assume that there exists \(F\) such that for every \(G\in E\), \(F\leq_q \lball{G}{m}\), for some \(m\in\Zz_{>0}\). If there exists \(g\in\ppoints{<\infty}{x}\) with \(g\leq_q F\) and \(c\in\resf_{F(a),\infty}(g(a))\), then, for all \(H\in E\) and \(n\in \Zz_{>0}\), \(H\leq_q \lball{F}{m}\), for some \(m\in\Zz_{>0}\). Let \(G(a) := \resf_{F(a),mn}(g(a))\), then \(c\in \lball{G}{n} \leq_q H\) and hypothesis (2) of \cref{open} holds with \(r = \card{g}\).

So we may assume that no such \(g\) exist, \emph{i.e.}, the hypotheses of \cref{closed} hold. As previously, we may assume that \(\card{\lball{F}{m}}\) is constant. Let \(\nu\) be as in \cref{sep finite}; we may assume that \(\nu\) is \(\cL_0(\hA)\)-definable. Let \(b(ac) \in F(a)\) contain \(c\).

\begin{claim}
There exists \(G(ac) \in\resf_{\lball{b(a)}{m},\infty}(\dcl_0(\hA a\nu(ac)))\), for some \(m\in\Zz_{>0}\).
\end{claim}

\begin{proof}
By construction of \(\hA\), there exists \(\sigma\in\aut(M/A)\) such that \(F^\sigma\) is \(\cL_0(\hA)\)-definable. By \(\aut(M/A)\)-invariance of \(p\), we have \(F^{\sigma} \in E\) and hence \(F^{\sigma}\leq_q \lball{F}{m}\), for some \(m\in\Zz_{>0}\). So, up to replacing $F$ by $F^\sigma$,  we may assume \(F\) is \(\cL_0(\hA)\)-definable. By \cref{pts balls}, and replacing \(F\) by some \(\lball{F}{m}\), we find \(g\in\ppoints{<\infty}{x}(\hA)\) with exactly one element in each ball of \(F\). It then suffices to consider the only element of \(\res_{F(a),\infty}(g(a))\) contained in \(b(ac)\).
\end{proof}

By \cref{closed}, \(\rv_\infty(M(ac))\subseteq\dcl_0(\rv_\infty(M(a))\nu(ac)\rv_\infty(\resf_{F(a),\infty}(c) - G(ac))) \subseteq \dcl_0(\rv_\infty(Ma)\hA ac)\).
\end{proof}

\begin{corollary}\label{high dim}
There exists a pro-\(\cL_0(\hA)\)-definable map \(\rho: \K^{x}\to \RV^{\xi}\), such that \(\rv_\infty(M(a))\subseteq \dcl_0(\rv_\infty(M)\rho(a))\).
\end{corollary}

\begin{proof}
We proceed by induction on an enumeration of \(a\). The induction step is \cref{rel one} and the limit case is trivial.
\end{proof}

\begin{corollary}\label{RV inv}
The type \(\tp(a/M)\) is \(\aut(M/\hA\RV)\)-invariant.
\end{corollary}

\begin{proof}
It follows from \cref{high dim} and \cref{inv crit}.
\end{proof}

\subsection{Invariant resolutions}

\begin{notation}
Let \(M\subsel N\models T\) both be sufficiently saturated and homogeneous and  \(A\subseteq \Geom(M)\).
\end{notation}

By transitivity, there remains to build a sufficiently saturated model containing \(A\) whose type is invariant.

\begin{lemma}\label{lift RV}
Assume that \(A\subseteq \K(M)\) and let \(R\subseteq\RV(M)\). There exists \(C\subseteq\K(N)\) and a pro-\(\cL_0(M)\)-definable map \(\rho : \K^{\card{C}} \to \RV^{\xi}\) such that \(R\subseteq\rv_\infty(A(C))\subseteq \dcl_0(AR)\), \(q := \qftp(C/M)\) and \(\germ{q}{\rho}\) are \(\aut(M/AR)\)-invariant and \(\rv_\infty(M(C))\subseteq \dcl_0(\rv_\infty(M)\rho(C))
\).
\end{lemma}

\begin{proof}
We proceed by induction on an enumeration of \(R\). Assume the property holds of \(R\) for some \(C\) and \(\rho\) and pick any \(\zeta\in\RV_\infty(M)\). If \(\zeta\in\rv_\infty(\acl_0(AC))\), let \(c\in\K(\acl_0(AC))\) be such that \(\rv_\infty(c) = \zeta\). Let \(D\) be a minimal finite \(\cL_0(AC\zeta)\)-definable set containing \(c\). Replacing \(c\) by the average of \(D\), we may assume that \(c\in\dcl_0(AC\zeta) \subseteq \dcl_0(MC) \subseteq N\). We have \(R\zeta\subseteq \rv_\infty(A(Cc))\subseteq\RV(\dcl_0(AC\zeta))\subseteq\dcl_0(\rv_\infty(A(C))\zeta)\subseteq\dcl_0(AR\zeta)\), \(\qftp(Cc/M)\) is \(\aut(M/AR\zeta)\)-invariant and, since \(c \in \dcl_0(MC) = \hens{M(C)}\), \(\rv_\infty(M(Cc))=\rv_\infty(M(C))\subseteq \dcl_0(\rv_\infty(M)\rho(C))\).

If \(\zeta\nin\rv_\infty(\acl_0(AC))\), let \(c\in N\) be generic in \(\rv_\infty^{-1}(\zeta)\) over \(M\). Then \(p := \qftp(Cc/M)\) is \(\aut(M/AR\zeta)\)-invariant. By \cref{open}, we find a pro-\(\cL_0(M)\)-definable map \(\rho' : \K^{\card{C}+1} \to \RV_\infty^\xi\) such that \(\rv_\infty(M(Cc))\subseteq\dcl_0(\rv_\infty(M(C))\rho'(Cc))\subseteq\dcl_0(\rv_\infty(M)\rho(C)\rho'(Cc))\) and whose \(p\)-germ is \(\aut(M/AR\zeta)\)-invariant. Moreover, no root of any \(P\in\K(\dcl(AC))[x]\) is in \(\rv_\infty^{-1}(\zeta)\). For any \(g\in\rv_\infty^{-1}(\zeta)\), by \cref{calcul}.(\ref{out}), \(\rv_\infty(P(c)) = \rv_\infty(P(g))\) does not depend on \(g\) and is thus in \(\RV(\dcl_0(AC\zeta))\subseteq\dcl_0(\rv_\infty(AC)\zeta)\subseteq\dcl_0(AR\zeta)\). By \cref{red irr}, \(\rv_\infty(A(Cc))\subseteq\dcl_0(AR\zeta)\).
\end{proof}

\begin{corollary}
\label{inv model}
Assume that \(A\subseteq \K(M)\). There exists \(C\subsel N\) containing \(A\)
and a pro-\(\cL_0(M)\)-definable map \(\rho : \K^{\card{C}} \to \RV^{\xi}\),
such that \(p := \qftp(C/M)\) is \(\aut(M/A\RV(M))\)-invariant,
\(\germ{p}{\rho}\) is \(\aut(M/A\RV(M))\)-invariant and
\(\rv_\infty(MC)\subseteq \dcl_0(\rv_\infty(M)\rho(C))\).
\end{corollary}

In particular, by \cref{inv crit}, \(\tp(C/M)\) is \(\aut(M/A\RV)\)-invariant.

\begin{proof}
Let \(A\subseteq M_1\subsel M\), with $M_1$ small. Applying \cref{lift RV} to
\(\rv_\infty(M_1)\), we find \(C\subseteq N\) and \(\rho\) such that
\(\rv_\infty(M_1)\subseteq\rv_\infty(A(C))\subseteq
\dcl_0(A\rv_\infty(M_1))\cap\rv_\infty(N)\subseteq \rv_\infty(M_1)\), \(p :=
\qftp(C/M)\) and \(\germ{p}{\rho}\) are \(\aut(M/A\rv_\infty(M_1))\)-invariant
and \(\rv_\infty(M(C))\subseteq \dcl_0(\rv_\infty(M)\rho(C))\). By replacing
\(C\) with \(\K(\dcl_0(AC))\), we may assume \(A\subseteq C = \hens{C}\). Since
\(\rv_\infty(C) = \rv_\infty(A(C)) = \rv_\infty(M_1)\subsel\rv_\infty(N)\) and
\(C\) is a characteristic zero henselian field, it follows from \cref{EQ rv}
that \(C\subsel N\).
\end{proof}

\begin{corollary}
\label{inv sat}
Assume that \(A\subseteq \K(M)\). There exists \(\widehat{A}\subsel N\)
containing a realisation of every \(\cL(A)\)-type such that
\(\tp(\widehat{A}/M)\) is \(\aut(M/A\RV)\)-invariant.
\end{corollary}

\begin{proof}
Let \(C\) be as in \cref{inv model}. Then any \(\cL(A)\)-definable set \(X\) has
a point in \(C\) and and its type is an \(\aut(M/A\RV)\)-invariant type
concentrating on \(X\). Since the set of \(\aut(M/A\RV)\)-invariant types is
closed, it follows by compactness that any \(\cL(A)\)-type has an
\(\aut(M/A\RV)\)-invariant extension. The corollary follows by the standard
construction relying on transitivity, \cref{inv trans} --- for the limit steps,
note that \(\aut(M/A\RV)\)-invariance is finitary: \(\tp(c/M)\) is
\(\aut(M/A\RV)\)-invariant if and only if for every finite \(c_0\subseteq c\),
\(\tp(c_0/M)\) is \(\aut(M/A\RV)\)-invariant.
\end{proof}

Recall that, by \cref{A-points}, \(\Lin_A(M)\) denotes the set of cosets \(c + \ell\Mid s\) where \(s\in\Lat(\dcl_0(A))\) has a basis in \(M\) and \(c\in s(M)\). 
\begin{lemma}\label{lift stdom} Assuming:
\begin{itemize}[leftmargin=50pt]
\item[\Infres] The residue field \(\res\) is infinite in models of \(T\),
\end{itemize}
there exists \(C\subseteq\K(N)\) and an \(\cL_0(A)\)-definable map \(\rho :  \K^{\card{C}} \to \Lin_A^{\xi}\) such that, for all \(n\in\Zz_{>0}\), \(\Lat_n(A) \subseteq s_n(C)\), \(\qftp(C/M)\) is \(\cL_0(A)\)-definable and \(\rv_\infty(M(C))\subseteq \dcl_0(\rv_\infty(M)\Lin_A(M)\rho(C))\).

In mixed characteristic, we may further assume that \(\bigcup_{n>0}\Tor_n(A) \subseteq \dcl_0(C)\).
\end{lemma}

\begin{proof}
Fix \(s\in \Lat_n(A)\) and let \(\beta\in\GL_n(M)\) be a basis of \(s\). Then, any \(\alpha\models \beta \cdot (\restr{\eta_{\Val}}{M})^{\tensor n^2}\) --- which is realised in \(N\) by \Infres{} --- where \(\eta_\Val\) is the generic (quantifier free) type of \(\Val\), is a basis of \(s\). Note that \(\qftp(\alpha/M) = \beta \cdot \eta_{\Val}^{\tensor n^2}\) only depends on \(s\) and is indeed \(\cL_0(A)\)-definable. Let \(\overline{\alpha}\), respectively \(\overline{\beta}\), be the basis of \(s/\cMid s\) --- seen as a pro-definable set--- induced by \(\alpha\), respectively \(\beta\). The matrix of \(\cres\)-coefficients of \(\alpha\) in the basis \(\beta\) is \(\cresf(\beta^{-1}\cdot\alpha)\) where \(\beta^{-1}\cdot\alpha \models (\restr{\eta_\Val}{M})^{\tensor n^2}\). It follows that \(\rv_\infty(M(\alpha)) = \rv_\infty(M(\beta^{-1}\cdot\alpha)) \subseteq \dcl_0(\RV(M)\cresf(\beta^{-1}\cdot\alpha)) \subseteq \dcl_0(\RV(M)\overline{\beta}\overline{\alpha})\). The first part of the statement follows by iterating the above construction independently for every \(s\in \Lat_n(A)\).

Let us now assume that we are in mixed characteristic. For every \(c\models\restr{\eta_{\Mid}}{M}\), \(\rv_\infty(M(Cc)) \subseteq\dcl_0(\rv_\infty(M(C))\cresf(c))\). This also holds for all maximal open subballs of \(\Val\). So, enlarging \(C\) further, we may also assume that \(\res(\dcl_0(AC))\cap M \subseteq \resf(C)\). Then, for every \(e \in \Tor_n(A)\), \(s = \tau_n(e)\in \Lat(A)\) has a basis in \(C\) and every coordinate of \(e\) in that basis is the residue of an element of \(C\). It then follows that \(e \in t_n(C)\).
\end{proof}

\begin{lemma}\label{lift Tn}
In equicharacteristic zero, assuming that, for all \(n\in\Zz_{>0}\), \(\tau_n(\Tor_n(A))\subseteq s_n(\K(A))\). Then there exists \(C\subseteq\K(N)\) and an \(\cL_0(M)\)-definable map \(\rho : \K^{\card{C}} \to \RV^{\xi}\) such that \(A\subseteq\dcl_0(C)\), \(q := \qftp(C/M)\) is \(\aut(M/A)\)-invariant, \(\germ{q}{\rho}\) is \(\aut(M/A)\)-invariant and \(\rv_\infty(M(C))\subseteq \dcl_0(\rv_\infty(M)\rho(C))\).
\end{lemma}

\begin{proof}
For every \(e\in\Tor_n(A)\), by hypothesis, \(s := \tau(e)\) has a basis in \(\K(A)\). It follows that \(s/\Mid s\) also has a basis of \(\dcl_0(A)\)-points and hence is \(\cL_0(A)\)-definably isomorphic to \(\res^n\). By \cref{lift RV} applied to $R:=\res(\dcl_0(A))\cap M\subseteq \dcl_0(A)$, we find \(\K(A)\subseteq C\leq \K(N)\) such that \(\res(\dcl_0(A))\cap M\subseteq\resf(C)\), \(q := \qftp(C/M)\) is \(\aut(M/A)\)-invariant, \(\germ{q}{\rho}\) is \(\aut(M/A)\)-invariant and \(\rv_\infty(M(C))\subseteq \dcl_0(\rv_\infty(M)\rho(C))\). Then, we have \(A = \K(A)\cup\bigcup_n\Lat_n(A)\cup\bigcup_n\Tor_n(A) \subseteq \dcl_0(C)\).
\end{proof}

\begin{corollary}\label{inv lift}
Assume that \Infres{} holds. Then there exists \(C\subseteq \K(N)\) such that \(\qftp(C/M)\) has \(\aut(M/A\RV(M)\Lin_A(M))\)-invariant \(\RV\)-germs and \(A\subseteq\dcl_0(C)\).
\end{corollary}

\begin{proof}
In mixed characteristic, this follows immediately from \cref{inv crit,lift stdom}. In residue characteristic zero, it follows from \cref{lift stdom,lift Tn} and transitivity, \emph{cf.} \cref{inv trans}.\end{proof}

\begin{proof}[Proof of \cref{inv}]
By \cref{inv lift}, we find \(C\subseteq\K(N)\) such that \(A\subseteq \dcl_0(C)\) and \(\tp(C/M)\) has \(\aut(M/A\RV(M)\Lin_A(M))\)-invariant \(\RV\)-germs. Let \(M\subsel M_1 \subsel N\) be sufficiently saturated and homogeneous and contain \(C\). By \cref{inv sat}, we find \(\hC \subsel N\) containing a realisation of every \(\cL(C)\)-type such that \(\tp(\hC/M_1)\) is \(\aut(M_1/C\RV)\)-invariant. Let \(M_1\subsel M_2 \subsel N\) be sufficiently saturated and homogeneous and contain \(\hC\). 

Let $p:=\tp_0(a/M)$, which is $\aut(M/A)$-invariant by assumption.

\begin{claim}
\(\tp(a/M)\cup\restr{p}{M_2}\) is consistent.
\end{claim}

\begin{proof}
Let \(\phi(x,m)\) be some \(\cL(M)\)-formula such that \(N\models\phi(a,m)\) and let \(\psi(x,d) \in \restr{p}{M_2}\). Let \(\sigma\in\aut(N/Amd)\) be such that \(\sigma(d)\in M\). Then \(N\models\psi(a,\sigma(d))\) and hence \(\sigma^{-1}(a)\models \psi(x,d)\wedge\phi(x,m)\).
\end{proof}

Let \(a' \models\tp(a/M)\cup \restr{p}{M_2}\). Then \(\qftp(a'/M_2)=\restr{p}{M_2}\) is \(\aut(M_2/C)\)-invariant and thus, by \cref{RV inv}, \(\tp(a'/M_2)\) is \(\aut(M_2/\hC\RV)\)-invariant. Since  \(\tp(\hC/M_1)\) is \(\aut(M_1/C\RV)\)-invariant, by transitivity, \emph{cf.} \cref{inv trans}, \(\tp(a'/M_1)\) is  \(\aut(M_1/C\RV)\)-invariant. By transitivity, since \(\tp(C/M)\) has \(\aut(M/A\RV(M)\Lin_A(M))\)-invariant \(\RV\)-germs,  \(\tp(a/M) = \tp(a'/M)\) is  \(\aut(M/A\RV(M)\Lin_A(M))\)-invariant.
\end{proof}

Let us conclude this section by relating \cref{inv} to imaginaries:

\begin{proposition}\label{EI to RV}
Let \(T\supseteq\Hen[0]\) be an \(\cL\)-theory, such that:
\begin{itemize}
\item[\Dens] For every strict pro-\(\cL(A)\)-definable \(X\subseteq\K^x\), with \(A = \acleq(A) \subseteq \eq{M}\models\eq{T}\), there exist an \(\aut(M/\Geom(A))\)-invariant \(p\in\TP^0_x(M)\) consistent with \(X\);
\item[\EQ] For every tuple \(a \in \K(M)\), with \(M\models T\), \(\tp_1(f(a)) \vdash \tp(a)\), where \(f : \K \to \K^x\) is pro-\(\cL\)-definable and \(\cL_0\subseteq \cL_1\subseteq \cL\) such that $\cL_1$ is an \(\RV\)-enrichment of $\cL_0$;
\item[\Infres] The residue field \(\res\) is infinite;
\item[\SE] \(\RV\) and \(\Res = \bigcup_\ell \Val/\ell\Mid\) are stably embedded. 
\end{itemize}
Let \(M\models T\), \(e\in\eq{M}\) and \(A = \acleq(e)\). Then
\[e\in\dcleq(\K(A)\cup\eq{(\RV\cup\Lin_{\Geom(A)})}(A)).\]
\end{proposition}

\begin{proof}
Let \(M\models T\) be saturated and sufficiently large, \(e\in\eq{M}\) and \(A = \acleq(e)\). Then \(e =g(a)\) for some \(\cL\)-definable map \(g\) and tuple \(a\in\K(M)\). Let \(Y = g^{-1}(e)\) and \(X = f(Y)\), which is a strict pro-\(\cL(A)\)-definable set. By \Dens{}, there exists an \(\aut(M/\Geom(A))\)-invariant \(p\in\TP^0_x(M)\) consistent with \(X\). We may assume that \(f(a)\models p\). By \cref{inv}, \(\tp_1(f(a)/M)\) is \(\aut(M/\Geom(A)\RV(M)\Lin_{\Geom(A)}(M))\)-invariant. By \EQ, \(\tp(a/M)\) --- and hence \(e \in \eq{M}\) --- is also \(\aut(M/\Geom(A)\RV(M)\Lin_{\Geom(A)}(M))\)-invariant.

Since \(\Lin_{\Geom(A)}\) is a collection of free \(\Val/\ell\Mid\)-modules, \(\Lin_{\Geom(A)}\) and, in fact \(\Lin_{\Geom(A)} \cup \RV\), is stably embedded. By \cref{ste aut}, \(e\in\dcleq(\Geom(A)\cup \RV(M)\cup\Lin_{\Geom(A)}(M))\) --- \emph{i.e.}, \(e = h(c)\) for some \(\cL(\Geom(A))\)-definable map \(h\) and tuple \(c\in\RV^m(M)\times\Lin_{\Geom(A)}^n(M)\). Let \(Z = h^{-1}(e)\). Then \(\code{Z}\in \eq{(\RV\cup\Lin_{\Geom(A)})}(A)\) and
\begin{eqnarray*}e\in\dcleq(\Geom(A)\code{Z}) &\subseteq& \dcleq(\Geom(A)\cup \eq{(\RV\cup\Lin_{\Geom(A)})}(A))\\
&\subseteq& \dcleq(\K(A)\cup\eq{(\RV\cup\Lin_{\Geom(A)})}(A)),
\end{eqnarray*}
since $\Geom(A)\setminus\K(A)\subseteq\eq{\Lin_{\Geom(A)}}(A)$.
\end{proof}


\section{Imaginaries in short exact sequences}\label{sequence}

In this section we will establish results which yield a relative understanding of imaginaries in certain pure short exact sequences of modules.

\subsection{The core case}
We start with a well known lemma. We include a proof for convenience.

\begin{lemma}\label{L:WEI-Orth}
Let $D$ and $C$ be stably embedded (ind-)definable sets in $D\cup C$, such that $D\perp C$ (i.e., $D$ and $C$ are orthogonal). 
\begin{enumerate} 
\item Assume that both $D$ and $C$ (considered with the full induced structures) weakly eliminate imaginaries. Then $D\cup C$ weakly eliminates imaginaries.
\item Assume that both $D$ and $C$ eliminate imaginaries and that in $C$ one has $\dcl=\acl$. Then  $D\cup C$ eliminates imaginaries.
\end{enumerate}
\end{lemma}

\begin{proof}
(1) Let $X\subseteq D^m\times C^n$ be a definable subset. Since $D\perp C$, the 

equivalence relation $\sim$ on $C^n$, given by $c\sim c':\Leftrightarrow X_c=X_{c'}$, has finitely many equivalence classes $Z_1=c_1/\sim, \ldots, Z_k=c_k/\sim$. As $\sim$ is $\code{X}$-definable and $C$ is stably embedded, the $Z_i$ are all $\eq{C}(\acleq(\code{X}))$- definable. 

For $i=1,\ldots,k$, set $Y_i= X_{c_i}\subseteq D^m$, which is $\eq{D}(\acleq(\code{X}))$-definable since $D$ is stably embedded. As $D$ and $C$ weakly eliminate imaginaries, there are finite tuples $d\in D(\acleq(\code{X}))$ and $c\in C(\acleq(\code{X}))	$ 
such that the $Y_i$ are all $d$-definable and the $Z_i$ are all $c$-definable. Thus $X=\bigcup_{i=1}^k(Y_i\times Z_i)$ is $dc$-definable, so in particular $D(\acleq(\code{X}))C(\acleq(\code{X})$-definable. 

\smallskip
(2) The assumptions on $C$ yield $\eq{C}(\acleq(\code{X})) \subseteq \dcleq(C(\code{X}))$, and so the sets $Z_1,\ldots Z_k$ are $c$-definable for some $c\in C(\code{X})$. In particular 
the $Y_i$ are then all $\code{X}$-definable, thus $\eq{D}(\code{X})$-definable. By elimination of imaginaries in $D$, we find $d\in D(\code{X})$ such that all 
$Y_i$ are $d$-definable. We now finish as in (1).
\end{proof}

\begin{fact}\label{F:BHNeumann}
Let $G$ be a group, and let $H_1,\ldots,H_N$ be subgroups of $G$. Then the left cosets of the $H_i$ form a pre-basis of 
closed sets for a noetherian topology on $G$. Moreover, setting $H_I:=\bigcap_{i\in I}H_i$ for $I\subseteq\{1,\ldots, N\}$, the irreducible 
closed sets for this topology are precisely the left cosets of those $H_I$ with the property that any proper subgroup of $H_I$ of the 
form $H_J$ is of infinite index in $H_I$.
\end{fact}

\begin{proof}
This is an easy consequence of Neumann's Lemma.
\end{proof}

Let $R$ be an integral domain, $\cL\supseteq \cL_{R-\mathrm{mod}}$ and $M$ an $\cL$-expansion of an infinite torsion free $R$-module. Let  $Z\subseteq M^n$ be an $\cL(M)$-definable set. We set $\dim_R(Z):=\max\{\dim_R(c/M)\mid c\in Z(N)\}$, where $N\succcurlyeq M$ is sufficiently saturated and $\dim_R(a/B)$ denotes the $Q(R)$-linear dimension of $a$ over $B$, for $Q(R)$ the field of fractions of $R$. 

\begin{lemma}\label{L:LinearLocus}
In the above situation, assume $\dim_R(Z)\leq r$. Then there are definable sets $C_1,\ldots,C_s\subseteq M^n$ with the following properties:
\begin{enumerate}
\item $\bigcup_{i=1}^s C_i$ is $\code{Z}$-definable.
\item Each $C_i$ is $\acleq(\code{Z})$-definable.
\item\label{good-coset} Each $C_i$ is of the form $\gamma_i+H_i$, where $H_i$ is
a definable $R$-submodule of $M^n$ given by a condition $L_ix'=mx''$, for some
matrix $L_i\in R^{(n-r)\times r}$ and $m\in R\setminus\{0\}$, where \(x\) is the
tuple \(x'x''\) up to permutation and $|x'|=r$. (In particular, $\dim_R(C_i)=r$
for all $i$.)
\item $Z\subseteq\bigcup_{i=1}^sC_i$
\end{enumerate}
Moreover, if $\mathcal{Z}$ is a definable family such that
$\dim_R(\mathcal{Z}_b)\leq r$ for all $b$, there are finitely many
$R$-submodules $H_1,\ldots, H_N$ as in the statement such that for any $b$ the
$C_i$ may be chosen among the cosets of the $H_k$.
\end{lemma}

\begin{proof}
Since \(\dim_R(Z)\leq r\), any \(c \in Z\) is in a set of the form $\gamma+H$ as
in (\ref{good-coset}). By compactness there are sets $C_1,\ldots, C_N$ with
$C_i=\delta_i+H_i$ as in (\ref{good-coset}) such that
$Z\subseteq\bigcup_{i=1}^NC_i$. Note that all $H_i$ are $\emptyset$-definable
subgroups.

By Fact \ref{F:BHNeumann}, the cosets of the $H_i$ form a pre-basis of closed sets for a noetherian topology on $M^n$. In particular, in this topology there exists a smallest closed subset 
$W$ of $M^n$ containing $Z$, and this $W$ is clearly $\code{Z}$-definable. The (finitely many) irreducible components of $W$ are then all $\acleq(\code{Z})$-definable. If $W_j$ is such an irreducible component, it is of the form 
$W_j=\gamma_j+\bigcap_{i\in I_j} H_{i}$ (where $I_j\neq\emptyset$ in case $r<n$, since $Z\subseteq\bigcup_{i=1}^N\delta_i+H_i$ by assumption). As the $H_i$ are $\emptyset$-definable, it is easy to see that if we replace each component $W_j=\gamma_j+\bigcap_{i\in I_j} H_{i}$ by $W_j'=\bigcup_{i\in I_j}\gamma_j+H_i$, then the union of all $W_j'$ is $\code{Z}$-definable 
and each coset $\gamma_j+H_i$ occurring in this union is $\acleq(\code{Z})$-definable.

The moreover part follows by compactness.
\end{proof}

\begin{theorem}\label{T:EI-ShortExSequ}
Let $R$ be an integral domain and \(M\) be
\[0\rightarrow \bA\rightarrow \bB\rightarrow \bC\rightarrow 0\]
a short exact sequence of $R$-modules, in an $\bA$-$\bC$-enrichment $\cL$ of the pure (in the sense of model theory) three sorted sequence of $R$-modules. Assume the following properties hold:
\begin{enumerate}
\item $\bA$ is a pure submodule of $\bB$ (in the sense of module theory), \label{As:pure}
\item $\bC$ is torsion free,
\item For any $l\in R\setminus\{0\}$, the quotient $\bC/l\bC$ is finite and the preimage in $\bB$ of any coset $c+l\bC$ contains an element which is algebraic over $\emptyset$.\label{As:represented}
\end{enumerate}
Let $e\in \eq{M}$. Then, setting $E:=\acleq(e)$ and $\Delta:=\bC(E)$, we have \[e\in\dcleq(\eq{\bC}(E)\eq{\bB_\Delta}(E)),\] where $\bB_\Delta$ denotes the union of all fibers
$\bB_\delta$ with $\delta\in\Delta$.
\end{theorem}

We will prove a a slight generalization of Theorem~\ref{T:EI-ShortExSequ}, namely the following Theorem~\ref{T:EI-ShortExSequ-Var}.

\begin{theorem}\label{T:EI-ShortExSequ-Var}
Let $\tilde{R}$ be a ring and $R=\tilde{R}/I$ an integral domain, with $I$ a finitely generated ideal.
Let \(M\) be
\[0\rightarrow \tilde{\bA}\rightarrow \tilde{\bB}\rightarrow \tilde{\bC}\rightarrow 0\]
a short exact sequence of $\tilde{R}$-modules, in an $\tilde{\bA}$-$\tilde{\bC}$-enrichment $\cL$ of the pure (in the sense of model theory) three sorted sequence of $\tilde{R}$-modules. Let $\bA=\{a\in\tilde{\bA}\mid Ia=(0)\}$, and let $\bB$ and $\bC$ be the $\tilde{R}$-submodules of $\tilde{\bB}$ and $\tilde{\bC}$ defined similarly. Consider $\bA,\bB,\bC$ as $R$-modules in the natural way.

Assume the following properties hold:
\begin{enumerate}
\item $\tilde{\bA}$ is a pure $\tilde{R}$-submodule of $\tilde{\bB}$ (in the sense of module theory), \label{As:pure-Var}
\item $\bC=\tilde{\bC}$,
\item $\bC$ is a torsion free $R$-module,
\item For any $l\in R\setminus\{0\}$, the quotient $\bC/l\bC$ is finite and the preimage in $\tilde{\bB}$ of any coset $c+l\bC$ contains an element which is algebraic over $\emptyset$.\label{As:represented-Var}
\end{enumerate}

Let $e\in \eq{M}$. Then, setting $E:=\acleq(e)$ and $\Delta:=\bC(E)$, we have \[e\in\dcleq(\eq{\bC}(E)\eq{\tilde{\bB}_\Delta}(E)),\] where $\tilde{\bB}_\Delta$ denotes the union of all 
$\tilde{\bB}_\delta$ for $\delta\in\Delta$.
\end{theorem}

\begin{proof}
Denote $\iota:\tilde{\bA}\rightarrow \tilde{\bB}$ and $v:\tilde{\bB}\rightarrow
\tilde{\bC}$ the structural  maps.

Note that in particular $\tilde{\bA}=\tilde{\bB}_0\subseteq \tilde{\bB}_\Delta$.
By (\ref{As:pure}),  $\tilde{\bA}$ and $\tilde{\bC}$ are (purely) stably
embedded in $T$, with $\tilde{\bA}\perp \tilde{\bC}$. Indeed, since sufficiently
saturated \(\tilde{R}\)-modules are pure-injective
(\cite[Corollary~2.9]{Pre-Modules}), replacing $M$ by a sufficiently saturated
extension, the purity assumption (\ref{As:pure}) entails that the short exact
sequence of $\tilde{R}$-modules is split. Adding a splitting to the structure
yields an expansion in which $\tilde{\bB}$ is just the product structure of
$\tilde{\bA}$ and $\tilde{\bC}$. Thus, with a splitting, $\tilde{\bA}$ and
$\tilde{\bC}$ are stably embedded and orthogonal. They remain so without the
splitting. For any subgroup \(\Delta\leq \bC\), as $\tilde{\bB}_\Delta$ is internally
$\tilde{\bA}$-internal, it follows from Lemma~\ref{L:InternallyInternal} that
$\tilde{\bB}_\Delta$ is stably embedded in $T$ (over $\Delta$), and one has
$\tilde{\bB}_\Delta\perp \tilde{\bC}$.

Since $I$ is finitely generated, $\bA,\bB$ and $\bC$ are definable and the
induced sequence of $R$-modules 
\[0\rightarrow \bA\rightarrow \bB\rightarrow \bC\rightarrow 0\] is also exact.
Indeed, this is first order expressible and holds in any sufficiently saturated
elementary extension which is split.

We now fix some \(e\in \eq{M}\) and let \(X\) be $\cL(M)$-definable with
$\code{X} = \dcleq(e)$. Lifting it to \(\tilde{\bB}\), we may assume $X\subseteq
\tilde{\bB}^n$. Assume $\dim_R(v(X))=r$. Up to passing to some subset of $X$
which is definable over $\acleq(\code{X})=E$, we may assume, using
Lemma~\ref{L:LinearLocus}, that there are $L\in R^{(n-r)\times r}$, $m\in
R\setminus\{0\}$ and $\delta''\in \bC^{n-r}$ such that for $x=x'x''$ with
$|x'|=r$, we have $mv(x'')=Lv(x')+\delta''$ for any $x=x'x''\in X$. We now
consider the (fiberwise) action of $(\bA^r,+)$ on $\tilde{\bB}^n$ given by 
\[a\cdot(x',x''):=(x'+ma,x''+La).\] For every \(c\in\bC^n\), let \(X_c\) denote
the fiber above \(c\), \emph{i.e.} \(X_c = X\cap v^{-1}(c)\).

\begin{claim}\label{C:rotation-Var}
There is $l\in R\setminus\{0\}$ such that
\[\dim_R(\{z\in v(X)\mid (l\bA)^r\cdot X_z\neq X_z\})<r.\]
\end{claim}

\begin{proof}
Let $\pi(z)$ be the partial type expressing that $z\in v(X)$ and that
$\dim_R(z/M)=r$, and let $\tau(y)$ be the partial type expressing that $y\in
(l\bA)^r$ for any $l\in R\setminus\{0\}$. Fix a realization
$(c,a)\models\pi(z)\cup\tau(y)$ in $N\supsel M$, with $c=c'c''$. Then there is
an $R$-linear map $\theta:\bC(N)\rightarrow \bA(N)$ which is trivial on $\bC(M)$
and such that $\theta(c')=ma$. Indeed, such a map $\theta$ may be found as the
restriction of an $R$-linear map from $\bC(N)\otimes_RQ(R)$ to $\bA(N)$ which is
the identity on a supplement of \(Q(R)c'\) and such that, for every non zero
\(l\in R\), \(\theta(l^{-1} c') = a_l\) where the sequence \((a_l)_{l\in
R\sminus\{0\}}\) is a coherent sequence of roots: we have \(a_1 = m
a\) and every non zero \(l,s \in R\), \(sa_{sl} = a_l\).

For $b\in \tilde{\bB}(N)$, let $\rho(b):=b+\theta(v(b))$. Then $\rho$ is an
automorphism the $\tilde{R}$-module $\tilde{\bB}(N)$ whose inverse is \(b
\mapsto b - \theta(v(b))\), since \(v(\rho(b)) = \rho(b)\). It is the identity
on \(\tilde{\bA}(N)\) and on \(\tilde{\bB}(M)\), since \(\theta\) is trivial on
\(v(\bC(M))\). Moreover, it induces the identity on \(\bC = \tilde{\bB} /
\tilde{\bA}\). Since \(M\) is an $\tilde{\bA}$-$\tilde{\bC}$-enrichment of the
pure short exact sequence, $\rho$ preserves all the structure,
\emph{i.e.}~$\rho\in\aut[\cL](N/M)$.

In particular, for any $b\in X_c$, we have $\tp(b/M)=\tp(\rho(b)/M)$ and so
$\rho(b)\in X_c$. On the other hand, as $c''=(Lc'+\delta'')/m=Lc'/m+\delta''/m$
and $v(b)=(c',c'')$, using $L\circ\theta=\theta\circ L$ we compute 
\[\theta(v(b))=(\theta(c'),\theta(c''))=(\theta(c'),\theta(L c'/m)+\theta(\delta''/m))=(ma,La),\]
from which it follows that \[\rho(b)=b+\theta(v(b))=(b'+\theta(c'),b''+\theta(c''))=(b'+ma, b''+La)=a\cdot(b',b'').\]

Thus $X_c$ is stabilized by $\bigcap_{l\in R\setminus\{0\}}(l\bA)^r$.
\cref{C:rotation-Var} now follows by compactness.
\end{proof}

Fix $l\in R\setminus \{0\}$ as in the claim. Let \(X_0 = \{x\in X\mid (l\bA)^r
\cdot x \subseteq X\}\). Then \(\dim_R(v(X\sminus X_0)) < r\) and \(X\sminus
X_0\) is coded by induction. So we can assume that \(X = X_0\) is globally
stabilized by \((l\bA)^r\). In addition, using
assumption~(\ref{As:represented-Var}) and cutting \(X\) in finitely many pieces,
we may suppose that $v(X)\subseteq ml\bC^n$. Indeed, there are only finitely
many cosets of $ml\bC^n$, all $\acleq(\emptyset)$-definable, so we may assume
$X\subseteq v^{-1}(W)$ for a coset $W$ of $ml\bC^n$. Replacing $X$ by $X-h$ for
some $h\in W\cap\acl(\emptyset)$ if necessary, we may assume $W=ml\bC^n$. Let
$a\in \tilde{\bA}^r$ and $c\in \bC^n$. If there exist $b_1=b_1'b_1''\in X_c$ and
$b_0'\in \bB^r$ such that $b_1'=a+mlb_0'$, we set

\[Y_{a,c}:=\{b''-lLb_0'\mid(b_1',b'')\in X\}=X_{(b_1')}-lLb_0',\]

where $X_{(b_1')}$ denotes the fiber $\{b''\in
\tilde{\bB}^{n-r}\mid(b_1',b'')\in X\}$. Else we set  $Y_{a,c}:=\emptyset$. Let
us first show that in the first case, $Y_{a,c}$ does not depend on the choice of
$b_1$ and $b_0'$. Indeed, if $d_1=d_1'd_1''\in X_c$ and $d_0'\in \bB^r$ are such
that $d_1'=a+mld_0'$, then $b_1'-d_1'=ml(b_0'-d_0')$, so $mlv(b_0'-d_0')=0$,
thus $v(b_0'-d_0')=0$, \emph{i.e.}, $b_0'-d_0'\in \bA^r$. Set
$a_0':=l(b_0'-d_0')$. For $d''\in X_{(d_1')}$, since \((l\bA)^r\cdot X_c =
X_c\), we have
\[a_0'\cdot(d_1',d'')=(d_1'+ma_0',d''+La_0')=(b_1',d''+La_0')\in X_c,\]
so $d''+lL(b_0'-d_0')=d''+La_0'\in X_{(b_1')}$, showing that $X_{(d_1')}-lLd_0'\subseteq X_{(b_1')}-lLb_0'$. By symmetry, we get the other inclusion $X_{(b_1')}-lLb_0'\subseteq X_{(d_1')}-lLd_0'$, thus $X_{(b_1')}-lLb_0'=X_{(d_1')}-lLd_0'$.

Let $\delta''=(\delta_i'')_{1\leq i\leq n-r}$ and let $b'=a+mlb_0'$ be as in the definition of $Y_{a,c}$. Then for $y=b''-lLb_0'\in Y_{a,c}$, we compute 
\[mv(y)=mv(b'')-mlLv(b_0')=(Lv(b')+\delta'')-Lv(b')=\delta'',\]

yielding $Y_{a,c}\subseteq \tilde{\bB}_{\delta''/m}=\prod_{i=1}^{n-r}\tilde{\bB}_{\delta_i''/m}$. It follows that 
\[Y\subseteq(\tilde{\bB}_{\delta''/m}\times \tilde{\bA}^r)\times \bC^n.\] As
$\delta''/m\in \bC(E)=\Delta$ and $\tilde{\bB}_\Delta\perp \bC$, by
Lemma~\ref{L:WEI-Orth}(1), \(Y\) is coded in \(\eq{\tilde{\bB}_\Delta}\cup
\eq{\bC}\). So \(X\) is coded in the same sorts once the following claim is
established.

\begin{claim}
$\code{X}$ and $\code{Y}$ are interdefinable.
\end{claim}

It is clear by construction that $Y$ is $\code{X}$-definable. For the converse, we will use that we have reduced to the case where $v(X)\subseteq ml\bC^n$. We may thus reconstruct $X$ from $Y$ as follows:
\[d=d'd''\in X\Leftrightarrow \exists a\in \tilde{E}^r\,\exists d_0'\in \bB^r:\,d'=a+mld_0'
\text{ and }
d''\in Y_{a,v(d)}+lLd_0'\]
This yields the claim.
\end{proof}

\subsection{Some variants}
We will now state two variants of Theorem~\ref{T:EI-ShortExSequ-Var},  tailor made for our applications to (enriched) henselian valued fields.

\begin{variant}\label{V:EI-ShortExSequ-Sorts}
Let $\cL$ be a multisorted language, $\mathcal{A}\sqcup\{\tilde{\bB}\}\sqcup\mathcal{C}$ a partition of the sorts of $\cL$ and let 
$\tilde{\bA}\in \mathcal{A}$ and $\tilde{\bC}\in \mathcal{C}$. Let $\tilde{R}$ be a ring and $R=\tilde{R}/I$ an integral domain, with $I$ a finitely generated ideal.
Let 
\begin{equation}
0\rightarrow \tilde{\bA}\rightarrow \tilde{\bB}\rightarrow \tilde{\bC}\rightarrow 0\label{Eq:SES}
\end{equation}
be a short exact sequence of $\tilde{R}$-modules. Let $M$ be an $\cL$-structure which is an $\mathcal{A}$-$\mathcal{C}$-enrichment of the pure (in the sense of model theory) sequence of $\tilde{R}$-modules \eqref{Eq:SES}. 
Assume that the properties (\ref{As:pure-Var})-(\ref{As:represented-Var}) from the statement of Theorem~\ref{T:EI-ShortExSequ-Var} hold.

Let $e\in \eq{M}$. Then, setting $E:=\acleq(e)$ and $\Delta:=\bC(E)$, we have \[e\in\dcleq(\eq{\mathcal{\bC}}(E)\cup\eq{(\mathcal{A}\cup\tilde{\bB}_\Delta)}(E)),\] where $\tilde{\bB}_\Delta$ denotes the union of all 
$\tilde{\bB}_\delta$ for $\delta\in\Delta$.
\end{variant}

\begin{proof}
The proof is a slight variation of the proof of Theorem~\ref{T:EI-ShortExSequ-Var}. Let us indicate the necessary adaptations. 

Firstly, it follows from the assumptions that $\mathcal{A}$ and $\mathcal{C}$ are (purely) stably embedded with $\mathcal{A}\perp\mathcal{C}$. Thus, by Lemma~\ref{L:InternallyInternal}, $(\mathcal{A}\cup\tilde{\bB}_\Delta)\perp\mathcal{C}$ and $\mathcal{A}\cup\tilde{\bB}_\Delta$ is stably embedded. 

Given $e\in \eq{M}$, we choose an $\cL(M)$-definable set $X\subseteq \bA'\times \bC'\times\tilde{\bB}^n$  with $e=\code{X}$, where $\bA'$ is a finite product of sorts from $\mathcal{A}$ and $\bC'$ is a finite product of sorts from $\mathcal{C}$. For $(a',c')\in \bA'\times \bC'$, let \[\leftidx{_{(a',c')}}{X}{}\subseteq \tilde{\bB}^n\] be the fiber over $(a',c')$. Performing the same reductions as in the proof of Theorem~\ref{T:EI-ShortExSequ-Var}, by compactness, we may assume that there is an $\emptyset$-definable set 
\[Y\subseteq \bA'\times \bC'\times \tilde{\bB}_{\delta}\times 
\bC^n\] for some finite tuple $\delta\in\Delta$ such that for any $(a',c')\in \bA'\times \bC'$, $\code{\leftidx{_{(a',c')}}{X}{}}$ and $\code{\leftidx{_{(a',c')}}{Y}{}}$ are interdefinable, so in particular, $\code{X}$ and $\code{Y}$ are interdefinable. The result then follows 
from Lemma~\ref{L:WEI-Orth}, 
since $(\mathcal{A}\cup\tilde{\bB}_\Delta)\perp\mathcal{C}$.
\end{proof}

The second variant is designed for applications to henselian valued fields in mixed characteristic.

\begin{variant}\label{V:EI-ShortExSequ-Complex-C}
Let $\cL$ be a multisorted language, $\mathcal{A}\sqcup\{\tilde{\bB}_n\mid n\in\Nn\}\sqcup\mathcal{C}$ a partition of the sorts of $\cL$. For any $n\in\Nn$, let 
$\tilde{\bA}_n\in \mathcal{A}$, and let $\tilde{C}\in \mathcal{C}$. Let $\tilde{R}$ be a ring and $R=\tilde{R}/I$ an integral domain, with $I$ a finitely generated ideal.
Let $\tilde{\bA}=(\tilde{\bA}_n)_{n\in\Nn}$, $\tilde{\bB}=(\tilde{\bB}_n)_{n\in\Nn}$ be projective systems of $\tilde{R}$-modules with surjective transition functions, let  $\bC =(\tilde{\bC}_n)_{n\in\Nn}$ be the projective system with 
$\tilde{\bC_n}=\tilde{\bC}$ for all $n$ and identical transition functions. Let
\begin{equation}\label{Eq:SESC-C}
0\rightarrow \tilde{\bA}\rightarrow \tilde{\bB}\rightarrow \tilde{\bC}\rightarrow 0
\end{equation}
be a short exact sequence of projective systems of $\tilde{R}$-modules.

Let $M$ be an $\cL$-structure which is an $\mathcal{A}$-$\mathcal{C}$-enrichment of the pure (in the sense of model theory) sequence of projective systems of $\tilde{R}$-modules \eqref{Eq:SESC-C}.

Assume that for very $n\in\Nn$ the exact sequence $0\rightarrow \tilde{\bA}_n\rightarrow \tilde{\bB}_n\rightarrow \tilde{\bC}\rightarrow 0$ satisfies the properties (\ref{As:pure-Var})-(\ref{As:represented-Var}) from the statement of Theorem~\ref{T:EI-ShortExSequ-Var}.

Let $e\in \eq{M}$. Then, setting $E:=\acleq(e)$ and $\Delta:=\bC(E)$, we have \[e\in\dcleq(\eq{\tilde{\mathcal{C}}}(E)\eq{(\mathcal{A}\cup\tilde{\bB}_\Delta)}(E)),\] where $\tilde{\bB}_\Delta$ denotes the union of all 
$(\tilde{\bB_n})_\delta$ for $\delta\in\Delta$ and $n\in \Nn$.
\end{variant}

\begin{proof}
Let us first show that $\mathcal{A}$ and $\mathcal{C}$ are (purely) stably embedded in $M$ such that $\mathcal{A}\perp\mathcal{C}$. For this, given $N\in\Nn$, we consider the structure $M_N$ given by restricting $M$ to the sorts $\mathcal{A}\sqcup\{\tilde{\bB}_m\mid m\leq N\}\sqcup\mathcal{C}$. For $m\leq N$ we denote by $p_{N,m}$ the structural map from $\tilde{\bB}_N$ to $\tilde{\bB}_m$ and by $q_{N,m}$ the one from $\tilde{\bA}_N$ to $\tilde{\bA}_m$. For any $m\leq N$, the sequence $\tilde{S}_m$ of $\tilde{R}$-modules
\[0\rightarrow \tilde{\bA}_m\rightarrow\tilde{\bB}_m\rightarrow \tilde{\bC}\rightarrow0\]
is interpretable in the sequence $\tilde{S}_N$ once a predicate for $\ker(q_{N,m})\leq \tilde{\bA}_N$ is added. Thus, $M_N$ may be seen as an $\mathcal{A}$-$\mathcal{C}$-enrichment of $\tilde{S}_N$

As in the previous proofs, it follows that $\mathcal{A}\cup\tilde{\bB}_\Delta$ is stably embedded in $M$, with $(\mathcal{A}\cup\tilde{\bB}_\Delta)\perp \tilde{\bC}$.
Given $e\in \eq{M}$, we choose $X\subseteq \bA'\times \bC'\times\tilde{\bB}_k^n$ $\cL(M)$-definable with $e=\code{X}$, where $\bA'$ is a finite product of sorts from $\mathcal{A}$, $\bC'$ is a finite product of sorts from $\mathcal{C}$ and $k\in\Nn$. Let $N\geq k$ be such that 
$X$ may be defined using formulas involving only variables from sorts in $\mathcal{A}\cup\mathcal{C}\cup\{\tilde{\bB}_i\mid i\leq N\}$. Since, for $N\geq m$, \(\tilde{S}_m\) is interpretable in an $\mathcal{A}$-$\mathcal{C}$-enrichment of \(\tilde{S}_N\), we may conclude with \cref{V:EI-ShortExSequ-Sorts}.
\end{proof}

\subsection{Imaginaries in RV} Recall that in a finitely ramified henselian valued field, the projective system of short exact sequences:
\[1\to \Res_n^\times  \to \RV_n^\times \to \vg^\times\to 0\]
is stably embedded with the induced structure a \(\vg\)-\(\Res\)-enrichment of the pure short exact sequence of abelian groups. Thus \cref{V:EI-ShortExSequ-Complex-C} applies and yields the following elimination of imaginaries:

\begin{proposition}\label{EI RV Hen0}
Let \(M\) be a \(\vg\)-\(\Res\)-enriched finitely ramified henselian field, \(A \subseteq \Geom(M)\), \(e \in\eq{(\RV\cup\Lin_{A})}(M)\) and \(E = \acleq(e)\). Assume that:
\begin{itemize}
\item For every \(n,\ell\in\Zz_{>0}\), \(\vg/\ell\vg\) is finite and the preimage in $\RV_n$ of any coset of $\ell \vg$ contains an element which is algebraic over $\emptyset$.
\end{itemize}
Then \(e\in\dcleq(\eq{\vg}(E) \cup \eq{(\Lin_{A}\cup\RV_{\vg(E)})}(E))\).
\end{proposition}

In particular, for \(A = \acleq(A) \subseteq \eq{M}\), \(\eq{(\RV\cup\Lin_{\Geom(A)})}(A) \subseteq \dcleq(\eq{\vg}(A)\cup \eq{\Lin_{\Geom(A)}}(A))\).

\begin{proof}
We apply \cref{V:EI-ShortExSequ-Complex-C} with \(R = \tilde{R} = \Zz\). Since \(\vg\) is ordered is it a torsion free \(\Zz\)-module, so (1) - (3) hold and (4) holds by hypothesis.
\end{proof}

These result also apply with an automorphism:

\begin{proposition}\label{EI RV VFA}
Let \(M\models \VFA\), \(A\subseteq\Geom(M)\), \(e \in\eq{(\RV\cup\Lin_{A})}(M)\) and \(E = \acleq(e)\). Then \(e\in\dcleq(A\vg(E) \RV_{\vg(E)}(E) \Lin_{A}(E))\).
\end{proposition}

In particular, for \(A = \acleq(A) \subseteq \eq{M}\), \(\eq{(\RV\cup\Lin_{\Geom(A)})}(A) \subseteq \dcleq(\Geom(A))\).

\begin{proof}
We apply \cref{V:EI-ShortExSequ-Sorts} with \(R = \Zz[\sigma]\) and \(I := \{P\in\Zz[\sigma]\mid P(\vg) = 0\}\), which is finitely generated since $R$ is noetherian. Hypothesis (1) holds by assumption. Hypothesis (2) and (3) hold by multiplicativity: if \(c\in\vg_{>0}\) and \(P\in\Zz[\sigma]\) are such that \(P(c) = 0\), then for all \(c\in\vg\), \(P(c) = 0\) and \(P\in I\). Finally hypothesis (4) holds by divisibility.

So \(e \in \dcleq(\eq{\vg})(E)\cup\eq{(\Lin_A\cup\RV_{\vg(E)})}(E)\). But \(\vg\) is an ordered vector field over (the field of fraction of) \(\Zz[\sigma]/I\), so it eliminates imaginaries. Also, by \cref{EI Lin}, \(\Lin_{A}\cup\RV_{\vg(E)}\) weakly eliminate imaginaries. So \(\eq{\vg}(E) \subseteq \dcleq(\vg(E))\) and \[\eq{(\Lin_{A}\cup\RV_{\vg(E)})}(E) \subseteq \dcleq(\Lin_{A}(E)\RV_{\vg(E)}(E)).\]
The result follows.
\end{proof}

\section{Imaginaries in valued fields}

\subsection{The henselian case}
Let \(\Hen[0]^\ac\) denote the \(\RV\)-enrichment of \(\Hen[0]\) with a compatible system of angular components \(\ac_n : \K\to\Res_n^\times\) --- we denote this language by \(\Lac\). We fix \(T\) a \(\vg\)-\(\res\)-enrichment of either \(\Hen[0]\) or \(\Hen[0]^\ac\).

Recall that the two sorted language \(\Lvs\) is given by a sort \(\Ann\) endowed with the language of rings, a sort \(\V\) endowed with the language of abelian groups and a function symbol \(\mu:\Ann\times\V\rightarrow\V\) for scalar multiplication. Let \(t_\ell\) denote the \(\emptyset\)-induced theory on \(\Res_\ell\). For every \(X\subseteq\V^2\) definable in the \(\Ann\)-enriched theory of free rank \(n\in\Zz_{> 0}\) modules over models of \(t_\ell\), we define the equivalence relation \(E_X\) on \(\V\) by \(v E_X w\) if \(X_{v} = X_w\) and \(\Tor_{n,\ell,X} := \bigsqcup_{s\in\Lat_n} (\V/E_X)^{(\Res_\ell,s/\ell\Mid s)}\). Let \(\lin{\Res}\) denote \(\bigsqcup_{n,\ell,X} \Tor_{n,\ell,X}\), the \(\Res\)-linear imaginaries. We can now prove our imaginary Ax-Kochen-Ershov principle:

\begin{theorem}[Theorem~A]\label{AKE EI}
Let \(T\) be a \(\vg\)-\(\res\)-enrichment of either \(\Hen[0]\) or \(\Hen[0]^\ac\), such that:
\begin{itemize}[leftmargin=50pt]
\item[\Cvg]  \(T\) has definably complete value group;
\item[\Ram] For every \(\ell\in\Zz_{>0}\), the interval \([0,\val(\ell)]\) is finite and \(\res\) is perfect;
\item[\Infres] The residue field \(\res\) is infinite;
\item[\UFres] The induced theory on \(\res\) eliminates \(\exists^\infty\).
\end{itemize}
Then \(T\) weakly eliminates imaginaries in \(\K\cup\eq{\vg}\cup\lin{\Res}\).
\end{theorem}

\begin{proof}
We will use  \cref{EI to RV}. By \cref{inv dens}, hypothesis \Dens{} holds. Hypothesis \EQ{} holds trivially for \(\cL _1= \cL\) --- and \(f = \id\). Also, \(\RV\) and \(\Res\) are stably embedded in characteristic zero henselian fields. 
Let \(M\models T\), \(e\in\eq{M}\) and \(A = \acleq(e)\). By \cref{EI to RV}, \(e\in\dcleq(\K(A)\cup\eq{(\RV\cup\Lin_{\Geom(A)})}(A))\). 

\begin{claim}
\(\eq{(\RV\cup\Lin_{\Geom(A)})}(A) \subseteq \dcleq(\eq{\vg}(A)\cup \eq{\Lin_{\Geom(A)}}(A))\)
\end{claim}

\begin{proof}
If \(T\supseteq \Hen[0]^\ac\), then \(\RV_n\) is \(\Lac\) isomorphic to \(\Res_n^\times\times\vg\) and the isomorphisms are compatible as \(n\) varies. It follows that \(\eq{(\RV\cup\Lin_{\Geom(A)})} \subseteq \eq{(\vg\cup\Lin_{\Geom(A)})}\), and since \(\vg\) and \(\Lin_{\Geom(A)}\) are orthogonal, the claim follows.

If \(T\) is a \(\vg\)-\(\res\)-enrichment of \(\Hen[0]\), the claim follows from \cref{EI RV Hen0}. Note that by \Cvg, \(\vg \equiv \Qq\) or \(\vg\equiv \Zz\) and hence \(\vg / n\vg \) is finite and every coset is represented in \(\vg(\dcleq(\emptyset))\). 
\end{proof}

\begin{claim}
\(\eq{\Lin}_{\Geom(A)}(A)\subseteq \dcleq(\lin{\Res}(A))\)
\end{claim}

\begin{proof}
Recall that \(\Lin_{\Geom(A)}\) is stably embedded. It follows that, for every \(e\in \eq{\Lin}_{\Geom(A)}(A)\), taking tensor products of lattices, we may assume that there exists  \(n,\ell\in\Zz_{>0}\) and \(s\in\Lat_n(A)\) such that  \(e\) codes some subset \(X_a\) of \(s/\ell\Mid s\) and \(a\) a single parameter in \(s/\ell\Mid s\). Since \(s/\ell\Mid s\) is definably isomorphic to \(\Res_\ell^n\) once we name a basis, it follows that \(X\) is definable with parameters in the \(\Lvs\)-structure \((\Res_\ell,s/\ell\Mid s)\) --- so \(e\in\dcleq(\Tor_{n,\ell,X}(A))\).
\end{proof}

It follows that \(e\in \dcleq(\K(A)\cup\eq{\vg}(A)\cup\lin{\Res}(A))\), which concludes the proof.
\end{proof}

Let \(\lin{\res} := \bigsqcup_{s,n,X} (\V/E_X)^{(\res,s/\Mid s)}\).

\begin{corollary}\label{absolute EI Hen}
Let \(F\) be a characteristic zero field that eliminates \(\exists^\infty\). Then any valued field elementarily equivalent to \(F((t))\) or \(F((t^\Qq))\) (with or without angular components) weakly eliminate imaginaries in \(\K\cup\lin{\res}\).\qed
\end{corollary}

In certain cases, the elimination of imaginaries in \(\Th(\res)\)-linear structures allows to further reduce this result to the geometric sorts. We can then also code finite sets, by adapting an argument of Johnson \cite[Section~5.3]{Joh-EIACVF}:

\begin{proposition}\label{fin set}
Let \(M\) be a henselian valued field such that:
\begin{itemize}[leftmargin=50pt]
\item[\Infres] The residue field \(\res\) is infinite;
\item[\Cod] For any \(A = \dcleq(A) \subseteq \eq{M}\), any \(\cL(A)\)-definable type \(p\in\TP^0_{\K^n}(M)\) finitely satisfiable in \(M\) is \(\cL_0(\Geom(A))\)-definable.
\end{itemize}
Then every finite set in \(\Geom\) is coded in \(\Geom\).
\end{proposition}

\begin{proof}
The proposition follows from:

\begin{claim}
Let \(C\subseteq \Geom\) be a finite set. There exists an \(\cL(\code{C})\)-definable type \(p_C \in\TP_{\K^n}^0(M)\) finitely satisfiable in \(M\) such that \(C\) is \(\cL(a)\)-definable for any \(a \models p_C\).
\end{claim}

Indeed taking \(p_C\) as in the claim, by \Cod, \(p_C\) is \(\Geom(\dcleq(\code{C}))\)-definable and hence, so is \(C\).

We now prove the claim. We start by considering \(C = \{s\}\subseteq \Lat_n\).
By \Infres, the generic type \(q_n\) of \(\GL_n(\Val)\) (that is, the quantifier
free type that reduces to the generic of \(\GL_n(\res)\) modulo \(\Mid\)) is
finitely satisfiable in \(M\). Note also that it is \(\cL_0\)-definable and
symmetric (\emph{cf.}, \cite[Section~3.3]{Joh-EIACVF}): for every definable quantifier free type \(r\), \(q_n\tensor r =
r\tensor q_n\). Let \(B\) be the matrix associated to some basis of \(s\) in
\(M\) and let \(p_s = B\cdot q_n\). This type does not depend on \(B\), is
\(\cL_0(s)\)-definable, finitely satisfiable in \(M\) and symmetric.

If \(C\subseteq \Lat_n\), let \(a\models\bigotimes_{s\in C} p_s\) and \(A\) be
the set of the \(a_{s}\), for \(s\in S\). Since finite subsets of \(\K\) are
coded in \(\K\), \(\code{A}\) can be identified with a tuple in \(\K\). Then
\(p_C = \tp_0(\code{A}/M)\) does not depend on a choice of enumeration of \(C\)
and thus it is  \(\cL(\code{C})\)-definable (and finitely satisfiable in \(M\)).
Note that if \(\vg(M)\) has a smallest positive element, we are done since, for
any lattice \(s\), \(\Mid s\) is (\(\cL\)-definably isomorphic to) a lattice and
hence \(\Tor_n\) \(\cL\)-definably embeds in \(\Lat_{n+1}\) by the usual
identification of translates of linear spaces with higher dimensional linear
spaces.

Let us now assume that \(\vg(M)\) does not admit a smallest positive element. We
first consider the case where \(C\subseteq \K^i\times \res^j\). Let \(E
\subseteq \res\) be the set of elements of \(\res\) appearing as coordinates of
elements in \(C\), let \(b\) be the tuple of coefficients of the polynomial \(\prod_{e\in
E} (x - e)\) and let \(p_E\in\TP^0(M)\) be the type of generic lifts of
\(b\in\res\) to \(\Val\). This type is \(\cL(\code{C})\)-definable and finitely satisfiable
in \(M\) since the value group does not admit a smallest positive element. Then
for any \(a\models p_E\), \(\resf\) induces a bijection between the \(\cL_0(a)\)-definable set of roots of the polynomial \(\sum_\ell a_\ell x^\ell\) and \(E\). It follows that \(C\) is in \(\cL_0(a)\)-definable bijection with a subset \(D\) of
\(\K^{i+j}\), which is coded in some cartesian power of \(\K\). Then \(p_C =
\tp_0(a\code{D}/M)\) has the required properties.

Let us finally consider \(e\in \Tor_n(M)\) and let \(s =\tau_n(e)\) (see page~\pageref{tau}). If \(a\models p_s\), then \(e\) is \(\cL_0(a)\)-definably isomorphic to a tuple \(b \in\res\). The type \(p_e = \tp_0(ab/M)\) is \(\cL_0(e)\)-definable, finitely satisfiable in \(M\) and symmetric --- since \(\res\) is quantifier free stable. If \(C\subseteq \Tor_n\), let \(a\models \bigotimes_{e\in C} p_e\) and \(A \subseteq \K^i\times \res^j\) be the set of the \(a_e\), for \(e\in C\). Then, applying the previous paragraph to \(A\), we find \(p_C\) as required.
\end{proof}

The authors would like to thank Ehud Hrushovski for his insights on the correct statement of the following corollary.

\begin{corollary}\label{ACF-RCF-Psf}
Let \(F\) be a characteristic zero field that is either algebraically closed, pseudofinite or real closed. Then any valued field \(M\) elementarily equivalent to \(F((t))\) or \(F((t^\Qq))\) (with or without angular components), eliminates imaginaries in  \(\Geom\) provided we add the following (imaginary) constants:
\begin{itemize}
\item if \(F\) is real closed and the value group is a \(\Zz\)-group with minimal positive element \(\gamma_0\), a constant for a half line of \(\RV_{1,\gamma_0}\);
\item if \(F\) is pseudofinite, constants for a generator of Galois of \(F\) --- see \cite[Section~5.9]{Hru-GpIm};
\item if \(F\) is pseudofinite and the value group is a \(\Zz\)-group with minimal positive element \(\gamma_0\), a constant for a  \((\res^\star)^n\)-orbit in \(\RV_{1,\gamma_0}\), for every \(n\geq 1\).
\end{itemize}
\end{corollary}

\begin{proof}
Let \(A = \dcleq(A)\subseteq \eq{M}\). By \ref{F:lin-St}, \(\Lin_A(M)\) is a
\(\Th(\res(M))\)-linear structure with flags. If \(\res(M)\) is algebraically
closed, by \cite[Lemma~5.6]{Hru-GpIm}, it eliminates imaginaries. If \(\res(M)\)
is real closed, \(\Lin_A\) also eliminates imaginaries by \cref{RCF lin
EI,oriented,def order}. If \(\res(M)\) is pseudofinite and \(\vg(M)\) is
divisible, \(\Lin_A(M)\) has roots and we conclude with
\cite[Theorem~5.10]{Hru-GpIm}. If \(\vg(M)\) is a \(\Zz\)-group, \(\Lin_A\) does
not have roots, but by \cite[Remark~5.8]{Hru-GpIm} it suffices to insure the
existence of an \(\emptyset\)-definable \((\res^\star)^n\)-orbit inside each one
dimensional vector space of \(\Lin_A\), \emph{i.e.}, each \(\RV_{1,\gamma}\)
with \(\gamma\in\vg(A)\). For every \(n\), write \(\gamma\) as \(i \gamma_0 +
n\delta\). By assumption, there exists \(\emptyset\)-definable
\((\res^\star)^n\cdot \xi \subseteq \RV_{1,\gamma_0}\). Then
\((\res^\star)^n\cdot \xi^i\zeta^n\subseteq \RV_{1,\gamma}\) is independent of
the choice of \(\zeta\in\RV_{1,\delta}\) and thus \(\emptyset\)-definable.

It now follows from \cref{absolute EI Hen} that \(M\) weakly eliminates imaginaries in \(\Geom\) and we conclude with \cref{fin set}.
\end{proof}

Not all of these results are new, although all of the statements with angular components are. The case of \(\Cc((t^\Qq))\) just amounts to Haskel-Hrushovski-Macpherson's result \cite{HasHruMac-ACVF} for \(\ACVF\). The case of \(\Rr((t^\Qq))\) is Mellor's result \cite{Mel-RCVF} for \(\RCVF\) and the case of \(F((t))\), with \(F\) pseudofinite, is Hrushovski-Martin-Rideau's result \cite{HruMarRid} for pseudolocal fields --- slightly improved since we only require algebraic constants in \(\eq{\RV_1}\) and not in \(\K\).

\begin{corollary} Let \(F\) be a positive characteristic perfect field that eliminates \(\exists^\infty\). Then \(\W(F)\) (with or without compatible angular components) weakly eliminate imaginaries in \(\K\cup\lin{\Res}\).\qed
\end{corollary}

It seems plausible that, if \(F\models\ACF\), the \(\Res\)-linear imaginaries can also be eliminated, yielding elimination of imaginaries in \(\Geom\) for \(\W(\alg{\Ff_p})\). However, this remains an open problem.

\subsection{The \texorpdfstring{\(\sigma\)}{sigma}-henselian case} Let us conclude with the description of the imaginaries in \(\VFA\):

\begin{theorem}[Theorem~B]\label{VFA-EI}
The theory \(\VFA\) (with or without equivariant angular components) eliminates imaginaries in \(\Geom\).
\end{theorem}

\begin{proof}
Any model of \(\VFA\) is elementarily equivalent to a maximally complete one (\emph{cf.} \cref{max complete}) and
hence \Cball{} holds. By \cref{EQ VFA}, the structure induced on \(\vg\) is
\(o\)-minimal. So \Cvg{} and \UFvg{} hold. Finally, \UFres{} holds since
\(\ACFA\) eliminates \(\exists^\infty\). By \cref{def dens}, \Dens{} holds.
Hypothesis \EQ{} follows from \cref{EQ sHen}, with \(\cL_1 :=
\LRV\cup\{\sigma_\RV\}\) and \(f(x) := (\sigma^n(c))_{n\in\Zz_{\geq 0}}\). So,
by \cref{EI to RV}, for every \(M\models\VFA\), \(e\in\eq{M}\) and \(A =
\acleq(e)\), we have \(e\in\dcleq(\K(A)\cup\eq{(\RV\cup\Lin_{\Geom(A)})}(A))\).
By \cref{EI RV VFA} --- or using the angular components --- we have
\(\eq{(\RV\cup\Lin_{\Geom(A)})}(A) \subseteq \dcleq(\Geom(A))\).
\end{proof}

Similar results hold in differential valued fields.

\begin{corollary}
The following two families of difference valued fields, indexed by integer primes \(p\):
\begin{enumerate}
\item \(K_p := (\alg{\Ff_p(t)},\val_t,\Frob_p)\), where \(\Frob_p\) is the Frobenius automorphism;
\item \(K_p := (\Cc_p,\val_p,\sigma_p)\), where \(\sigma_p\) is an isometric lift of the Frobenius automorphism on \(\res(\Cc_p) = \alg{\Ff_p}\).
\end{enumerate}
uniformly eliminate imaginaries in \(\Geom\) for large \(p\): for every \(\LRVsig\)-definable sets \(X\subseteq Y\times Z\), there exists an \(\LRVsig\)-definable map \(f : Z \to W\), where \(W\) is a product of sorts in \(\Geom\) and some \(N\in\Zz_{\geq 0}\) such that, for every prime \(p > N\) and \(z_1, z_2\in Z(K_p)\), \(f(z_1) = f(z_2)\) iff and only if \(X_{z_1}(K_p) = X_{z_1}(K_p)\).  \qed
\end{corollary}

As noted earlier, in case (1), since \(K_p\) is a definable expansion of a model of \(\ACVF\) which also eliminates imaginaries in the geometric sorts, the result is even uniform in all \(p\).



\sloppy
\printbibliography

\end{document}